\documentclass[11pt]{amsart}

\usepackage[a4paper,margin=1in]{geometry}
\usepackage{amsmath,amssymb,amsthm,mathtools}
\numberwithin{equation}{section}
\usepackage{bm}
\usepackage[numbers,sort&compress]{natbib}
\usepackage{microtype}
\usepackage{longtable,booktabs,array,calc}
\usepackage{xcolor}
\usepackage[colorlinks=true,linkcolor=blue!45!black,citecolor=blue!45!black,urlcolor=blue!55!black]{hyperref}
\usepackage{bookmark}

\hypersetup{
  pdftitle={The Polya--Szego conjecture for convex polygons with many sides},
  pdfauthor={Zhuo Cheng, Changfeng Gui, Yeyao Hu, Qinfeng Li},
  pdfsubject={Global Dirichlet spectral minimality and rigidity in the convex polygonal class for all sufficiently large side counts},
  pdfkeywords={Faber--Krahn inequality, Dirichlet eigenvalue, convex polygon, Fourier energy, Sobolev quadrature, spectral rigidity}
}

\allowdisplaybreaks[2]
\newcommand{\Tphys}{\mathbb T_{2\pi}}
\newcommand{\Tunit}{\mathbb T_{1}}

\newtheorem{theorem}{Theorem}[section]
\newtheorem{proposition}[theorem]{Proposition}
\newtheorem{lemma}[theorem]{Lemma}
\newtheorem{corollary}[theorem]{Corollary}
\theoremstyle{definition}
\newtheorem{definition}[theorem]{Definition}
\theoremstyle{remark}

\title[The polygonal Faber--Krahn inequality]
{The P\'olya--Szeg\H{o} conjecture for convex polygons with many sides}

\author{Zhuo Cheng}
\address{School of Mathematics and Statistics, HNP-LAMA, Central South University, Changsha, Hunan 410083, P. R. China}
\email{zhuo@outlook.cz}

\author{Changfeng Gui}
\address{Department of Mathematics, University of Macau, Macau SAR, P. R. China; Zhuhai UM Science and Technology Research Institute, Hengqin, Guangdong 519031, P. R. China}
\email{changfenggui@um.edu.mo}

\author{Yeyao Hu}
\address{School of Mathematics and Statistics, HNP-LAMA, Central South University, Changsha, Hunan 410083, P. R. China}
\email{huyeyao@gmail.com}

\author{Qinfeng Li}
\address{School of Mathematics, Hunan University, Changsha, P. R. China}
\email{liqinfeng1989@gmail.com}

\date{}

\begin{document}

\begin{abstract}
For every sufficiently large $N$, we prove that the regular $N$-gon uniquely minimizes the first Dirichlet eigenvalue among convex polygons of prescribed area with at most $N$ sides, up to rigid motions. Quantitative Faber--Krahn stability localizes minimizers near the disk but leaves their discrete geometry undetermined. We encode that geometry by a centered measure of exterior angles and prove an inverse-cubic Fourier inequality whose equality case is the uniform root configuration. Its defect gives a lower bound for the quadratic spectral excess over the regular polygon. Cyclic cancellation and a common material-coordinate comparison make the nonlinear difference small relative to the same defect. Exact constraint coordinates extend this comparison to arbitrarily small positive defects. Minimality then forces zero defect, which characterizes the regular polygon.
\end{abstract}

\maketitle

\section{Introduction}\label{sec:introduction}

The Faber--Krahn inequality is one of the basic principles of spectral geometry: among planar domains of prescribed area, the disk uniquely minimizes the first Dirichlet eigenvalue \citep{Henrot06}. A natural polygonal analogue was proposed by P\'olya and Szeg\H{o}: among polygons with a prescribed number of sides and fixed area, the regular polygon should minimize the first Dirichlet eigenvalue; see \citep{PS51}. The conjecture has remained open for more than seventy years and has become a model problem for spectral shape optimization under a finite-sided, hence combinatorial, geometric constraint. Its appeal comes partly from the contrast between its elementary formulation and the analytic difficulty of the problem: even though an $N$-gon is determined by finitely many parameters, the quantity being optimized is the ground-state energy of a PDE on a nonsmooth moving domain.

This finite-dimensional appearance is deceptive. The symmetrizations behind the classical Faber--Krahn inequality \citep{Polya48} do not preserve the number of sides, and there is no known geometric deformation which simultaneously regularizes an arbitrary polygon and controls the first eigenvalue in the required direction. Moreover, corners introduce singularities into the eigenfunction and into shape derivatives, while first-order optimality conditions by themselves do not distinguish the regular polygon from other possible critical configurations. In particular, a local Hessian analysis near the regular polygon does not by itself address the global minimization problem. These difficulties explain why, despite sustained interest, the conjecture was classically known only for triangles and quadrilaterals; see the discussion in \citep{BB24}. Recent work of Bogosel and Bucur developed a finite-dimensional shape-Hessian framework and reduced the conjecture for each fixed $N\ge5$ to finitely many certified numerical computations \citep{BB24}; a subsequent validated-computing argument establishes local minimality for $N=5,6$ \citep{BB24local}. Related large-$N$ and regular-polygon results include \citep{AF06,Nitsch14,Indrei24,BGMR24,DGP26}; see also \citep{GHL26} for a P\'olya--Szeg\H{o}-type extremal result for torsional rigidity in the tangential class. For sharp spectral gap bounds on convex domains, see \citep{AC11}.

The purpose of the present paper is to prove the conjecture analytically in the asymptotic regime of many sides. More precisely, we show that there exists $N_0$ such that the regular $N$-gon is the unique global minimizer among all bounded convex polygons with at most $N$ effective sides whenever $N\ge N_0$. Here an effective side means a maximal nondegenerate line segment in the boundary. For $A>0$, set
\[
 \mathcal P_{\le N}(A)=\left\{\operatorname{int}K:\begin{array}{l}
 K\subset\mathbb R^2\text{ is a bounded convex polygon},\ |K|=A,\\
 K\text{ has at most }N\text{ effective sides}
 \end{array}\right\},
\]
and let $R_N$ denote the regular $N$-gon of area $\pi$.

\begin{theorem}[Large-$N$ polygonal Faber--Krahn theorem]\label{thm:main}
There exists $N_0\in\mathbb N$ such that, for every $N\ge N_0$ and every $P\in\mathcal P_{\le N}(\pi)$,
\[
 \lambda_1(P)\ge\lambda_1(R_N).
\]
Equality holds if and only if $P$ is a rigid image of $R_N$. Equivalently, for every bounded convex polygon $P$ of positive area with at most $N$ effective sides,
\[
 |P|\lambda_1(P)\ge|R_N|\lambda_1(R_N),
\]
with equality precisely for polygons similar to $R_N$.
\end{theorem}

The proof has two features which we believe are as important as the large-$N$ theorem itself. The first is a global-to-local mechanism: global minimality forces an arbitrary convex minimizer into a small neighborhood of the disk, without assuming any a priori regularity of its side lengths, angles, or vertex spacing. The second, and more structural, feature is a perhaps unexpected Fourier bridge between the continuous spectral geometry of the disk and the discrete geometry of a polygon. This bridge identifies the correct microscopic defect of a near-circular polygon and ultimately singles out the regular $N$-gon.

We now describe this connection. Let $P$ be a convex polygon which is already known to be close to the unit disk and translate it to a distinguished polar center. If $z_j$ are its vertices, write
\[
 \psi_j=\arg z_j
\]
for their directions from the center and let $\omega_j$ be the corresponding exterior angles. Since $\sum_j\omega_j=2\pi$, the normalized turning angles
\[
 w_j=\frac{\omega_j}{2\pi}
\]
define a probability measure on the circle,
\begin{equation}\label{eq:intro-turning-measure}
 \nu_P=\sum_jw_j\delta_{\psi_j},\qquad \widehat\nu_P(1)=0.
\end{equation}
Thus the polygon is encoded by the places at which the normal direction jumps and by the amount of turning carried by each jump. For the regular $N$-gon this measure is, up to rotation,
\[
 \sigma_N=\frac1N\sum_{j=0}^{N-1}\delta_{2\pi j/N},
\]
the uniform measure on the $N$th roots of unity. Its Fourier transform has the rigid form
\begin{equation}\label{eq:intro-root-fourier}
 \widehat\sigma_N(k)=\mathbf1_{\{N\mid k\}}.
\end{equation}
At first sight, the atomic measure \eqref{eq:intro-turning-measure} belongs to discrete convex geometry, whereas the first eigenvalue is a PDE quantity. The main observation is that, near the disk, the two are governed by the same Fourier scale. Set
\[
 F(\Omega)=|\Omega|\lambda_1(\Omega),
\]
which is invariant under dilations. Because the disk is rotationally symmetric, the Hessian of $F/F(\mathbb D)$ at the disk is diagonal in angular Fourier modes. If $f$ is a radial boundary perturbation, its quadratic contribution has the form
\begin{equation}\label{eq:intro-disk-hessian}
 Q(f)=4\sum_{k\ge2}\beta_k|\widehat f(k)|^2,\qquad
 \beta_k=1+j\frac{J_k'(j)}{J_k(j)},
\end{equation}
where $j=j_{0,1}$ is the first positive zero of $J_0$. The relevant point is the asymptotic size
\begin{equation}\label{eq:intro-bessel-order}
 \beta_k\sim k.
\end{equation}
To connect \eqref{eq:intro-disk-hessian} with the polygonal data, write the boundary of $P$ in polar form
\[
 r(\theta)=e^{v(\theta)},
\]
and let $p(\theta)=\operatorname{Arg}(e^{i\theta}\overline{n_P(\theta)})\in(-\pi/2,\pi/2)$ be the angle between the radial direction and the outward normal. Polygonal geometry gives the exact identities
\begin{equation}\label{eq:intro-turning-equation}
 v'=\tan p,\qquad Dp=d\theta-2\pi\nu_P.
\end{equation}
The second identity describes the unit rotation of the radial angle along each edge and the jump of the normal by the exterior angle at each vertex. Linearizing the first relation near the circle and solving \eqref{eq:intro-turning-equation} in Fourier variables gives a linear radial potential $v_L$ satisfying
\begin{equation}\label{eq:intro-turning-potential}
 \widehat v_L(k)=\frac{\widehat\nu_P(k)}{k^2},\qquad k\ne0,\qquad \widehat v_L(0)=0.
\end{equation}

In other words, passing from the turning distribution of the polygon to its radial displacement involves two angular integrations. Substituting \eqref{eq:intro-turning-potential} into \eqref{eq:intro-disk-hessian} yields
\[
 Q(v_L)=4\sum_{k\ge2}\frac{\beta_k}{k^4}|\widehat\nu_P(k)|^2
 \sim\sum_{k\ge2}\frac{|\widehat\nu_P(k)|^2}{k^3}.
\]
This is the Fourier bridge underlying the proof: the disk Hessian transforms the polygonal turning measure into an inverse-cubic Fourier energy.

The same inverse-cubic energy has an independent discrete rigidity property. Since \eqref{eq:intro-root-fourier} gives
\[
 \sum_{k\ge1}\frac{|\widehat\sigma_N(k)|^2}{k^3}=\frac{\zeta(3)}{N^3},
\]
we introduce the defect
\begin{equation}\label{eq:intro-defect}
 E_N(\nu)=\sum_{k\ge1}\frac{|\widehat\nu(k)|^2}{k^3}-\frac{\zeta(3)}{N^3},\qquad E_N(P):=E_N(\nu_P).
\end{equation}
A sharp Fej\'er--Poisson argument shows that
\begin{equation}\label{eq:intro-atomic-minimum}
 E_N(\nu)\ge0
\end{equation}
for every positive probability measure supported on at most $N$ points, with equality if and only if $\nu$ is a rotation of $\sigma_N$. Thus \eqref{eq:intro-defect} has two simultaneous meanings: it is, up to the Bessel correction in \eqref{eq:intro-disk-hessian}, the quadratic spectral excess over the regular polygon measured by the disk Hessian, and it is a discrete rigidity functional which vanishes only on the regular $N$-point configuration. This coincidence is the first main mechanism of the paper.

If the spectral comparison were exhausted by its quadratic part, the proof would essentially end here. The difficulty is that the polygonal scale is singular in $N$: the leading nonlinear interaction is not automatically negligible compared with \eqref{eq:intro-defect}. More precisely, after the quadratic term has been isolated, the principal cubic contribution is represented by
\begin{equation}\label{eq:intro-cubic-functional}
 T_N(\nu)=\sum_{a,b\ge1}\frac{\widehat\nu(a)\widehat\nu(b)\overline{\widehat\nu(a+b)}}{ab(a+b)^2}.
\end{equation}
A direct estimate controls the difference $T_N(\nu)-T_N(\sigma_N)$ only by $O(E_N)$, which is too large to preserve the positive quadratic gap. The relevant issue is therefore not the absolute size of the cubic term, but its structure relative to the regular polygon.

This is where the cyclic symmetry of the regular $N$-gon enters. Once minimality and the atomic defect have forced the weights and angular gaps into a quasi-uniform regime, the Fourier frequencies can be grouped as
\[
 k=mN+r,
\]
where $r$ measures the deviation from the lattice $N\mathbb Z$ associated with the regular configuration. A phase interpolation of the perturbed lattice produces Fourier columns $c_m$ obeying the exact convolution law
\begin{equation}\label{eq:intro-phase-convolution}
 c_m*c_n=c_{m+n}.
\end{equation}
In the zero-winding phase model, when the denominator in \eqref{eq:intro-cubic-functional} is expanded in the residue variables, the potentially dangerous first-order corrections are precisely first moments of the correlations generated by \eqref{eq:intro-phase-convolution}; these moments vanish. This finite cyclic Ward identity, together with the comparison estimates described below, improves the cubic comparison from the natural $O(E_N)$ scale to
\begin{equation}\label{eq:intro-relative-cubic}
 \operatorname{Re}[T_N(\nu_P)-T_N(\sigma_N)]=o(E_N(P))
\end{equation}
on the relevant near-regular regime. In this way the two symmetries in the problem play complementary roles: rotational symmetry of the disk diagonalizes the quadratic response, while cyclic symmetry of the regular polygon suppresses the leading nonlinear interaction.

Let us finally indicate how these ideas enter the global proof. A minimizer exists by compactness. A corner-cutting variation shows that every minimizer uses all $N$ available sides. Since $R_N$ is admissible and its eigenvalue is close to that of the disk, the quantitative Faber--Krahn inequality \citep{BDPV15}, together with convexity, places every minimizer in a canonical near-disk polar chart. The fixed-disk expansion then gives the exact decomposition
\begin{equation}\label{eq:intro-exact-splitting}
 \frac{F(P)-F(R_N)}{F(\mathbb D)}
 =Q_N^\circ(P)+D_N^\circ(P)+\widetilde R_N^\circ(P).
\end{equation}
Here $Q_N^\circ$ is the Bessel-corrected quadratic term arising from the Fourier bridge, $D_N^\circ$ is the geometric correction between the true radial graph and the linear atomic potential, and $\widetilde R_N^\circ$ is the nonlinear fixed-disk remainder, each taken relative to $R_N$, with $X^\circ(P)=X(P)-X(R_N)$. The inverse-cubic rigidity, together with the exact Bessel comparison, gives a uniform coercive estimate
\[
 Q_N^\circ(P)\ge cE_N(P)
\]
with an absolute $c>0$. A first bootstrap turns global minimality into microscopic quasi-uniformity: starting from $E_N=O(N^{-4})$, local multiplicity and sampling bounds give $N^4T_N(\nu_P)=\zeta(4)/2+o(1)$ before any minimum-gap bound. The spectral expansions at a minimizer and at the regular polygon therefore have the same leading cubic term. Subtracting this common term in the minimizing inequality yields $E_N=o(N^{-4})$, and atomic rigidity controls the weights and gaps. The cyclic Ward cancellation \eqref{eq:intro-relative-cubic}, transferred from its zero-winding phase model to weighted columns through finite cyclic completion and return and weight-residual estimates, first applies on the centered near-regular defect annulus in Proposition~\ref{prop:principal-cubic-relative}. Geometric and material comparison estimates then control the full correction in the refined minimizer regime of Proposition~\ref{prop:minimizer-localization}, above the cutoff. Below it, the exact constraint chart permits comparison along rays. The estimate at a point above the cutoff controls the quadratic coefficients; polynomial bounds for the third derivatives then control the Taylor errors at smaller defects. Together these estimates yield, uniformly down to arbitrarily small positive defects,
\[
 |D_N^\circ(P)+\widetilde R_N^\circ(P)|\le o(1)E_N(P).
\]
Consequently,
\[
 \frac{F(P)-F(R_N)}{F(\mathbb D)}\ge(c-o(1))E_N(P).
\]
For a global minimizer the left-hand side is nonpositive, and hence $E_N(P)=0$ for all sufficiently large $N$. The equality case in \eqref{eq:intro-atomic-minimum} identifies the turning measure with the uniform root measure, and the exact turning equations \eqref{eq:intro-turning-equation} then reconstruct the regular polygon.

The role of the large-$N$ assumption is therefore precise. It is not used to replace the polygon by the disk or to deduce the theorem from the classical Faber--Krahn inequality. Rather, it first places every global minimizer in a regime where the disk provides a universal continuous model, and then allows the remaining polygonal degrees of freedom to be resolved at their natural microscopic scale. The proof may thus be viewed as a two-scale rigidity argument: macroscopic Faber--Krahn stability produces the disk, while microscopic Fourier rigidity crystallizes the turning measure into the regular $N$-point lattice.

This viewpoint also suggests a broader use of the method. The inverse-cubic defect is not introduced ad hoc; it is forced by the order of the disk Hessian together with the two integrations which recover boundary position from turning. The same principle---identify a continuous symmetric limit, encode the discrete boundary geometry by an atomic measure, and compare the resulting Fourier energy with the appropriate discrete lattice---may be useful in other polygonal shape-optimization problems in which the regular polygon is expected to be extremal. For related distribution, harmonic-analysis, and discrete-energy problems, see \citep{IN15,ET48,KV23,Nagel25}. The cyclic Ward mechanism further shows how a discrete symmetry can control nonlinear interactions at a scale invisible to a purely quadratic stability argument.

The paper is organized as follows. Section~\ref{sec:compactness} contains the global reduction: existence of minimizers, activation of all sides, stability of the polar center, and the passage to a radial graph over the disk. Section~\ref{sec:spectral-rigidity} combines the main estimates and proves Theorem~\ref{thm:main}. The fixed-disk analytic expansion is developed in Section~\ref{sec:disk-perturbation}. Section~\ref{sec:curvature-coordinates} contains the atomic inverse-power rigidity and the Bessel-corrected coercivity. Section~\ref{sec:polar-geometry} develops the polar representation. Sections~\ref{sec:cubic-coefficient}--\ref{sec:relative-cubic} develop the cubic coefficient and finite cyclic calculus, and prove the relative cubic Ward cancellation and the defect bootstrap. Sections~\ref{sec:polygonal-transport}--\ref{sec:exact-constraints} supply the geometric and material comparison and the exact near-regular chart needed to close the argument at arbitrarily small positive defect. Appendix~\ref{sec:mesh-tools} contains the periodic mesh estimates.

\section{Geometric reduction of the variational problem}\label{sec:compactness}

We begin by bringing a global minimizer into the range of the disk expansion. Sides may disappear under compactness, so the variational class must allow fewer than $N$ effective sides. A corner-cutting argument then shows that a minimizer uses the full side budget. Comparison with the regular polygon gives a small radial graph over the disk, and centering the exterior-angle measure fixes translation. None of these reductions requires control of individual edge lengths or turning angles.

We use the standard trace and extension theory for Lipschitz domains,
compact Dirichlet embeddings, and common-domain perturbation theory
\citep{AdamsFournier03,Daners03,BucurButtazzo05,Kato76}. The corner and
nonsmooth shape-calculus inputs are taken from
\citep{Grisvard85,Dauge88,NovruziPierre02,LamboleyNovruziPierre16,Laurain20}.
Estimates requiring uniformity in the polygonal mesh are derived below.

Write $\mathbb D=B=B_1(0)$, $\mathbb T=\Tphys$, and $j=j_{0,1}$.
For $L>0$, let $\mathbb T_L=\mathbb R/(L\mathbb Z)$, with normalized
measure $dm_L(t)=dt/L$.  We use

\begin{equation}\label{eq:compactness-001}
 \widehat f_L(k)=\int_{\mathbb T_L}f(t)e^{-2\pi ikt/L}\,dm_L(t),
 \qquad
 \widehat\nu_L(k)=\int_{\mathbb T_L}e^{-2\pi ikt/L}\,d\nu(t).
\end{equation}
Pullback by $x=\theta/(2\pi)$ identifies the Fourier coefficients on
$\Tphys$ and $\Tunit$.

\subsection{Existence and the number of effective sides}

\begin{proposition}\label{prop:existence}

For every \(N\ge3\), the variational problem

\[
\min\{\lambda _1(P):P\in\mathcal P_{\le N}(\pi)\}
\]

has a minimizer.

\end{proposition}

\begin{proof}
Let $(P_m)$ be a minimizing sequence of area $\pi$, translated to have centroid zero. A uniform eigenvalue bound controls both geometric obstructions to compactness. The one-dimensional estimate
\begin{equation}\label{eq:compactness-002}
 \lambda_1(K)\ge \frac{\pi^2}{w(K)^2},
\end{equation}
prevents the minimal width from tending to zero. A diameter of length $D$ and the two extreme transverse support lines determine two triangles contained in the polygon, so that
\begin{equation}\label{eq:compactness-003}
 |K|\ge \frac12D\,w(K).
\end{equation}
The area constraint now bounds the diameter. Blaschke compactness produces a convex Hausdorff limit of area $\pi$.

To retain the combinatorial constraint, represent each $P_m$ as the convex hull of $N$ points, allowing repetitions. A further subsequence makes all $N$ points converge, and their convex hull is the same limit. It has at most $N$ effective sides. The Dirichlet eigenvalues converge under this convex-domain convergence \citep{Daners03}; the limit is therefore a minimizer.
\end{proof}

\begin{lemma}\label{lem:convex-corner-eigenfunction-decay}

Let \(0<\omega<\pi\), \(\alpha=\pi/\omega>1\), and

\[
W_{\omega,R}=\{(r,\theta):0<r<R,\ 0<\theta<\omega\}.
\]

Suppose \(u\in H^1(W_{\omega,R})\) is nonnegative, solves
\(-\Delta u=\lambda u\) weakly with \(\lambda\ge0\), and has zero trace
on the two radial sides. For every \(1<\beta<\alpha\) there are
\(r_1\in(0,R)\) and \(C<\infty\) such that

\[
0\le u(r,\theta)\le Cr^\beta,
\]

for almost every \((r,\theta)\in W_{\omega,r_1}\).

\end{lemma}

\begin{proof}
The positive indicial exponents of the angular Dirichlet problem are
\[
 \frac{k\pi}{\omega},\qquad k=1,2,\ldots.
\]
An $H^1$ solution has no negative-power component. The corner expansion \citep{Grisvard85,Dauge88} consequently implies, on a smaller sector,
\[
 |u(r,\theta)|\le C_\beta r^\beta
 \qquad\text{for every }\beta<\frac{\pi}{\omega}.
\]
Convexity gives $\pi/\omega>1$. Choose $\beta$ with $1<\beta<\pi/\omega$.
\end{proof}

\begin{proposition}\label{prop:full-activation}

For every \(N\ge3\), every global minimizer of

\[
\min\{\lambda _1(P):P\in\mathcal P_{\le N}(\pi)\}
\]

has exactly \(N\) effective sides.

\end{proposition}

\begin{proof}
Assume that an admissible minimizer $P$ has fewer than $N$ effective sides. Choose one of its vertices as the origin, and let $\omega<\pi$ be the interior angle. The positive $L^2$-normalized ground state has the corner bound
\begin{equation}\label{eq:compactness-004}
 u(x)\le C|x|^\beta
\end{equation}
for some $\beta>1$. Applying Caccioppoli's inequality on a slightly larger corner neighborhood gives
\begin{equation}\label{eq:compactness-005}
 \int_{P\cap B_{3\varepsilon}}|\nabla u|^2\le C\varepsilon^{2\beta},
 \qquad
 \int_{P\cap B_{3\varepsilon}}u^2\le C\varepsilon^{2\beta+2}.
\end{equation}

Join the two points at distance $\varepsilon$ from the vertex along its incident edges and remove the corner triangle. This produces a convex polygon $P_\varepsilon$ with one additional effective side and
\begin{equation}\label{eq:compactness-006}
 |P|-|P_\varepsilon|=c_\omega\varepsilon^2,
 \qquad c_\omega=\tfrac12\sin\omega>0.
\end{equation}
Use a cutoff of $u$ which vanishes near that triangle and agrees with $u$ outside $B_{3\varepsilon}$. The preceding estimates control the cutoff terms and give
\[
 \int_{P_\varepsilon}|\nabla v_\varepsilon|^2
 \le\lambda_1(P)+C\varepsilon^{2\beta},
 \qquad
 \int_{P_\varepsilon}v_\varepsilon^2
 \ge1-C\varepsilon^{2\beta+2}.
\]
Hence
\begin{equation}\label{eq:compactness-007}
 \lambda_1(P_\varepsilon)
 \le\lambda_1(P)+C\varepsilon^{2\beta}.
\end{equation}

The cost in eigenvalue is of order $\varepsilon^{2\beta}$, whereas the area removed has order $\varepsilon^2$. Restore the area by the dilation $s_\varepsilon=(\pi/|P_\varepsilon|)^{1/2}$. Then
\begin{equation}\label{eq:compactness-008}
\begin{aligned}
 \lambda_1(s_\varepsilon P_\varepsilon)
 &\le \left(1-\frac{c_\omega}{\pi}\varepsilon^2\right)
       \left(\lambda_1(P)+C\varepsilon^{2\beta}\right)\\
 &=\lambda_1(P)-\frac{c_\omega\lambda_1(P)}{\pi}\varepsilon^2
   +o(\varepsilon^2).
\end{aligned}
\end{equation}
Since $2\beta>2$, this competitor has a smaller eigenvalue for small $\varepsilon$. It still has at most $N$ sides, a contradiction.
\end{proof}

\subsection{Near-disk geometry and the polar center}

\begin{lemma}\label{lem:convex-cap}

There are absolute constants \(\delta _0,C>0\) such that, whenever \(K\)
is a planar convex body satisfying

\[
|K|=\pi,
\qquad
|K\triangle\mathbb D|\le\delta\le\delta _0,
\]

one has

\[
\sup_{e\in\mathbb S^1}|h_K(e)-1|
\le C\delta^{2/3}.
\]

Consequently,

\[
B_{1-C\delta^{2/3}}
\subset K
\subset B_{1+C\delta^{2/3}},
\qquad
d_H(K,\mathbb D)\le C\delta^{2/3}.
\]

\end{lemma}

\begin{proof}
The exponent $2/3$ comes from the area of a small circular cap. If a support line moves inward by $t$, with $0<t\le1/2$, the missing cap has area at least
\begin{equation}\label{eq:compactness-009}
 2\int_0^t\sqrt{2u-u^2}\,du\ge\frac43t^{3/2}.
\end{equation}
The symmetric-difference bound thus gives
\begin{equation}\label{eq:compactness-010}
 B_{1-\varepsilon}\subset K,
 \qquad \varepsilon=C\delta^{2/3}.
\end{equation}

For the opposite inclusion, take $y\in K$ with $y\cdot e=1+t$. Only $t>8\varepsilon$ needs consideration. In $B_{1-\varepsilon}$ use the chord
\[
 I=\{x_0e+ze^\perp:|z|\le a\},
\]
with $x_0=1-t/4$ and $a=((1-\varepsilon)^2-x_0^2)^{1/2}$. When $t$ is small and $t>8\varepsilon$, its half-length satisfies $a\ge c\sqrt t$. The triangle with this base and vertex $y$ lies in $K$ by convexity; its area outside the unit disk is at least $cat$. It follows that
\begin{equation}\label{eq:compactness-011}
 \delta\ge |K\setminus\mathbb D|\ge c\,a t\ge ct^{3/2}.
\end{equation}
A protrusion with $t>1/2$ would give a fixed positive exterior area, which is excluded by choosing $\delta_0$ small. Thus both support displacements are $O(\delta^{2/3})$.
\end{proof}

\begin{definition}\label{def:polar-center-functional}

Let \(P\subset\mathbb R^2\) be a nondegenerate convex polygon with
vertices \(z_1,\ldots,z_m\) in cyclic order and positive exterior angles
\(\omega_1,\ldots,\omega_m\), so that \(\sum_j\omega_j=2\pi\). For
\(a\notin\{z_1,\ldots,z_m\}\) define

\[
F_P(a)=\sum_{j=1}^m\omega_j|z_j-a|,
\qquad
\Phi_P(a)=-\nabla F_P(a)
=\sum_{j=1}^m\omega_j\frac{z_j-a}{|z_j-a|}.
\]

Whenever \(F_P\) has a unique minimizer, that minimizer is denoted by
\(c_{\mathrm{pol}}(P)\) and is called the polar center of \(P\).

\end{definition}

\begin{lemma}\label{lem:polar-center-stability}

There are \(\varepsilon _0,C>0\) such that, for every nondegenerate
convex polygon \(P\subset\mathbb R^2\) with vertices and exterior angles
as in
Definition~\ref{def:polar-center-functional},
if

\[
B_{1-\varepsilon}\subset P\subset B_{1+\varepsilon},
\qquad
0<\varepsilon\le\varepsilon _0,
\]

then

\[
|c_{\mathrm{pol}}(P)|\le C\varepsilon.
\]

After translation to the polar center,

\[
B_{1-C\varepsilon}
\subset P-c_{\mathrm{pol}}(P)
\subset B_{1+C\varepsilon}.
\]

\end{lemma}

\begin{proof}
The center equation is controlled by the first two angular moments of the boundary. Write $\partial P$ as $r(\theta)e^{i\theta}$ and set $v=\log r$. The annular bounds give $\|v\|_\infty\le C\varepsilon$, while the supporting-line relation gives
\begin{equation}\label{eq:compactness-012}
 v'=\tan p,
 \qquad
 \|p\|_\infty\le C\varepsilon^{1/2}.
\end{equation}
The turning function has slope one along an edge and a jump $-\omega_j$ at $\psi_j$. Testing its distributional derivative against $e^{ik\theta}$ gives
\begin{equation}\label{eq:compactness-013}
 \sum_j\omega_je^{ik\psi_j}
 =ik\int_0^{2\pi}p(\theta)e^{ik\theta}\,d\theta,
\end{equation}
for $k\ne0$. For $k=1,2$, the relation $p=v'+O(p^3)$ and integration by parts place both moments at order $\varepsilon$. In particular,
\begin{equation}\label{eq:compactness-014}
 |\Phi_P(0)|\le C\varepsilon,
 \qquad
 \left|\sum_j\omega_je^{2i\psi_j}\right|\le C\varepsilon.
\end{equation}

The first of these estimates bounds the residual at the origin. The second bounds the anisotropic part of its derivative; the bound $|z_j|=1+O(\varepsilon)$ controls the denominators. Thus
\begin{equation}\label{eq:compactness-015}
 D\Phi_P(0)
 =-\sum_j\frac{\omega_j}{|z_j|}(I-e_j\otimes e_j)
 =-\pi I+O(\varepsilon),
 \qquad e_j=z_j/|z_j|.
\end{equation}
On a fixed smaller disk, the distance to every vertex is bounded below. Together with $\sum_j\omega_j=2\pi$, this gives a uniform bound for $D^2\Phi_P$. The quantitative inverse-function theorem provides a zero within $C\varepsilon$ of the origin. Strict convexity of $F_P$ makes this zero the unique polar center. The translation required to reach it changes the annular radii by at most $C\varepsilon$.
\end{proof}

\begin{proposition}\label{prop:global-minimizer-radial-entry}

There are constants \(C<\infty\) and \(N_{\rm ent}\) such that the
following holds. Let \(N\ge N_{\rm ent}\) and let \(P_N^*\) be any
global minimizer in \(\mathcal P_{\le N}(\pi)\). After translating
either to its area centroid or to its polar center, the polygon
satisfies

\[
B_{1-C/N}\subset P_N^*\subset B_{1+C/N}.
\]

In either normalization its boundary has a unique radial representation

\[
\partial P_N^*
 =\{(1+\rho_N(\theta))e^{i\theta}:\theta\in\mathbb T\},
\]

and, with \(v_N=\log(1+\rho_N)\) and \(v_N'=\tan p_N\),

\[
\|\rho_N\|_{L^\infty}\le CN^{-1},
\]

\[
\|p_N\|_{L^\infty}
 +\|v_N'\|_{L^\infty}
 +\|\rho_N'\|_{L^\infty}
 \le CN^{-1/2}.
\]

In particular \(\|\rho_N\|_{W^{1,\infty}}\to0\). In the polar-center
normalization, if \(z_j\) are the vertices, \(\omega_j\) their exterior
angles,

\[
\psi_j=\arg z_j,
\qquad
w_j=\frac{\omega_j}{2\pi},
\qquad
\nu_{P_N^*}=\sum_jw_j\delta_{\psi_j},
\]

then the centering constraint is exact:

\[
\widehat\nu_{P_N^*}(1)=0.
\]

\end{proposition}

\begin{proof}
The regular comparison polygon satisfies
\[
 0\le\lambda_1(R_N)-\lambda_1(\mathbb D)\le CN^{-3}.
\]
Since a minimizer has no larger eigenvalue, quantitative Faber--Krahn stability \citep{BDPV15} supplies a translation $x_N$ such that
\[
 |P_N^*\triangle(\mathbb D+x_N)|\le CN^{-3/2}.
\]
Applying the convex-cap estimate gives
\begin{equation}\label{eq:compactness-016}
 B_{1-C/N}+x_N\subset P_N^*\subset B_{1+C/N}+x_N.
\end{equation}
Polar-center stability moves this reference point by at most $C/N$, preserving these bounds with a different constant.

For an edge with support number $d_e$ and normal angle $\varphi_e$, the radial equation is
\[
 1+\rho_N(\theta)=d_e\sec(\theta-\varphi_e).
\]
Both $d_e$ and the radial distance belong to $[1-C/N,1+C/N]$. Therefore
\[
 \cos p_N\ge\frac{1-C/N}{1+C/N},
 \qquad p_N=\theta-\varphi_e.
\]
This gives $\|p_N\|_\infty\le CN^{-1/2}$, and differentiation of the edge equation gives the estimates for $v_N'$ and $\rho_N'$. The area is unaffected by translation. At the polar center,
\[
 0=\sum_j\omega_j\frac{z_j}{|z_j|}
  =2\pi\sum_jw_je^{i\psi_j},
\]
so the first Fourier moment vanishes exactly.
\end{proof}

\section{Uniform eigenvalue estimates on a fixed disk}\label{sec:disk-perturbation}

The geometric reduction supplies a small Lipschitz radial graph, with no control of the smallest edge. We formulate the eigenvalue problem on the fixed disk so that the estimates depend on the Lipschitz norm and boundary energy of the deformation. In every higher-order coefficient, two energy factors retain its quadratic size and the remaining Lipschitz factors provide smallness. This distinction will let the later comparison track the defect of an individual polygon.

\begin{definition}
\label{def:fixed-disk}
Let $B=B_1(0)$, $j=j_{0,1}$ be the first positive zero of $J_0$, and
$u_0(r)=c_0J_0(jr)$ the positive $L^2(B)$-normalized ground state.  Thus
$\lambda_1(B)=j^2$; we also use the elementary bound $j^2<6$.  We adopt
\[
 \widehat f_m=\frac1{2\pi}\int_0^{2\pi}f(\theta)e^{-im\theta}\,d\theta,
 \qquad f^\perp=\sum_{|m|\ge2}\widehat f_me^{im\theta},
\]
\[
 \|f\|_{\dot H^{1/2}}^2=\sum_{m\ne0}|m||\widehat f_m|^2,
 \qquad
 \|f\|_{H^{1/2}}^2=\sum_m(1+|m|)|\widehat f_m|^2,
\]
and set
\begin{equation}\label{eq:disk-perturbation-001}
 \beta_m=1+j\frac{J_m'(j)}{J_m(j)},\qquad
 Q_B(f)=2j^2\sum_{|m|\ge2}\beta_{|m|}|\widehat f_m|^2.
\end{equation}
For real $\rho\in W^{1,\infty}(\mathbb S^1)$ with $\|\rho\|_\infty<1$,
write
\[
 \Omega_\rho=\{re^{i\theta}:0\le r<1+\rho(\theta)\},\qquad
 F_{\rm n}(\Omega)=\frac{|\Omega|}{\pi}\lambda_1(\Omega).
\]
Let
\[
 \mathcal X=\left\{U\in W^{1,\infty}(B;\mathbb R^2)\cap H^1(B;\mathbb R^2):
 \int_BU=0\right\},\quad
 L(U)=\|DU\|_\infty,\quad E(U)=\|DU\|_2,
\]
and, for $L(U)<1/2$,
\begin{equation}\label{eq:disk-perturbation-002}
 J_{{\rm sp}}(U)=F_{\rm n}((I+U)(B)).
\end{equation}
With $F_U=I+DU$, $J_U=\det F_U$,
$A_U=J_UF_U^{-1}F_U^{-T}$, and
$s_U=\pi^{-1}\int_BJ_U$, the pullback of $J_{{\rm sp}}(U)$ is the lowest
generalized eigenvalue of
\begin{equation}\label{eq:disk-perturbation-003}
 a_U(v,w)=s_U\int_BA_U\nabla v\cdot\nabla w,
 \qquad m_U(v,w)=\int_BJ_Uvw,
\end{equation}
on the common form domain $H_0^1(B)$.  For $L(U)$ in a fixed compact
subset of $[0,1)$ these forms are uniformly bounded and coercive.
\end{definition}

\subsection{Boundary extension and domain realization}

The spectral problem can be pulled back to the disk only after the boundary displacement has been extended into a genuine domain deformation. The extension below controls its Lipschitz norm and its energy simultaneously, so neither control is lost when the eigenvalue equation is moved to the fixed domain.

\begin{lemma}
\label{lem:fixed-disk-extension-realization}
There is a bounded linear operator
\[
 \operatorname{Ext}:W^{1,\infty}(\mathbb S^1;\mathbb R^2)\cap
 H^{1/2}(\mathbb S^1;\mathbb R^2)
 \longrightarrow W^{1,\infty}(B;\mathbb R^2)\cap H^1(B;\mathbb R^2),
\]
such that
\[
 \operatorname{Tr}(\operatorname{Ext}b)=b,\qquad \int_B\operatorname{Ext}b=0,
\]
\begin{equation}\label{eq:disk-perturbation-012}
 \|D\operatorname{Ext}b\|_\infty\le C\|b\|_{W^{1,\infty}},\qquad
 \|D\operatorname{Ext}b\|_2^2\le C\|b\|_{H^{1/2}}^2.
\end{equation}
Moreover, if $V\in\mathcal X$, $\|DV\|_\infty<1$, and
$\operatorname{Tr}V=\rho e_r$ with $\|\rho\|_\infty<1$, then $I+V$ is
bi-Lipschitz and
\begin{equation}\label{eq:disk-perturbation-013}
 (I+V)(B)=\Omega_\rho.
\end{equation}
\end{lemma}

\begin{proof}
Use a smooth even circular mollifier and put $b_\delta=\varphi_\delta*b$. In $3/4<r<1$ take $\widetilde{\operatorname{Ext}}b(r,\theta)=b_{1-r}(\theta)$, and extend harmonically inside the circle $r=3/4$. The radial choice of smoothing scale gives, by convolution estimates and Parseval,
\[
 \|D\widetilde{\operatorname{Ext}}b\|_\infty\le C\|b\|_{W^{1,\infty}},\qquad
 \|D\widetilde{\operatorname{Ext}}b\|_2^2\le C\sum_m(1+|m|)|\widehat b_m|^2.
\]
A fixed function $\chi\in C_c^\infty(B)$ of integral one allows the correction $(\int_B\widetilde{\operatorname{Ext}}b)\chi$. Subtract it to obtain zero mean; neither the trace nor the two bounds changes.

To identify the image domain, observe that
\[
 (1-\|DV\|_\infty)|x-y|\le |(I+V)(x)-(I+V)(y)|
 \le(1+\|DV\|_\infty)|x-y|.
\]
Thus $\|DV\|_\infty<1$ makes $I+V$ bi-Lipschitz. Its boundary image is the simple curve $(1+\rho(\theta))e^{i\theta}$, so its image is precisely the bounded region $\Omega_\rho$ enclosed by that curve.
\end{proof}

\subsection{Common-domain eigenpairs and mixed coefficients}

With the domain fixed, the geometric perturbation enters through the stiffness and mass forms. The reduced resolvent separates the simple ground state from its orthogonal complement. This permits coefficient estimates with constants determined by the disk spectral gap, rather than by a polygonal mesh.

\begin{lemma}
\label{lem:uniform-eigenpair-perturbation}
\label{lem:saddle-inverse}
\label{lem:mixed-eigenpair-coefficients}
\label{lem:two-energy-coefficients}
Let $V\hookrightarrow H$ be a compact dense embedding, $T$ a compact
parameter set, and $(a_{t,Z},m_{t,Z})$ a holomorphic family of form pairs on
the common domain $V$, continuous in $t\in T$.  Assume uniform coercivity and
positivity on the real slice and, at $Z=0$, a simple normalized eigenpair
$(u_0,\lambda_0)$ separated from the rest of the spectrum by a uniform gap.
Then, after shrinking a common complex ball, the normalized eigenpair is
holomorphic in $Z$, uniform in $t$, and its saddle operator
\begin{equation}\label{eq:disk-perturbation-004}
 A_{{\rm eig},t,Z}(w,\mu)=
 \left((a_{t,Z}-\lambda_{t,Z}m_{t,Z})(w,\cdot)
       -\mu m_{t,Z}(u_{t,Z},\cdot),\ \ell_t(w)\right),
\end{equation}
is uniformly invertible.  If the data are $C^2$ in $t$, so are the branch
and inverse, with
\begin{equation}\label{eq:disk-perturbation-005}
 (A_{{\rm eig}}^{-1})'=-A_{{\rm eig}}^{-1}A_{{\rm eig}}'A_{{\rm eig}}^{-1},
 \quad
 (A_{{\rm eig}}^{-1})''=
 2A_{{\rm eig}}^{-1}A_{{\rm eig}}'A_{{\rm eig}}^{-1}A_{{\rm eig}}'A_{{\rm eig}}^{-1}
 -A_{{\rm eig}}^{-1}A_{{\rm eig}}''A_{{\rm eig}}^{-1}.
\end{equation}

For the quantitative mixed-coefficient form, let $Z_i$ be finitely many
amplitude directions and put
\[
 [H]_I=\frac1{|I|!}\left.\partial_{\varepsilon_I}
 H\!\left(\sum_{i\in I}\varepsilon_iZ_i\right)\right|_{\varepsilon=0}.
\]
Then the normalized Leibniz rule is
\begin{equation}\label{eq:disk-perturbation-006}
 [H_1\cdots H_p]_I=
 \sum_{I_1\sqcup\cdots\sqcup I_p=I}
 \frac{\prod_q|I_q|!}{|I|!}\prod_q[H_q]_{I_q}.
\end{equation}
Given $0\le E_i\le\kappa L_i$, set
\begin{equation}\label{eq:disk-perturbation-007}
 L_I=\prod_{i\in I}L_i,\qquad
 U_I=\sum_{i\in I}E_iL_{I\setminus\{i\}},\qquad
 V_I=\sum_{i\ne j}E_iE_jL_{I\setminus\{i,j\}}.
\end{equation}
Assume that, for $a=0,1,2$, every nonempty mixed coefficient of
$B=a-\lambda_0m$ and $M=m$ is bounded by $C_0^{|I|}\delta^aL_I$, its
action on $u_0$ by $C_0^{|I|}\delta^aU_I$, and its scalar value at
$(u_0,u_0)$ by $C_0^{|I|}\delta^aV_I$ for $|I|\ge2$; assume also
$B_{\{i\}}(u_0,u_0)=0$ and
$\|\partial_t^aA_{{\rm eig},t,0}^{-1}\|+\|\partial_t^a\ell_t\|
\le C\delta^a$.  If $u_{I,t}$ and $\lambda_{I,t}$ are the normalized mixed
coefficients of the eigenpair, then, after enlarging $C_0$ by an absolute
factor,
\begin{equation}\label{eq:disk-perturbation-008}
 \|\partial_t^au_{I,t}\|_V\le C_0^{|I|}\delta^aU_I,
 \qquad
 |\partial_t^a\lambda_{I,t}|\le C_0^{|I|}\delta^aV_I,
\end{equation}
with $\lambda_{\{i\},t}=0$.

For a $C^2$ path $U_t$ and $r\ge4$, let
$L_k=L(\partial_t^kU_t)$ and $E_k=E(\partial_t^kU_t)$, $k=0,1,2$, and
\[
\begin{aligned}
 \Theta_r(Z)&=E_0^2L(Z)L_0^{r-3}+2E(Z)E_0L_0^{r-2},\\
 \Theta_r(Z,W)&=E_0^2L(Z)L(W)L_0^{r-4}
 +2E(Z)E_0L(W)L_0^{r-3}\\
 &\quad+2E(W)E_0L(Z)L_0^{r-3}
 +2E(Z)E(W)L_0^{r-2}.
\end{aligned}
\]
If $J_{{\rm sp},r,t}$ is the normalized $r$-linear eigenvalue coefficient
and $f_r(t)=J_{{\rm sp},r,t}(U_t,\ldots,U_t)$, then
\begin{equation}\label{eq:disk-perturbation-009}
\begin{aligned}
 |f_r|&\le C_0^rr^2E_0^2L_0^{r-2},\\
 |f_r'|&\le C_0^rr^2\{\delta E_0^2L_0^{r-2}+r\Theta_r(\dot U_t)\},\\
 |f_r''|&\le C_0^rr^2\{\delta^2E_0^2L_0^{r-2}
 +2r\delta\Theta_r(\dot U_t)+r\Theta_r(\ddot U_t)
 +r(r-1)\Theta_r(\dot U_t,\dot U_t)\}.
\end{aligned}
\end{equation}
\end{lemma}

\begin{proof}
Separate the state space into the ground-state direction and its $m_{t,0}$-orthogonal complement. The spectral gap gives uniform coercivity of $a_{t,0}-\lambda_0m_{t,0}$ on the complement, and the normalization fixes the ground-state component. This proves uniform invertibility of the saddle operator. Analytic implicit-function theory on the common form domain \citep{Kato76} then gives the normalized holomorphic branch. A common smaller complex ball has uniform Cauchy bounds. Differentiating $A_{{\rm eig}}A_{{\rm eig}}^{-1}=I$ once and twice gives the inverse-derivative formulas.

The quantitative estimate follows from the coefficient equation, rather than an operator-norm Cauchy estimate that would lose the energy factors. Expand $(a-\lambda_0m)u=(\lambda-\lambda_0)mu$ in the amplitude variables. For a nonempty set $I$,
\begin{equation}\label{eq:disk-perturbation-010}
 A_{{\rm eig},t}(u_{I,t},\lambda_{I,t})=(F_{I,t},0),
\end{equation}
with
\begin{equation}\label{eq:disk-perturbation-011}
\begin{aligned}
F_{I,t}={}&-\!\sum_{A\sqcup C=I,\ A\ne\varnothing}
 \frac{|A|!|C|!}{|I|!}B_{A,t}u_{C,t}\\
&+\!\sum_{\substack{P\sqcup B\sqcup C=I,\ P\ne\varnothing,\\
(P,B,C)\ne(I,\varnothing,\varnothing)}}
 \frac{|P|!|B|!|C|!}{|I|!}\lambda_{P,t}M_{B,t}u_{C,t}.
\end{aligned}
\end{equation}
The factorial weights are those of the normalized Leibniz formula. In the induction on $|I|$, retain one $E_i$ in the state equation and two distinct $E_i$ in the scalar equation. Bound the other factors by $E_i\le\kappa L_i$. The factorial ratios cancel the ordered-partition multiplicities; the residual polynomial factors are absorbed into a larger $C_0$. This gives \eqref{eq:disk-perturbation-008}. The same recurrence differentiated in $t$ retains $\delta^a$ for $a=1,2$.

For $U_t$, differentiate the symmetric multilinear coefficient along the path. The two energy factors can lie in undifferentiated arguments, in one differentiated argument, or in two differentiated arguments. These placements produce the expressions $\Theta_r$ and the bounds \eqref{eq:disk-perturbation-009}.
\end{proof}

\subsection{The disk Hessian and its uniform remainder}

The abstract second-order coefficient must now be identified in boundary variables. Removing the similarity modes exposes the coercive disk Hessian. The remainder estimate is kept in a form that records both the small Lipschitz norm and the quadratic boundary energy.

\begin{theorem}
\label{thm:uniform-fixed-disk-expansion}
There are absolute $\varepsilon_0,C>0$ such that the following hold.

\emph{(i) Stationarity and Hessian.}  The map $U\mapsto J_{{\rm sp}}(U)$ is
analytic for $L(U)<\varepsilon_0$, $J_{{\rm sp}}(0)=j^2$, and
$DJ_{{\rm sp}}(0)=0$.  Its Hessian is symmetric and
\begin{equation}\label{eq:disk-perturbation-014}
 |D^2J_{{\rm sp}}(0)[H,K]|\le CE(H)E(K).
\end{equation}
If $\operatorname{Tr}V=\rho e_r$, then
\begin{equation}\label{eq:disk-perturbation-015}
 \ \frac12D^2J_{{\rm sp}}(0)[V,V]=Q_B(\rho^\perp).\
\end{equation}
In particular, for a smooth zero-mean radial first velocity $f$, the
area-constrained second derivative of $\lambda_1$ is $2Q_B(f^\perp)$.

\emph{(ii) Explicit coercivity.}  For every integer $m\ge2$,
\begin{equation}\label{eq:disk-perturbation-016}
 \frac m2\le\beta_m\le\frac{3m}{2},
\end{equation}
and hence, for every $f\in H^{1/2}(\mathbb S^1)$,
\begin{equation}\label{eq:disk-perturbation-017}
 \ j^2\|f^\perp\|_{\dot H^{1/2}}^2
 \le Q_B(f^\perp)\le3j^2\|f^\perp\|_{\dot H^{1/2}}^2 \
\end{equation}

\emph{(iii) Support-uniform remainder.}  If $L(U)\le\varepsilon_0$, then
\begin{equation}\label{eq:disk-perturbation-018}
 \sup_{0\le t\le1}\left|\frac{d^3}{dt^3}J_{{\rm sp}}(tU)\right|
 \le CL(U)E(U)^2,
\end{equation}
and therefore
\begin{equation}\label{eq:disk-perturbation-019}
 \left|J_{{\rm sp}}(U)-j^2-
 \frac12D^2J_{{\rm sp}}(0)[U,U]\right|
 \le CL(U)E(U)^2 \
\end{equation}
All constants are independent of $\operatorname{supp}DU$.
\end{theorem}

\begin{proof}
Analyticity is supplied by the common-domain eigenpair construction. The disk minimizes the scale-invariant functional along both signs of a small deformation, hence $DJ_{{\rm sp}}(0)=0$. For a smooth mean-zero normal velocity, the constrained second variation is
\[
 \lambda_1''(0)=4j^2\sum_{m\ne0}
 \left(1+j\frac{J_{|m|}'(j)}{J_{|m|}(j)}\right)|\widehat f_m|^2
\]
by \cite{DambrineLamboley19}, Section~2.3. The first harmonics vanish, and area normalization removes the constant mode without changing the functional. The critical-domain structure theorem \citep{NovruziPierre02} therefore identifies the Hessian with $2Q_B$.

To bound its coefficients, use
\[
 \beta_m=m+1-j\frac{J_{m+1}(j)}{J_m(j)}
\]
and, with $a=j^2/4<3/2$,
\[
 J_m(j)=\frac{(j/2)^m}{m!}
 \sum_{k\ge0}\frac{(-a)^k}{k!(m+1)_k}.
\]
For $m\ge2$, the ratio of successive absolute terms is less than $1/2$. Thus the alternating sum $S_m$ lies in $[1-a/(m+1),1]$, giving
\[
 0<j\frac{J_{m+1}(j)}{J_m(j)}
 \le\frac{j^2}{2(m+1-a)}\le\frac{m+2}{2}\qquad(m\ge2).
\]
The last inequality is equivalent to $4a\le(m+2)(m+1-a)$; it holds for $m\ge2$ at $a=3/2$ and hence throughout the stated range. Summing gives \eqref{eq:disk-perturbation-017}.

The nonlinear estimate must remain valid when the deformation is supported on a small set. Let $(u_t,\lambda_t)$ be the normalized branch along $tU$, with $L=L(U)$ and $E=E(U)$. A pulled-form derivative of order $r\le3$ contains $r$ deformation factors. For arbitrary states estimate them in $L^\infty$; with one ground-state argument keep one in $L^2$, and with two ground-state arguments keep two. The quadratic area factor obeys the same bounds. This yields
\begin{equation}\label{eq:disk-perturbation-020}
\begin{aligned}
 |q_t^{(r)}(v,w)|&\le CL^r\|v\|_{H^1}\|w\|_{H^1},\\
 |q_t^{(r)}(u_0,v)|&\le CL^{r-1}E\|v\|_{H^1},\\
 |q_t^{(r)}(u_0,u_0)|&\le CL^{r-2}E^2\qquad(r=2,3),\\
 |(q_t^{(1)}-q_0^{(1)})(u_0,u_0)|&\le CtE^2.
\end{aligned}
\end{equation}
Writing $u_t=u_0+(u_t-u_0)$ in the first two differentiated equations and applying the saddle inverse gives
\begin{equation}\label{eq:disk-perturbation-021}
 \|u_t-u_0\|_{H^1}+|\lambda_t-j^2|\le CtE,
 \quad \|\dot u_t\|_{H^1}\le CE,
 \quad |\dot\lambda_t|+|\ddot\lambda_t|\le CE^2,
 \quad \|\ddot u_t\|_{H^1}\le CLE.
\end{equation}
In the third derivative of the Rayleigh identity, the terms contain either a third form derivative, a second derivative with $\dot u_t$, or a first derivative with $\ddot u_t$ or two copies of $\dot u_t$. The displayed bounds retain two $L^2$ factors in every case, so their sum is bounded by $CLE^2$. Polarization gives the mixed Hessian estimate; Taylor's integral formula gives the remainder.

For a Lipschitz boundary trace $\rho e_r$, the realization lemma gives $J_{{\rm sp}}(tV)=F_{\rm n}(\Omega_{t\rho})$ for both signs of small $t$. The Hessian consequently depends only on that trace. Approximation in $H^{1/2}$, using the mixed bound and \eqref{eq:disk-perturbation-012}, extends the identity to Lipschitz traces. No constant in this argument depends on the size of the support.
\end{proof}

\subsection{Analytic coefficients with two energy factors}

For the later endpoint comparison, a bound that remembers only the Lipschitz size of a deformation is too coarse. The mixed estimates below retain two energy factors at every order. They also justify differentiating the higher-order tail along the material path used later.

\begin{lemma}\label{lem:analytic-coefficient-chain}

Let \(\mathbb D=B_1(0)\), let \(e_r(\theta)=e^{i\theta}\), and write
\(j^2=\lambda _1(\mathbb D)\). For a real function

\[
f\in W^{1,\infty}(\mathbb S^1)\cap H^{1/2}(\mathbb S^1),
\]

use the extension operator of
Lemma~\ref{lem:fixed-disk-extension-realization} and
set

\begin{equation}\label{eq:cubic-coefficient-001}
U_f:=\operatorname{Ext}(fe_r),
\qquad
L_f:=\|DU_f\|_{L^\infty(\mathbb D)},
\qquad
E_f:=\|DU_f\|_{L^2(\mathbb D)}.
\end{equation}

Thus \(\int_{\mathbb D}U_f=0\) and \(U_f|_{\partial\mathbb D}=fe_r\).
For a zero-mean vector field

\[
V\in\mathcal X:=\left\{V\in W^{1,\infty}(\mathbb D;\mathbb R^2)
 \cap H^1(\mathbb D;\mathbb R^2):\int_{\mathbb D}V=0\right\},
\]

with \(\|DV\|_\infty<1\), define

\begin{equation}\label{eq:cubic-coefficient-002}
J_{{\rm sp}}(V):=\frac1\pi F((I+V)(\mathbb D)),
\qquad
 F(\Omega):=|\Omega|\lambda _1(\Omega).
\end{equation}

The first derivative \(DJ_{{\rm sp}}(0)[V]\) vanishes in every such
direction. Remove the modes \(0,1,-1\) from \(f\) to obtain \(f^\perp\).
By the extension independence in
Theorem~\ref{thm:uniform-fixed-disk-expansion},
the scalar

\[
Q_B(f^\perp):=\frac12D^2J_{{\rm sp}}(0)[U_f,U_f],
\]

depends only on \(f\). Define

\begin{equation}\label{eq:cubic-coefficient-003}
\begin{aligned}
\widetilde{R}(f)
&:=\frac1{j^2}\left(
 J_{{\rm sp}}(U_f)-j^2-Q_B(f^\perp)\right),\\
C_3(f,g,h)
&:=\frac1{6j^2}D^3J_{{\rm sp}}(0)[U_f,U_g,U_h],
\qquad C_3(f):=C_3(f,f,f),\\
R_{\ge4}(f)&:=
 \widetilde{R}(f)-C_3(f).
\end{aligned}
\end{equation}

Because \(F(\mathbb D)=\pi j^2\), the first line is exactly the
fixed-disk normalized remainder of
Definition~\ref{def:disk-normalized-splitting}
whenever its radial chart is \(f\) and its quadratic term is
\(Q(f)=Q_B(f^\perp)/j^2\).

There are absolute constants \(c,C>0\) such that, if \(L_f<c\), then

\[
\widetilde{R}(f)
=C_3(f)+R_{\ge4}(f),
\]

\[
|R_{\ge4}(f)|
\le
C\|f\|_{W^{1,\infty}}^2\|f\|_{H^{1/2}}^2,
\]

and the polarization of the cubic coefficient satisfies

\[
|C_3(f,g,h)|
\le
C\sum_{\rm cyc}
\|f\|_{W^{1,\infty}}
\|g\|_{H^{1/2}}\|h\|_{H^{1/2}}.
\]

The trilinear estimate holds for all real \(f,g,h\) in the displayed
intersection space; no smallness is needed to define the derivative at
the origin.

\end{lemma}

\begin{proof}
For $G=DU_f$, $L=\|G\|_\infty$, and $E=\|G\|_2$, the complex-amplitude coefficients are
\begin{equation}\label{eq:cubic-coefficient-004}
 F_z=I+zG,\qquad \delta_z=\det F_z,\qquad
 s_z=\pi^{-1}\!\int_{\mathbb D}\delta_z,
 \qquad \mathbf A_z=s_z\delta_zF_z^{-1}F_z^{-T}.
\end{equation}
The determinant and area factor are polynomials of degree two, with
\begin{equation}\label{eq:cubic-coefficient-005}
 |s_1|\le CE,\qquad |s_2|\le CE^2.
\end{equation}
On $|z|L<c$, the matrix inverses have convergent geometric expansions. At degree $r$, an argument equal to $u_0$ permits one deformation factor to be estimated in $L^2$; evaluation at $(u_0,u_0)$ permits two. All other factors are estimated in $L^\infty$. Hence
\begin{equation}\label{eq:cubic-coefficient-006}
\begin{aligned}
 \|A_r\|+\|M_r\|&\le C_0^rL^r,\\
 \|A_ru_0\|_{H^{-1}}+\|M_ru_0\|_{H^{-1}}
   &\le C_0^rL^{r-1}E,\\
 |A_r(u_0,u_0)|+|M_r(u_0,u_0)|
   &\le C_0^rL^{r-2}E^2\qquad(r\ge2).
\end{aligned}
\end{equation}
The normalized saddle recurrence and disk stationarity give
\begin{equation}\label{eq:cubic-coefficient-007}
 \lambda(z)=j^2+\sum_{r\ge2}\lambda_rz^r,
 \qquad
 |\lambda_r|\le K^rL^{r-2}E^2,
\end{equation}
with a radius and constants independent of $\operatorname{supp}DU_f$. The second coefficient is $Q_B$, and the third defines $C_3$. The remaining series satisfies
\begin{equation}\label{eq:cubic-coefficient-008}
 |R_{\ge4}(f)|
 \le \frac1{j^2}\sum_{r\ge4}K^rL^{r-2}E^2
 \le CL^2E^2
 \le C\|f\|_{W^{1,\infty}}^2\|f\|_{H^{1/2}}^2.
\end{equation}
For distinct amplitude directions the mixed recurrence preserves two separate energy factors:
\begin{equation}\label{eq:cubic-coefficient-009}
 |C_3(f_1,f_2,f_3)|
 \le C\sum_{\rm cyc}L_1E_2E_3
 \le C\sum_{\rm cyc}
 \|f_1\|_{W^{1,\infty}}
 \|f_2\|_{H^{1/2}}\|f_3\|_{H^{1/2}}.
\end{equation}
Since the extension realizes the prescribed radial domain, these coefficients are its spectral coefficients for traces in $W^{1,\infty}\cap H^{1/2}$.
\end{proof}

\section{Positive atomic energies and discrete rigidity}\label{sec:curvature-coordinates}

The discrete comparison is needed in a form that resolves near-minimizers. We first obtain a positive certificate for inverse-power energies. Its remainder gives weight control, low-frequency coercivity, and the microscopic conclusion stated in Theorem~\ref{thm:atomic-rigidity}\textup{(iv)}. The sharp inequality applies to positive atomic measures; centering is imposed for the subsequent quantitative conclusions as specified below. We then compare the inverse-cubic defect with the exact Bessel defect, using constrained local coercivity when an additive error would be too large.

\subsection{Inverse-power and Bessel energies}

\begin{definition}
\label{def:atomic-energies}
Fix $N\ge2$.  An $N$-admissible atomic probability measure is
\[
 \mu=\sum_{\ell=1}^M w_\ell\delta_{\theta_\ell},\qquad
 1\le M\le N,\quad w_\ell>0,\quad \sum_\ell w_\ell=1,
\]
with distinct nodes modulo $2\pi$.  Put
\[
 \widehat\mu(k)=\sum_\ell w_\ell e^{-ik\theta_\ell},\qquad
 \sigma_{N,\gamma}=\frac1N\sum_{j=0}^{N-1}\delta_{\gamma+2\pi j/N},
\]
\[
 E_{{\rm at},p}(\mu)=\sum_{k=2}^\infty\frac{|\widehat\mu(k)|^2}{k^p},
 \quad d_{p,N}=E_{{\rm at},p}(\mu)-\frac{\zeta(p)}{N^p},
 \quad
 P_{{\rm at},q}(\mu)=\sum_{k\ge1}q^k|\widehat\mu(k)|^2,
\]
\begin{equation}\label{eq:curvature-coordinates-001}
 G_N(q;\mu)=P_{{\rm at},q}(\mu)-\frac{q^N}{1-q^N}.
\end{equation}
Thus $\widehat\sigma_{N,0}(k)=\mathbf1_{\{N\mid k\}}$.

Let $j=j_{0,1}$ and, for $k\ge2$, define
\[
 r_k=j\frac{J_{k+1}(j)}{J_k(j)},\qquad d_k=\frac{r_k}{k^4},\qquad
 \beta_k=1+j\frac{J_k'(j)}{J_k(j)}=k+1-r_k,
\]
\begin{equation}\label{eq:curvature-coordinates-002}
 b_k=\frac{\beta_k}{k^4}=\frac1{k^3}+\frac1{k^4}-d_k.
\end{equation}
The fixed-disk estimates give $J_k(j)>0$, $0<r_k<6/(k+1)$,
$d_k=O(k^{-5})$, and $b_k=O(k^{-3})$. Set
\[
 E_{{\rm at},d}(\mu)=\sum_{k\ge2}d_k|\widehat\mu(k)|^2,
 \quad d_{d,N}=E_{{\rm at},d}(\mu)-\sum_{m\ge1}d_{mN},
\]
\[
 E_{{\rm at},b}(\mu)=\sum_{k\ge2}b_k|\widehat\mu(k)|^2,
 \quad d_{B,N}=E_{{\rm at},b}(\mu)-E_{{\rm at},b}(\sigma_{N,0}).
\]
Then
\begin{equation}\label{eq:curvature-coordinates-003}
 E_{{\rm at},b}(\sigma_{N,0})=\frac{\zeta(3)}{N^3}+\frac{\zeta(4)}{N^4}
 -\sum_{m\ge1}d_{mN},\qquad
 d_{B,N}=d_{3,N}+d_{4,N}-d_{d,N}.
\end{equation}
\end{definition}

\subsection{A positive certificate for the uniform root measure}\label{sec:atomic-rigidity}

Choose a polynomial below the pair potential, with nonnegative Fourier coefficients and equality at the nonzero regular root differences. Summing this lower bound over pairs produces squares of Fourier moments and a diagonal weight term; see Section~4 of Cohn and Kumar \citep{CK07} for the underlying harmonic method. In the identity below we keep the interpolation remainder as well. Its sign yields the minimum, and its size will control departures from regular weights and gaps.

\begin{lemma}
\label{lem:atomic-fejer-certificate}
Let $\mu$ be $N$-admissible and $0<q<1$. Put
\[
 u(\theta)=2-2\cos\theta,\qquad a=q+q^{-1}-2,\qquad
 F_N(\theta)=\left(\frac{\sin(N\theta/2)}{\sin(\theta/2)}\right)^2,
\]
with $F_N(0)=N^2$, write $F_N(\theta)=F_N^{\rm alg}(u(\theta))$, and define
\[
 h(u)=\frac{1-F_N^{\rm alg}(u)/F_N^{\rm alg}(-a)}{a+u},\qquad
 h(u(\theta))=c_0+2\sum_{k=1}^{N-2}c_k\cos(k\theta).
\]
Then $h$ is a polynomial of degree at most $N-2$,
\[
 F_N^{\rm alg}(-a)=q^{1-N}\frac{(1-q^N)^2}{(1-q)^2},\qquad
 c_0>c_1>\cdots>c_{N-2}>0,
\]
\[
 d:=a^{-1}-h(0)=\frac{N^2q^N}{(1-q^N)^2},
\]
and
\begin{equation}\label{eq:atomic-rigidity-001}
\begin{aligned}
G_N(q;\mu)
&=\frac{q^{-1}-q}{2}\Bigg[
2\sum_{k=1}^{N-2}c_k|\widehat\mu(k)|^2
+d\left(\sum_{\ell=1}^{M}w_\ell^2-\frac1N\right)\\
&\hspace{22mm}
+\sum_{\ell\ne m}w_\ell w_m
\frac{F_N(\theta_\ell-\theta_m)}
 {F_N^{\rm alg}(-a)[a+u(\theta_\ell-\theta_m)]}\Bigg].
\end{aligned}
\end{equation}
Every term in brackets is nonnegative.  In particular, with
\[
 A_N(q)=\frac{q^{-1}-q}{2}d,\qquad A_{N,k}(q)=(q^{-1}-q)c_k,
\]
\begin{equation}\label{eq:atomic-rigidity-002}
 G_N(q;\mu)\ge A_N(q)\left(\sum_\ell w_\ell^2-\frac1N\right),
 \qquad
 G_N(q;\mu)\ge\sum_{k=1}^{N-2}A_{N,k}(q)|\widehat\mu(k)|^2,
\end{equation}
and, for $1\le k\le N-2$ and $n=N-k$,
\begin{equation}\label{eq:atomic-rigidity-003}
 \frac{A_{N,k}(q)}{q^k}
 =\frac{1-q^{2n}-n(q^{-1}-q)q^n}{(1-q^N)^2}.
\end{equation}
\end{lemma}

\begin{proof}
The factorization
\[
 F_N(\theta)=z^{1-N}\left(\frac{z^N-1}{z-1}\right)^2,
 \qquad z=e^{i\theta}
\]
reduces the construction of $h$ to polynomial division by $a+u$. The quotient satisfies
\[
 (q+q^{-1}-2\cos\theta)h(u(\theta))
 =1-\frac{F_N(\theta)}{F_N^{\rm alg}(-a)},
\]
and its Fourier coefficients solve
\[
 (q+q^{-1})c_k-c_{k-1}-c_{k+1}
 =-\frac{N-k}{F_N^{\rm alg}(-a)}.
\]
Solving the finite recurrence gives the coefficients in the statement. The remainder has the sign needed for the comparison:
\[
 f(u(\theta))-h(u(\theta))
 =\frac{F_N(\theta)}{F_N^{\rm alg}(-a)[a+u(\theta)]}\ge0.
\]

Sum
\[
 (q^{-1}-q)f(u(\theta))=1+2\sum_{k\ge1}q^k\cos(k\theta),
\]
over pairs of atoms, subtracting the value at the uniform roots. The diagonal pairs give $\sum_jw_j^2-1/N$; the off-diagonal pairs retain the nonnegative interpolation remainder. The remaining low-frequency coefficients are
\[
 A_{N,k}(q)=q^k-F_N^{\rm alg}(-a)^{-1}
 \sum_{|j|\le N-1}(N-|j|)q^{|k-j|}.
\]
Their finite-sum formula gives positivity. All other factors are squares or positive denominators, proving the decomposition and both lower bounds.
\end{proof}

The theorem separates the sharp minimum from the information in its positive remainder. The latter will be used twice: low modes and weights enter the preliminary spectral bounds, while the microscopic conclusion becomes available only after the scaled defect tends to zero.

\begin{theorem}
\label{thm:atomic-rigidity}
Let $\mu$ be $N$-admissible.

\smallskip
\noindent\textup{\bfseries (i) Sharp Poisson and inverse powers.}
For every $0<q<1$,
\begin{equation}\label{eq:atomic-rigidity-004}
 P_{{\rm at},q}(\mu)\ge\frac{q^N}{1-q^N}.
\end{equation}
Equality for one $q$ holds if and only if $\mu=\sigma_{N,\gamma}$; in that case equality holds for every $q$.  Moreover, for every $p>1$,
\begin{equation}\label{eq:atomic-rigidity-005}
 \sum_{k\ge1}\frac{|\widehat\mu(k)|^2}{k^p}
 \ge\frac{\zeta(p)}{N^p},
\end{equation}
with the same equality cases.  If $\widehat\mu(1)=0$, the sum starts at
$k=2$.

\smallskip
Assume from now on that $\widehat\mu(1)=0$ and append zero weights if
$M<N$.

\smallskip
\noindent\textup{\bfseries (ii) Weight rigidity.}
\begin{equation}\label{eq:atomic-rigidity-006}
 \sum_{\ell=1}^{N}\left(w_\ell-\frac1N\right)^2
 \le2e^4N^2d_{3,N},
\end{equation}
so
\begin{equation}\label{eq:atomic-rigidity-007}
 N-M\le2e^4N^4d_{3,N},\qquad
 \max_\ell|Nw_\ell-1|^2\le2e^4N^4d_{3,N}.
\end{equation}

\smallskip
\noindent\textup{\bfseries (iii) Low-band coercivity.}
For every fixed $0<\alpha<1$ there are $N_\alpha,C_\alpha$ such that
\begin{equation}\label{eq:atomic-rigidity-008}
 \sum_{k=2}^{\lfloor\alpha N\rfloor}
 \frac{|\widehat\mu(k)|^2}{k^3}\le C_\alpha d_{3,N}
 \qquad(N\ge N_\alpha).
\end{equation}
For $N\ge20$,
\begin{equation}\label{eq:atomic-rigidity-009}
 \sum_{k=2}^{\lfloor N/10\rfloor}
 \frac{|\widehat\mu(k)|^2}{k^3}\le\frac54d_{3,N}.
\end{equation}

\smallskip
\noindent\textup{\bfseries (iv) Microscopic lattice rigidity.}
\label{thm:microscopic-rigidity}
If $N_\nu\to\infty$ and $N_\nu^4d_{3,N_\nu}(\mu_\nu)\to0$, then
for all large $\nu$ there are exactly $N_\nu$ atoms.  In cyclic order,
\[
 \varepsilon_j=Nw_j-1,\qquad
 \theta_{j+1}-\theta_j=\frac{2\pi}{N}(1+\eta_j),
\]
satisfy
\begin{equation}\label{eq:atomic-rigidity-010}
 \max_j(|\varepsilon_j|+|\eta_j|)\to0,
 \qquad
 \sum_j(\varepsilon_j^2+\eta_j^2)
 \le C N^4d_{3,N}.
\end{equation}
After a common rotation and cyclic relabeling, write
\[
 \theta_j=\gamma+jh+y_j,\qquad \sum_jy_j=0,
 \qquad u_j=w_j-N^{-1},\qquad h=2\pi/N.
\]
With $B_N=N^4d_{3,N}$,
\begin{equation}\label{eq:atomic-rigidity-011}
 \sum_j|y_j|^2\le CB_N,\quad
 \max_j|y_j|\le C\sqrt{B_N/N},\quad
 \sum_j|u_j|^2\le CB_N/N^2,\quad
 \max_j|u_j|\le CB_N^{1/2}/N.
\end{equation}
\end{theorem}

\begin{proof}
A zero Poisson defect forces the diagonal term in the certificate to vanish. There must therefore be $N$ atoms, each of weight $1/N$. Every off-diagonal difference must also be a zero of $F_N$, so the measure is $\sigma_{N,\gamma}$. Its Fourier support proves the converse. Integrate the certificate using
\[
 k^{-p}=\Gamma(p)^{-1}\int_0^\infty s^{p-1}e^{-ks}\,ds,
\]
to obtain the inverse-power inequality. Tonelli applies separately to the two nonnegative Fourier series. Vanishing of the integrated defect, together with positivity and continuity, forces $G_N(e^{-s};\mu)=0$ for all $s>0$ and gives the same equality case.

For the weight estimate, the rescaling at $p=3$ is
\begin{equation}\label{eq:atomic-rigidity-012}
 N^3d_{3,N}=\frac12\int_0^\infty t^2G_N(e^{-t/N};\mu)\,dt.
\end{equation}
Some $t\in[1,2]$ satisfies $G_N(e^{-t/N})\le2N^3d_{3,N}$. On this interval $A_N(e^{-t/N})\ge e^{-4}N$, and the diagonal term gives
\[
 \sum_\ell w_\ell^2-\frac1N\le2e^4N^2d_{3,N}.
\]
Appending zero weights when $M<N$ yields the missing-atom estimate and the bound for the largest weight deviation.

The low modes are controlled by the same certificate. Its coefficients satisfy
\begin{equation}\label{eq:atomic-rigidity-013}
 A_{N,k}(e^{-t/k})\ge c_\alpha e^{-t}
 \quad(2\le k\le\alpha N,\ 1\le t\le2).
\end{equation}
For this estimate put $n=N-k$ in the explicit formula. After division by $q^k$, the numerator is $1-e^{-2nt/k}-n(e^{t/k}-e^{-t/k})e^{-nt/k}$. It is at least $99/100$ for $k\le N/10$ and $t\ge1$. For $k\le\alpha N$, the compact ratio range gives $g((\alpha^{-1}-1)t)+O_\alpha(N^{-2})$, with $g(u)=1-2ue^{-u}-e^{-2u}>0$. For large $N$ this is at least $g(\alpha^{-1}-1)/2$. Integration over $1/k\le s\le2/k$ gives \eqref{eq:atomic-rigidity-008}. On $k\le N/10$, integrate over $s\ge1/k$: the identity $\int_1^\infty t^2e^{-t}\,dt=5/e$ makes the resulting coefficient larger than $4/5$, proving \eqref{eq:atomic-rigidity-009}.

Now suppose $B_N=N^4d_{3,N}\to0$. The weight bounds give $M=N$ for large $N$, $\max_j|\varepsilon_j|\to0$, and $\sum_j\varepsilon_j^2\le CB_N$. Fix $0<\alpha<1$, set $K=\lfloor\alpha N\rfloor$, and take
\[
 F_K^{\rm av}(\theta)=K^{-1}\left(\frac{\sin(K\theta/2)}{\sin(\theta/2)}\right)^2.
\]
Then low-band coercivity implies
\begin{equation}\label{eq:atomic-rigidity-014}
 \|F_K^{\rm av}*\mu-1\|_\infty\le C_\alpha N^2d_{3,N}^{1/2}=o(1).
\end{equation}
Rescale angular distance by $2\pi/N$. With $\alpha>1/2$, the convolution bound excludes collisions. It also excludes gaps tending to infinity, since the convolution at a midpoint of such a gap would tend to zero. The relevant kernel estimates are
\[
 N^{-1}F_K^{\rm av}(2\pi x/N)\longrightarrow
 \alpha\,\operatorname{sinc}^2(\pi\alpha x),
 \qquad
 N^{-1}F_K^{\rm av}(2\pi x/N)\le \frac{C_\alpha}{1+x^2}.
\]
Pointed limits are thus separated, relatively dense sets $\Lambda\subset\mathbb R$ containing zero, with masses converging uniformly to one. Separation makes the majorant summable. First pass to the limit for a countable dense family of $\alpha$, and then use dominated convergence to obtain
\begin{equation}\label{eq:atomic-rigidity-015}
 \sum_{x\in\Lambda}\alpha\,\operatorname{sinc}^2(\pi\alpha x)=1
 \qquad(0<\alpha<1).
\end{equation}
When $\alpha\uparrow1$, the point at zero exhausts the sum. Fatou's lemma puts every other point in $\mathbb Z\setminus\{0\}$. Poisson summation for the triangular Fourier transform (see Chapter~5, Section~3 of \cite{SteinShakarchi03}) gives
\begin{equation}\label{eq:atomic-rigidity-016}
 \sum_{m\in\mathbb Z}\alpha\,\operatorname{sinc}^2(\pi\alpha m)=1
 \qquad(0<\alpha\le1).
\end{equation}
For irrational $\alpha$, all nonzero-integer terms are positive. Comparison excludes missing integers, hence $\Lambda=\mathbb Z$. A gap with relative error bounded away from zero would persist in a pointed limit at its left endpoint. Therefore $\max_j|\eta_j|\to0$.

To quantify this conclusion, retain adjacent pairs in one orientation with $q=e^{-t/N}$, $1\le t\le2$. At $\theta=2\pi(1+\eta)/N$, $|\eta|\le1/4$, their coefficient obeys
\[
 \frac{F_N(\theta)}{2F_N^{\rm alg}(-a)}
 \frac{1-q^2}{1-2q\cos\theta+q^2}\ge cN\eta^2.
\]
Since $Nw_j\to1$ uniformly, $G_N(e^{-t/N})\ge cN^{-1}\sum_j\eta_j^2$. Integration gives $\sum_j\eta_j^2\le CB_N$. Apply cyclic Poincar\'e and the short-arc estimate to $y_{j+1}-y_j=h\eta_j$ for the position bounds, and use $u_j=\varepsilon_j/N$ for the weight-coordinate bounds.
\end{proof}

\subsection{Constrained coercivity of the local Bessel form}

The global comparison will encounter defects too small to absorb an additive error. We first establish a local comparison that has no such additive loss. The relevant variations satisfy the linearized mass and centering constraints, with rotation fixed; these constraints determine the coercive directions.

\begin{lemma}
With $(r_k),(d_k)$ as above, there is a completely monotone $R$ on
$[2,\infty)$ such that
\[
 R(k)=r_k,\qquad 0\le R(s)\le\frac3s.
\]
Consequently $D(s)=R(s)s^{-4}$ is completely monotone,
$D(k)=d_k$, and $0\le D(s)\le3s^{-5}$.

Moreover there is a nonnegative locally integrable $h_d$ such that
\begin{equation}\label{eq:atomic-rigidity-017}
 D(s)=\int_0^\infty e^{-st}h_d(t)\,dt,
 \qquad
 0\le h_d(t)\le\frac{e^3}{2}t^4\quad(0<t\le1).
\end{equation}
\end{lemma}

\begin{proof}
Set $z=j^2/4$. The alternating Bessel expansion, with $j^2<6$, gives
\[
 J_k(j)=\frac{(j/2)^k}{k!}A_k,
 \qquad 1-\frac z{k+1}\le A_k\le1,
\]
for $k\ge1$. Hence $J_k(j)>0$ and $0<r_k<6/(k+1)$, and the ratios satisfy
\begin{equation}\label{eq:atomic-rigidity-018}
 r_k=\frac{j^2}{2(k+1)-r_{k+1}}.
\end{equation}
Extend this recurrence by the map
\[
 (Tf)(s)=\frac{j^2}{2(s+1)-f(s+1)}
\]
on $0\le f\le2$ in $C_b([2,\infty))$. It preserves the cone, has $Tf(s)\le3/s$, and contracts by at most $3/8$. Let $R$ be its fixed point.

The expansion
\[
 Tf=j^2[2(s+1)]^{-1}\sum_{m\ge0}
 \left(\frac{f(s+1)}{2(s+1)}\right)^m,
\]
shows that iteration preserves complete monotonicity: this class is closed under products and pointwise limits; see Theorem~1.4 and Corollary~1.6 in \cite{SchillingSongVondracek12}. Iteration from zero gives a completely monotone $R$. The bounded recurrence is unique, so $R(k)=r_k$, and $D(s)=R(s)s^{-4}$ interpolates $d_k$.

Bernstein's theorem gives $R(u+2)=\int e^{-uy}\,d\tau(y)$ with $\tau\ge0$. For $0<t\le1$, $R(2+t^{-1})\le3t$ yields $\tau([0,t])\le3et$. Convolving with the Laplace density of $(u+2)^{-4}$ gives, by Tonelli,
\[
 h_d(t)=\frac16\int_{[0,t]}e^{2y}(t-y)^3\,d\tau(y).
\]
The last measure bound implies $h_d(t)\le e^3t^4/2$.
\end{proof}

\begin{definition}
\label{def:atomic-inverse-cubic-blocks}
For $1\le r\le\lfloor N/2\rfloor$ and $Z,X\in\mathbb C$ define
\begin{equation}\label{eq:atomic-rigidity-019}
\begin{aligned}
H_{3,N}^{(r)}(Z,X)
={}&\mathbf1_{\{r\ge2\}}r^{-3}|Z|^2\\
&+\sum_{m\ge1}\Bigl
 (mN+r)^{-3}|Z+i(mN)X|^2
 +(mN-r)^{-3}|Z-i(mN)X|^2\\
&\hspace{35mm}-2(mN)^{-3}(mN)^2|X|^2\Bigr).
\end{aligned}
\end{equation}
If $N$ is even and $r=N/2$, require $X$ and $Z-irX$ to be real and take one
half of the right-hand side. The grouped series is absolutely convergent.
\end{definition}

\begin{lemma}
\label{lem:pair-energy-kernel-calculus}
Let $a_k=k^{-3}$ or $a_k=b_k$ and
$K_a(s)=\sum_{k\ge2}a_k\cos(ks)$. Then, with
$d(s)=\operatorname{dist}(s,2\pi\mathbb Z)$,
\begin{equation}\label{eq:atomic-rigidity-020}
 |K_a|+|K_a'|\le C,
 \quad |K_a''(s)|\le C(1+\log(2/d(s))),
 \quad |K_a'''(s)|\le C/d(s),
\end{equation}
on the punctured circle. Moreover
\begin{equation}\label{eq:atomic-rigidity-021}
 K_3'''(s)=\frac12\cot(s/2)-\sin s,
\end{equation}
and, with $\kappa(s)=s^{-4}-D(s)$,
\begin{equation}\label{eq:atomic-rigidity-022}
 |\kappa^{(j)}(s)|\le C_js^{-4-j}\quad(s\ge4,\ j=0,1,2),
\end{equation}
so for $M\ge8$, $1\le r\le M/2$,
\begin{equation}\label{eq:atomic-rigidity-023}
\begin{aligned}
 |\kappa(M+r)|+|\kappa(M-r)|&\le CM^{-4},\\
 |\kappa(M-r)-\kappa(M+r)|&\le CrM^{-5},\\
 |\kappa(M+r)+\kappa(M-r)-2\kappa(M)|&\le Cr^2M^{-6}.
\end{aligned}
\end{equation}
Finally, on
\begin{equation}\label{eq:atomic-rigidity-024}
 \|x\|_\infty\le\frac\pi{4N},\qquad \|u\|_\infty\le\frac1{2N},
\end{equation}
the finite-pair energy
\[
 E_{{\rm pair},a}(u,x)=\sum_{j,\ell}(N^{-1}+u_j)(N^{-1}+u_\ell)
 K_a(2\pi(j-\ell)/N+x_j-x_\ell),
\]
is $C^3$ and
\begin{equation}\label{eq:atomic-rigidity-025}
 \max_{1\le q\le3}\|D^qE_{{\rm pair},a}(u,x)\|\le CN^4.
\end{equation}
If
\[
 U_r=\sum_j\alpha_je^{2\pi irj/N},\quad
 X_r=N^{-1}\sum_j\xi_je^{2\pi irj/N},\quad Z_r=U_r+irX_r,
\]
the coefficient of $t^2$ at the regular configuration is the sum over
paired DFT residues of
\begin{equation}\label{eq:atomic-rigidity-026}
\begin{aligned}
 H_{a,N}^{(r)}={}&\mathbf1_{\{r\ge2\}}a_r|Z_r|^2\\
 &+\sum_{m\ge1}\{a_{mN+r}|Z_r+i(mN)X_r|^2
 +a_{mN-r}|Z_r-i(mN)X_r|^2
 -2a_{mN}(mN)^2|X_r|^2\},
\end{aligned}
\end{equation}
with the half-block convention for the Nyquist residue.
\end{lemma}

\begin{proof}
Absolute convergence gives $K_a$ and $K_a'$. Abel summation identifies $K_3^{\prime\prime\prime}$ and bounds its derivatives away from zero. Since $b_k-k^{-3}=O(k^{-4})$, summation by parts gives the corresponding bounds for $K_b$. The Laplace representation gives $|D^{(j)}(s)|\le C_js^{-5-j}$; integrating these derivative estimates proves \eqref{eq:atomic-rigidity-023}.

In the box \eqref{eq:atomic-rigidity-024}, distinct nodes stay $c/N$ apart modulo $2\pi$. The diagonal terms have no position dependence. The finite pair sum is therefore $C^3$, with \eqref{eq:atomic-rigidity-025}. To compute the Hessian, impose a Fourier cutoff and diagonalize by the discrete transform. Before letting the cutoff tend to infinity, combine $mN-r,mN,mN+r$. Their second differences are summable, so the grouped expressions converge to the classical pair-energy Hessian and give \eqref{eq:atomic-rigidity-026}. The self-conjugate Nyquist residue occurs once when $N$ is even, giving the half-block factor.
\end{proof}

\begin{theorem}
\label{thm:local-bessel-coercivity}
There are absolute $c_*,C>0$ and $N_0$ such that the following hold for
$N\ge N_0$.

\smallskip
\noindent\textup{\bfseries (i) Pure inverse-cubic blocks.}
For $2\le r\le N/2$,
\begin{equation}\label{eq:atomic-rigidity-027}
 H_{3,N}^{(r)}\ge c_*\left(
 \frac{|Z_r|^2}{r^3}+\frac{r^2|X_r|^2}{N^3}\right),
\end{equation}
and for $r=1$, under $Z_1=0$,
$H_{3,N}^{(1)}\ge c_*N^{-3}|X_1|^2$.

\smallskip
\noindent\textup{\bfseries (ii) Bessel alias perturbation.}
For $\kappa_k=b_k-k^{-3}=(1-r_k)k^{-4}$ and $M=mN$,
\begin{equation}\label{eq:atomic-rigidity-028}
\begin{aligned}
|\kappa_{M+r}|+|\kappa_{M-r}|&\le CM^{-4},\\
|\kappa_{M-r}-\kappa_{M+r}|&\le CrM^{-5},\\
|\kappa_{M+r}+\kappa_{M-r}-2\kappa_M|&\le Cr^2M^{-6},
\end{aligned}
\end{equation}
and the grouped alias correction obeys
\begin{equation}\label{eq:atomic-rigidity-029}
 |H_{\kappa,N}^{(r),\mathrm{alias}}|
 \le\frac CN\left(\frac{|Z_r|^2}{r^3}
 +\frac{r^2|X_r|^2}{N^3}\right).
\end{equation}
Also $\kappa_r\ge0$ for $r\ge3$ and
$\kappa_2>-2^{-3}/10$.

\smallskip
\noindent\textup{\bfseries (iii) Local constrained comparison.}
Let $A=20$ and
\begin{equation}\label{eq:atomic-rigidity-030}
 w_j=N^{-1}+u_j,\qquad \theta_j=2\pi j/N+x_j,
 \qquad \|(u,x)\|_2\le N^{-A},
\end{equation}
with the exact constraints
\begin{equation}\label{eq:atomic-rigidity-031}
 \sum_ju_j=0,\qquad \sum_jw_je^{i\theta_j}=0,
 \qquad N^{-1}\sum_jx_j=0.
\end{equation}
Then the configuration is $N$-admissible and
\begin{equation}\label{eq:atomic-rigidity-032}
 d_{B,N}\ge\frac12d_{3,N}.
\end{equation}
For the tangent space
\[
 \mathcal V_N=\{(\alpha,\xi):U_0=X_0=Z_1=0\},
\]
one also has
\begin{equation}\label{eq:atomic-rigidity-033}
 H_{3,N}(y)\ge cN^{-4}|y|^2,
\end{equation}
regular criticality on $U_0=X_0=0$, and the $C^3$ bound \eqref{eq:atomic-rigidity-025} for both
$d_{3,N}$ and $d_{B,N}$.
\end{theorem}

\begin{proof}
Write one inverse-cubic block as $A|Z|^2+2\operatorname{Re}(iBZ\overline X)+C|X|^2$. In the range $r\le N/16$, the elementary estimates
\[
 (1-x)^{-3}+(1+x)^{-3}-2\ge12x^2,
 \quad 0\le(1-x)^{-3}-(1+x)^{-3}\le7x,
\]
for $0\le x\le1/16$ give
\[
 A\ge r^{-3},\qquad C\ge12\zeta(3)r^2N^{-3},\qquad
 |B|\le7\zeta(3)rN^{-3}.
\]
Weighted Young's inequality absorbs the cross term.

For $N/16\le r\le N/2$, positivity is more transparent in the certificate. Put $U=Z-irX$ and, away from the Nyquist class, use
\[
 u_j=\frac2N\operatorname {Re}(Ue^{-ir\theta_j^0}),\qquad
 x_j=2\operatorname {Re}(Xe^{-ir\theta_j^0}),\qquad
 \theta_j^0=2\pi j/N.
\]
This real variation preserves mass and remains admissible for a small path parameter. Finite summation gives
\[
 \sum_j u_j^2=\frac2N|U|^2,
 \qquad
 \sum_j|x_{j+1}-x_j|^2
   =2N|e^{2\pi ir/N}-1|^2|X|^2.
\]
The quadratic coefficient is \eqref{eq:atomic-rigidity-019}. At $q=e^{-t/N}$, $1\le t\le2$, the diagonal coefficient is at least $cN$. The adjacent-pair factor vanishes quadratically at $h$, with second derivative at least $cN^3$. Keeping these terms and integrating gives
\[
 H_{3,N}^{(r)}(Z,X)
 \ge c\bigl(N^{-3}|U|^2+N^{-1}|X|^2\bigr),
\]
since $|e^{2\pi ir/N}-1|\ge2\sin(\pi/16)$. The first harmonic adds no quadratic term:
\[
 \left.\frac d{d\tau}\widehat\mu_\tau(1)\right|_{\tau=0}
 =\sum_j\left(u_j+\frac iN x_j\right)e^{i\theta_j^0}=0.
\]
Here neither $1+r$ nor $1-r$ is divisible by $N$. Since $Z=U+irX$ and $r\sim N$, this proves \eqref{eq:atomic-rigidity-027}. At even Nyquist use the real variation $(-1)^j$ and its half-block convention. At $r=1$, substituting $Z_1=0$ gives $H_{3,N}^{(1)}\ge12\zeta(3)N^{-3}|X_1|^2$.

For the Bessel perturbation, sum the grouped finite-difference bounds \eqref{eq:atomic-rigidity-023} over aliases. The result is
\[
 C N^{-4}(|Z|^2+r|Z||X|+r^2|X|^2).
\]
Young's inequality, with $r\le N/2$, gives \eqref{eq:atomic-rigidity-029}. For $k\ge3$, $R(k)\le3/k$ implies $\kappa_k\ge0$. At the exceptional index, $J_0(j)=0$ gives $r_2=4-j^2/2$; using $28/5<j^2<6$ gives $r_2<6/5$ and $\kappa_2>-1/80$. The $r=2$ lower bound absorbs this loss, so on $\mathcal V_N$,
\begin{equation}\label{eq:atomic-rigidity-034}
 H_{b,N}(y)\ge\frac45H_{3,N}(y).
\end{equation}
Parseval and $Z_1=0$ also give \eqref{eq:atomic-rigidity-033}.

The exact center constraint differs quadratically from its tangent equation. Write it as $L_N^{\rm cen}z+Q_N^{\rm cen}(z)=0$, where $L_N^{\rm cen}z=Z_1(z)$ and $|Q_N^{\rm cen}(z)|\le|z|^2$. The correction $n_j(q)=2N^{-1}\operatorname{Re}(qe^{-2\pi ij/N})$ has zero mass, satisfies $L_N^{\rm cen}(n(q),0)=q$, and has norm at most $2N^{-1/2}|q|$. Therefore
\begin{equation}\label{eq:atomic-rigidity-035}
 v=z-(n(L_N^{\rm cen}z),0)\in\mathcal V_N,
 \qquad |z-v|\le2N^{-1/2}|z|^2.
\end{equation}
Stationarity at the roots and the pair-energy $C^3$ bound give, for $a=3,b$,
\begin{equation}\label{eq:atomic-rigidity-036}
 d_{a,N}(z)=H_{a,N}(v)+O(N^4|z|^3).
\end{equation}
The leading term is at least a fixed multiple of $N^{-4}|z|^2$. The relative remainder is bounded by $CN^8|z|\le CN^{-12}$ and is absorbed for large $N$. The same box gives positive masses and separated nodes.
\end{proof}

\subsection{A global comparison of the two atomic defects}

The positive certificate controls the inverse-cubic defect, while the spectral quadratic form contains the exact Bessel weights. We compare their difference at a fixed frequency threshold. Where the resulting additive tail is not absorbable, microscopic rigidity places the configuration in the local constrained chart just analyzed.

\begin{theorem}
\label{thm:bessel-coercivity}
There is $N_{\rm KB}$ such that, for every $N\ge N_{\rm KB}$ and every
$N$-admissible $\mu$ with $\widehat\mu(1)=0$,
\begin{equation}\label{eq:atomic-rigidity-037}
 4d_{B,N}(\mu)\ge\frac14d_{3,N}(\mu)\ge0.
\end{equation}
Moreover, $d_{3,N}=0$ if and only if $\mu=\sigma_{N,\gamma}$; hence
$d_{B,N}=0$ has the same equality case.  The coordinate estimates
\eqref{eq:atomic-rigidity-011} hold with $B_N=N^4d_{3,N}$ whenever $B_N\to0$.
\end{theorem}

\begin{proof}
The centered defects admit the representations
\begin{equation}\label{eq:atomic-rigidity-038}
 d_{d,N}=\int_0^\infty h_d(t)G_N(e^{-t};\mu)\,dt,
 \qquad
 d_{3,N}=\frac12\int_0^\infty t^2G_N(e^{-t};\mu)\,dt.
\end{equation}
Split the first integral at $a\in(0,1]$. Its part below $a$ is at most $e^3a^2d_{3,N}$. On the complementary interval, $D(k)\le3k^{-5}$ for $k\ge3$ and $\kappa_2>-1/80$ imply $d_k\le(3/5)k^{-3}$ for $k\ge2$. The exponential tail has the bound
\[
 \int_a^\infty e^{-kt}h_d(t)dt\le e^{-ka/2}D(k/2)
 \le96e^{-ka/2}k^{-5}.
\]
The regular subtraction is favorable, so the high-frequency contribution is at most
\[
 \sum_{k>N/10}e^{-ka/2}k^{-5}|\widehat\mu(k)|^2
 \le Ce^{-aN/20}N^{-2}
       \sum_{k>N/10}\frac{|\widehat\mu(k)|^2}{k^3}
 \le Ce^{-aN/20}(N^{-2}d_{3,N}+N^{-5}).
\]
Low-band coercivity bounds the part $2\le k\le N/10$ by $3d_{3,N}/4$. Consequently,
\begin{equation}\label{eq:atomic-rigidity-039}
 d_{d,N}\le
 \left(\frac34+e^3a^2+CN^{-2}e^{-aN/20}\right)d_{3,N}
 +Ce^{-aN/20}N^{-5}.
\end{equation}
Choose $e^3a^2\le1/32$. If $d_{3,N}\ge C_1e^{-aN/20}N^{-5}$, a sufficiently large fixed $C_1$ gives $d_{d,N}\le15d_{3,N}/16$. Since $d_{4,N}\ge0$, \eqref{eq:curvature-coordinates-003} yields $d_{B,N}\ge d_{3,N}/16$.

Below that threshold, microscopic rigidity gives exactly $N$ atoms and, after rotation,
\[
 \|(u,x)\|_2^2
 \le C\left(B_N+\frac{B_N}{N^2}\right),
 \qquad B_N=N^4d_{3,N}
 \le C_1e^{-aN/20}N^{-1}.
\]
This assertion is uniform, since a contrary sequence would have $B_N\to0$. For large $N$ the coordinate norm is at most $N^{-20}$. Rotation and cyclic relabeling preserve mass and centering, and the angular slice has zero mean. The local constrained comparison applies and gives the stronger bound $d_{B,N}\ge d_{3,N}/2$.

Vanishing of $d_{B,N}$ forces $d_{3,N}=0$, hence a rotated uniform-root measure. Its Fourier support gives equality in the reverse direction. The remaining coordinate estimates are \eqref{eq:atomic-rigidity-011}.
\end{proof}

\section{Polar coordinates and the critical defect scale}\label{sec:polar-geometry}

We next return from turning measures to domains. The eigenvalue expansion is written in the true radial displacement, whereas the atomic energy uses the twice-integrated turning measure. Their difference is a nonlinear reconstruction error. Separating it from the higher spectral terms gives an exact identity for the eigenvalue difference; minimality will first use its absolute bounds to improve the defect. We also record the zero-defect reconstruction needed to recover the geometric equality case.

\subsection{Polar reconstruction and cyclic coordinates}

\begin{definition}
\label{def:polar-coordinates}
Fix $N\ge2$ and $h=2\pi/N$.  For a positive atomic probability measure
\[
 \nu=\sum_{j=0}^{M-1}w_j\delta_{\psi_j},\qquad M\le N,
\]
with distinct nodes, use
\begin{equation}\label{eq:curvature-coordinates-004}
 \widehat\nu(k)=\sum_jw_je^{-ik\psi_j},\qquad
 E_N(\nu)=\sum_{k\ge1}\frac{|\widehat\nu(k)|^2}{k^3}
 -\frac{\zeta(3)}{N^3},
 \qquad B_N(\nu)=N^4E_N(\nu).
\end{equation}
If $M=N$, choose cyclic lifts and write
\begin{equation}\label{eq:curvature-coordinates-005}
 a_j=Nw_j,
 \qquad \psi_{j+1}-\psi_j=h(1+\eta_j),
 \qquad \sum_j\eta_j=0.
\end{equation}

Let $P$ be a nondegenerate convex polygon translated to its polar center and
lying in the fixed radial chart
\[
 \partial P=\{(1+\rho(\theta))e^{i\theta}:\theta\in\mathbb T\},
 \qquad \|\rho\|_{W^{1,\infty}}\le\varepsilon_{\rm chart}<\frac12.
\]
At the effective vertices let $\psi_j=\arg z_j$ and let $\omega_j$ be the
exterior angles.  Define
\begin{equation}\label{eq:curvature-coordinates-006}
 w_j=\frac{\omega_j}{2\pi},\qquad
 \nu_P=\sum_jw_j\delta_{\psi_j},
\end{equation}
so $\nu_P$ is a probability measure and the polar-center equation is
\begin{equation}\label{eq:curvature-coordinates-007}
 \widehat\nu_P(1)=0.
\end{equation}
Put
\begin{equation}\label{eq:curvature-coordinates-008}
 E_N(P)=E_N(\nu_P)=d_{3,N}(\nu_P),
 \qquad B_N(P)=N^4E_N(P).
\end{equation}
The centered quasi-uniform chart is the case $M=N$ with
$|a_j-1|+|\eta_j|\le1/4$.  Fix the rotation gauge
\begin{equation}\label{eq:curvature-coordinates-009}
 \gamma=\frac1N\sum_j(\psi_j-jh),\qquad
 y_j=\psi_j-\gamma-jh,
 \qquad \sum_jy_j=0,
\end{equation}
rotate by $-\gamma$, and retain the symbols $\psi_j=jh+y_j$.  Then
\begin{equation}\label{eq:curvature-coordinates-010}
 x_j=\frac{\psi_j}{2\pi}\in\Tunit,
 \qquad x_{j+1}-x_j=\frac{1+\eta_j}{N},
\end{equation}
and the canonical chord is
\begin{equation}\label{eq:curvature-coordinates-011}
 a_j(t)=1+t(a_j-1),\qquad
 \psi_j(t)=jh+ty_j,
 \qquad T_t(jh)=\psi_j(t),
 \qquad J_t=T_t'=1+t\eta.
\end{equation}
On $\Tunit$ let $\Phi_N$ be the periodic piecewise-affine phase with
\begin{equation}\label{eq:curvature-coordinates-012}
 \Phi_N(x_{j+1})-\Phi_N(x_j)=2\pi\eta_j,
 \qquad e^{i\Phi_N(x_j)}=e^{2\pi iNx_j}.
\end{equation}
The finite columns are
\begin{equation}\label{eq:curvature-coordinates-013}
 H_\nu(k)=\overline{\widehat\nu(k)},\qquad f_k=e^{ik\Phi_N},
 \qquad c_k(r)=\int_{\Tunit}f_k(x)e^{2\pi irx}\,dx,
 \qquad C_k(r)=H_\nu(kN+r),
\end{equation}
with
\begin{equation}\label{eq:curvature-coordinates-014}
 I_N=\{-\lfloor(N-1)/2\rfloor,\ldots,\lfloor N/2\rfloor\},
 \qquad \Pi=\Pi_{I_N}.
\end{equation}
In the exact constraint coordinates,
\begin{equation}\label{eq:curvature-coordinates-015}
 x_j^{\rm ex}=y_j,
 \qquad u_j=w_j-N^{-1}=N^{-1}(a_j-1).
\end{equation}
These conventions fix the cyclic gauge for the comparisons. If, for a fixed $C_0$,
\begin{equation}\label{eq:curvature-coordinates-016}
 0\le E_N(P)\le C_0h^5\log^3(2N),
 \qquad 0\le B_N(P)\le C_0h\log^3(2N),
\end{equation}
then microscopic rigidity gives, after the common threshold,
\begin{equation}\label{eq:curvature-coordinates-017}
 \max_j(|a_j-1|+|\eta_j|)\le\frac14.
\end{equation}

Finally put $v=\log(1+\rho)$ and, on each edge, let
$p=\theta-\varphi$ where $\varphi$ is the outward-normal angle.  Then
\begin{equation}\label{eq:curvature-coordinates-018}
 v'=\tan p,
 \qquad Dp=d\theta-2\pi\nu_P,
 \qquad \int_{\mathbb T}\tan p=0.
\end{equation}
Define
\begin{equation}\label{eq:curvature-coordinates-019}
 \widehat v_L(0)=0,
 \qquad \widehat v_L(k)=\frac{\widehat\nu_P(k)}{k^2}\quad(k\ne0),
\end{equation}
so $v_L'=p-\overline p$, while
\begin{equation}\label{eq:curvature-coordinates-020}
 |P|=\frac12\int_{\mathbb T}e^{2v},
 \qquad |P|=\pi\iff\frac1{2\pi}\int_{\mathbb T}e^{2v}=1.
\end{equation}
We continue to use $F(\Omega)=|\Omega|\lambda_1(\Omega)$.
\end{definition}

For a measure that is not centered, the full and truncated defects differ by the first mode:
\[
 E_N(\nu)=d_{3,N}(\nu)+|\widehat\nu(1)|^2.
\]
Under exact centering they coincide. The notation $d_{p,N}$ always uses the
$k\ge2$ convention of Definition~\ref{def:atomic-energies};
$E_N$ includes the first mode.

The following reconstruction identifies the polygon once its atomic measure is known. It uses only the turning equation and the area constraint.

\begin{lemma}
\label{lem:polar-equality-reconstruction}
Let $P$ be a convex polygon in the near-disk radial chart, translated to its
polar center and normalized by $|P|=\pi$.  If
\[
 \nu_P=\frac1N\sum_{j=0}^{N-1}\delta_{\gamma+jh},
\]
then $P$ is the area-$\pi$ regular $N$-gon, up to rotation.
\end{lemma}

\begin{proof}
The atomic identity determines the vertex directions and masses. After cyclic relabeling,
\begin{equation}\label{eq:spectral-rigidity-010}
 \psi_j=\gamma+jh,\qquad w_j=N^{-1},\qquad \omega_j=h.
\end{equation}
Let $v=\log r$ and denote the normal angle on edge $j$ by $\varphi_j$. The function $p(\alpha)=\alpha-\varphi_j$ has slope one and jumps $-h$, hence
\[
 Dp=d\alpha-h\sum_{j=0}^{N-1}\delta_{\gamma+jh},\qquad
 p(\alpha)=c+\alpha-\gamma-jh,
\]
with one constant $c$ for all cells. Periodicity of $v$ gives
\[
 0=\int_0^{2\pi}v'
  =N\int_0^h\tan(c+s)\,ds
  =N\log\frac{\cos c}{\cos(c+h)}.
\]
The range $p\in(-\pi/2,\pi/2)$ then requires $c=-h/2$, and integration of $v'=\tan p$ gives
\[
 r(\alpha)=R_j\sec\!\left(\alpha-\gamma-jh-\frac h2\right).
\]
At a common vertex the adjacent radial values agree, so all support numbers $R_j$ coincide. Consecutive normal angles differ by $h$, making the polygon regular. Finally,
\[
 |P|=NR^2\tan(\pi/N)=\pi,
\]
fixes the scale; rotation is the only freedom left.
\end{proof}

\subsection{The normalized spectral decomposition}

The linear turning potential reproduces the atomic quadratic energy exactly. Its difference from the true radial displacement creates a geometric correction, and the terms beyond the radial Hessian create a spectral correction. Keeping those two sources of error separate gives the endpoint identity on which both localization and rigidity will rely.

\begin{definition}
\label{def:disk-normalized-splitting}
Let $\mathbb D=B_1(0)$ and $F(\Omega)=|\Omega|\lambda_1(\Omega)$.
An admissible splitting datum consists of a real sequence
$(\beta_k)_{k\ge2}$ with $|\beta_k|\le Ck$, together with polar measures and
distinguished polar radial graphs for $\Omega=P,R_N$, such that
\[
 Q(f)=4\sum_{k=2}^\infty\beta_k|\widehat f(k)|^2,
 \qquad
 Q_N(\Omega)=4\sum_{k=2}^\infty
 \frac{\beta_k}{k^4}|\widehat\nu_\Omega(k)|^2,
\]
converge.  Define
\[
 D_N(\Omega)=Q(\rho_\Omega)-Q_N(\Omega),
\]
\begin{equation}\label{eq:curvature-coordinates-021}
 \widetilde{R}_N(\Omega)=
 \frac{F(\Omega)-F(\mathbb D)}{F(\mathbb D)}
 -Q(\rho_\Omega).
\end{equation}
Write $X^\circ(P)=X(P)-X(R_N)$. Each term is evaluated in its domain's own polar chart.

Subtracting these identities at the two endpoints gives
\begin{equation}\label{eq:curvature-coordinates-022}
\frac{F(P)-F(R_N)}{F(\mathbb D)}
=Q_N^\circ(P)+D_N^\circ(P)
 +\widetilde{R}_N^\circ(P).
\end{equation}
\end{definition}

\begin{definition}
\label{def:bessel-splitting}

For $N\ge3$ and $h=2\pi/N$, use the Fourier convention \eqref{eq:compactness-001} and the Bessel coefficients above to define
\[
 \sigma_N=\frac1N\sum_{j=0}^{N-1}\delta_{jh},\qquad
 Q_N(\nu)=4\sum_{k=2}^\infty
 \frac{\beta_k}{k^4}|\widehat\nu(k)|^2,
\]
\begin{equation}\label{eq:curvature-coordinates-023}
 Q_N^\circ(\nu)=Q_N(\nu)-Q_N(\sigma_N).
\end{equation}
With the same Fourier convention and $b_k=\beta_k/k^4$, we have the exact identity
\begin{equation}\label{eq:curvature-coordinates-024}
 Q_N^\circ(\nu)=4d_{B,N}(\nu),
 \qquad
 E_N(\nu)=d_{3,N}(\nu)
 \quad\text{when }\widehat\nu(1)=0.
\end{equation}
The bound $|\beta_k|\le Ck$ gives convergence for every probability measure. The two Lipschitz radial traces belong to the form domain of $Q$, and their respective polar representations satisfy
\begin{equation}\label{eq:curvature-coordinates-025}
 Q_N(P)=Q_N(\nu_P),\qquad
 Q_N(R_N)=Q_N(\sigma_N),\qquad
 Q(v_L)=Q_N(P).
\end{equation}
For real periodic functions define
\[
 K_{{\rm B}}(f,g)=4\operatorname{Re}\sum_{k=2}^\infty
 \frac{\beta_k}{k^2}\widehat f(k)\overline{\widehat g(k)}.
\]
Put
\[
 b=e^v\tan p-p,\qquad q=\rho-v_L.
\]
The identities $\rho'=e^v\tan p$ and $v_L'=p-\overline p$ give, for
$k\ge2$,
\[
 \widehat q(k)=\frac{\widehat{b}(k)}{ik},\qquad
 \widehat v_L(k)=\frac{\widehat p(k)}{ik}.
\]
Weighted Cauchy--Schwarz justifies polarization, and therefore
\begin{equation}\label{eq:curvature-coordinates-026}
 D_N(P)=2K_{{\rm B}}(p,b)
 +K_{{\rm B}}(b,b).
\end{equation}
\end{definition}

\subsection{Discrepancy and nonlinear reconstruction}\label{sec:localization}

Before the mesh has any regularity, the atomic defect must control quantities that enter the turning equation. Discrepancy bounds give the needed control of its piecewise-linear primitive. The reconstruction estimates then quantify the effect of replacing the linear turning potential by the actual area-normalized radial graph.

\begin{lemma}\label{lem:turning-discrepancy}

There are absolute constants \(C<\infty\) and \(N_0\) such that the
following holds. Let \(N\ge N_0\), put \(h=2\pi/N\) as in
Definition~\ref{def:polar-coordinates},
and let

\[
\nu=\sum_{j=1}^{M}w_j\delta_{\psi_j},
\qquad 1\le M\le N,
\qquad w_j>0,
\qquad \sum_jw_j=1,
\]

where coincident atoms have been merged. Assume \(\widehat\nu(1)=0\) in
the Fourier convention of \eqref{eq:compactness-001}, and set
\(d=d_{3,N}(\nu)\). Then

\[
\operatorname{Disc}(\nu)
\le C\left(h+d^{1/4}\right),
\]

where the supremum defining \(\operatorname{Disc}\) is taken over all
half-open oriented arcs \(I\) of length \(|I|\in[0,2\pi]\), with the
normalization

\[
\operatorname{Disc}(\nu)
:=
\sup_I\left|\nu(I)-\frac{|I|}{2\pi}\right|.
\]

If \(\nu=\nu_P\) for a polygon in the near-disk polar chart of
Definition~\ref{def:polar-coordinates} and
\(p\) is its geometric turning representative, then, with
\(\operatorname{osc}p:=\operatorname*{ess\,sup}p- \operatorname*{ess\,inf}p\),

\[
\|p\|_\infty
\le\operatorname{osc}p
\le C\left(h+d^{1/4}\right).
\]

\end{lemma}

\begin{proof}
The weighted discrepancy estimate follows first for rational masses by replication, and then for arbitrary positive masses by total-variation approximation:
\begin{equation}\label{eq:localization-001}
 \operatorname{Disc}(\nu)
 \le C\left(\frac1K+\sum_{k=1}^{K}\frac{|\widehat\nu(k)|}{k}\right).
\end{equation}
Atomic positivity gives $d\ge0$. For $K\le L=\lfloor N/2\rfloor$, centering and low-band coercivity yield
\begin{equation}\label{eq:localization-002}
 \sum_{k=1}^{K}\frac{|\widehat\nu(k)|}{k}
 \le\left(\sum_{k=2}^{K}\frac{|\widehat\nu(k)|^2}{k^3}\right)^{1/2}
     \left(\sum_{k=2}^{K}k\right)^{1/2}
 \le CKd^{1/2}.
\end{equation}
Set $\delta=d^{1/4}$. The trivial estimate covers $\delta\ge1/4$. Otherwise take $K=L$ when $\delta<L^{-1}$, and $K=\lfloor\delta^{-1}\rfloor$ when $\delta\ge L^{-1}$. In either case $K^{-1}+K\delta^2\le C(N^{-1}+\delta)$, so
\begin{equation}\label{eq:localization-003}
 \operatorname{Disc}(\nu)\le C(h+d^{1/4}).
\end{equation}

For the turning function, integrate over the lifted half-open arc $I=(x,y]$:
\[
 p(y)-p(x)=-2\pi\left(\nu_P(I)-\frac{|I|}{2\pi}\right).
\]
This gives $\operatorname{osc}p\le2\pi\operatorname{Disc}(\nu_P)$. The conditions $\int\tan p=0$ and $p\in(-\pi/2,\pi/2)$ put zero in the essential range, and therefore $\|p\|_\infty\le\operatorname{osc}p$.
\end{proof}

\begin{lemma}
\label{lem:endpoint-fejer-hilbert}
Let $(\nu,p,v_L)$ be centered near-regular polar data, put
$\eta=\|p\|_\infty$, and suppose that for some $2\le K\le N/10$,
\[
 \sum_{k=2}^{K}\frac{|\widehat\nu(k)|^2}{k^3}\le C_Ld,
 \qquad d\ge0.
\]
Then
\begin{equation}\label{eq:localization-004}
 \|v_L\|_\infty
 \le C\left(\eta\frac{\log(2K)}K
       +d^{1/2}\sqrt{\log(2K)}\right).
\end{equation}
For $K=\lfloor N/10\rfloor$, define, with modes $0,\pm1$ removed,
\[
 \widehat{M_{{\rm B}}f}(k)=\frac{4\beta_{|k|}}{k^2}\widehat f(k).
\]
Then
\begin{equation}\label{eq:localization-005}
 \|M_{{\rm B}}p\|_\infty
 \le C\left(\eta h+\|v_L\|_\infty\log(2N)\right),
\end{equation}
and, for every real $g\in L^1\cap L^2$,
\begin{equation}\label{eq:localization-006}
 |K_{{\rm B}}(p,g)|\le\frac12\|M_{{\rm B}}p\|_\infty\|g\|_1,
 \qquad |K_{{\rm B}}(g,g)|\le C\|g\|_2^2.
\end{equation}
The constants depend only on $C_L$.
\end{lemma}

\begin{proof}
Use the normalized Fej\'er kernel, whose bounds are
\[
 F_{K+1}^{\rm av}\ge0,\qquad \frac1{2\pi}\int F_{K+1}^{\rm av}=1,
 \qquad F_{K+1}^{\rm av}(t)\le C\min\{K+1,((K+1)t^2)^{-1}\},
\]
with $F_m^{\rm av}(\theta)=m^{-1}(\sin(m\theta/2)/\sin(\theta/2))^2$, and set $\Sigma_Kf=F_{K+1}^{\rm av}*f$. Its first absolute moment is $O(\log(2K)/K)$. Since $v_L'=p-\overline p$,
\[
 \|v_L-\Sigma_Kv_L\|_\infty\le C\eta\frac{\log(2K)}K.
\]
The retained part has no modes $0,\pm1$. Cauchy--Schwarz and low-band coercivity give
\[
 \|\Sigma_Kv_L\|_\infty
 \le C\sum_{k=2}^{K}\frac{|\widehat\nu(k)|}{k^2}
 \le Cd^{1/2}\sqrt{\log(2K)}.
\]
Adding proves \eqref{eq:localization-004}.

The identity $\widehat p(k)=ik\widehat v_L(k)$ and the decomposition $\beta_k=k+1-r_k$ give
\[
 \widehat{M_{{\rm B}}p}(k)
 =4i\operatorname {sgn}(k)\widehat v_L(k)
 +4i\operatorname {sgn}(k)\frac{1-r_{|k|}}{|k|}\widehat v_L(k).
\]
For the leading Hilbert-transform term, use
\[
 H_rv_L(x)=\frac1{2\pi}\int_{-\pi}^{\pi}
 [v_L(x-t)-v_L(x)]\frac{2r\sin t}{1-2r\cos t+r^2}\,dt.
\]
The kernel is bounded by $C/|t|$. Estimate the difference by its Lipschitz bound on $|t|\le h$ and by $2\|v_L\|_\infty$ elsewhere. The bound is uniform in the Abel parameter. A weak-$*$ limit, identified through the Fourier coefficients, yields
\[
 \|Hv_L\|_\infty
 \le C\left(\eta h+\|v_L\|_\infty\log(2N)\right).
\]
The remaining multiplier has a bounded sawtooth kernel for the $i/k$ term and a summable $O(k^{-2})$ error. This proves \eqref{eq:localization-005}. Fourier truncation gives $K_{{\rm B}}(p,g)=\tfrac12\overline{(M_{{\rm B}}p)g}$ for real arguments. This gives the first pairing estimate; Parseval and $|\beta_k|/k^2\le C$ give the second.
\end{proof}

\begin{lemma}
\label{lem:polar-transfer}
Let $X$ be an area-$\pi$ radial domain in the Bessel splitting above and in the chart of
Definition~\ref{def:polar-coordinates}, with polar variables
$(\nu,p,v,\rho,v_L)$, exact centering $\widehat\nu(1)=0$, and
$\eta=\|p\|_\infty$ smaller than a fixed absolute constant.  Put
$q=\rho-v_L$, $V_2=\|v_L\|_2$, and $V_\infty=\|v_L\|_\infty$.

\smallskip
\noindent\textup{\bfseries (i) Nonlinear reconstruction.}
There are unique $e,c$ with
\[
 v=v_L+e+c,\qquad \overline e=0,\qquad
 e'=\tan p-p+\overline p,
\]
and, whenever $V_\infty+\eta$ is sufficiently small,
\begin{equation}\label{eq:localization-007}
 |\overline p|+\|e\|_{W^{1,\infty}}+\|e\|_{H^1}\le C\eta^3,
 \qquad |c|\le C(V_2^2+\eta^6),
\end{equation}
\begin{equation}\label{eq:localization-008}
 \|v\|_2\le C(V_2+\eta^3),\qquad
 \|v\|_\infty\le C(V_\infty+\eta^3+V_2^2),
\end{equation}
\begin{equation}\label{eq:localization-009}
 \|q'\|_2\le C(\eta V_2+\eta^3),\qquad
 \|q'\|_\infty\le C(\eta\|v\|_\infty+\eta^3),
 \qquad \|q\|_\infty\le C(V_\infty^2+V_2^2+\eta^3),
\end{equation}
\begin{equation}\label{eq:localization-010}
 \|q\|_{H^{1/2}}^2
 \le C\|q\|_\infty(\|q\|_\infty+\|q'\|_2),\qquad
 \|b\|_1+\|b\|_2+\|b\|_\infty
 \le C(\eta\|v\|_\infty+\eta^3),
\end{equation}
where $b=e^v\tan p-p$.

\smallskip
\noindent\textup{\bfseries (ii) Quadratic and fixed-disk transfer.}
For $k\ge2$,
\begin{equation}\label{eq:localization-011}
 \frac k2\le\beta_k\le\frac{3k}{2},\qquad
 Q(q)\le C\|q'\|_2^2,
\end{equation}
and the exact polarization of $Q$ gives
\begin{equation}\label{eq:localization-012}
 |D_N(X)|
 \le C\left(Q(v_L)^{1/2}\|q'\|_2+\|q'\|_2^2\right).
\end{equation}
If $\rho\in W^{1,\infty}\cap H^{1/2}$ lies in the common fixed-disk
smallness ball and $U=\operatorname{Ext}(\rho e_r)$ is the simultaneous extension,
then
\begin{equation}\label{eq:localization-013}
 L(U)\le C\|\rho\|_{W^{1,\infty}},\qquad
 E(U)^2\le C\|\rho\|_{H^{1/2}}^2,
\end{equation}
\[
 (I+U)(\mathbb D)=X,\qquad
 \frac12D^2J_{{\rm sp}}(0)[U,U]=Q_B(\rho^\perp)=j^2Q(\rho),
\]
\begin{equation}\label{eq:localization-014}
 \widetilde{R}_N(X)
 =\frac1{j^2}\left(J_{{\rm sp}}(U)-J_{{\rm sp}}(0)
 -\frac12D^2J_{{\rm sp}}(0)[U,U]\right),
 \qquad
 |\widetilde{R}_N(X)|
 \le C\|\rho\|_{W^{1,\infty}}\|\rho\|_{H^{1/2}}^2.
\end{equation}

\smallskip
\noindent\textup{\bfseries (iii) Regular comparator.}
For the area-$\pi$ regular polygon, on $|s|\le h/2$,
\begin{equation}\label{eq:localization-015}
 p_R(s)=s,\qquad e^{v_R(s)}=s_N\sec s,\qquad
 v_{L,R}(s)=\frac{s^2}{2}-\frac{h^2}{24},\qquad
 s_N=\left(\frac{h/2}{\tan(h/2)}\right)^{1/2}.
\end{equation}
Consequently
\[
 \|p_R\|_\infty+\|\rho_R'\|_\infty\le Ch,
 \quad \|\rho_R\|_\infty+\|v_R\|_\infty+\|v_{L,R}\|_\infty\le Ch^2,
\]
\begin{equation}\label{eq:localization-016}
 \|\rho_R\|_{H^{1/2}}^2+Q(v_{L,R})\le Ch^3,
 \quad \|q_R'\|_2+\|b_R\|_1+\|b_R\|_2
 +\|b_R\|_\infty\le Ch^3,
\end{equation}
and, for the multiplier $M_{{\rm B}}$ in the endpoint estimate below,
\begin{equation}\label{eq:localization-017}
 \|M_{{\rm B}}p_R\|_\infty\le Ch^2,
 \qquad |D_N(R_N)|\le Ch^5,
 \qquad |\widetilde{R}_N(R_N)|\le Ch^4.
\end{equation}
All constants are independent of the smallest polar weight and gap.
\end{lemma}

\begin{proof}
The periodicity condition gives $\overline{\tan p}=0$, hence $\overline p=\overline{p-\tan p}$ and $|\overline p|\le C\eta^3$. The equation for $e'$ consequently has a unique mean-zero primitive. The area constraint determines the constant by
\[
 c=-\frac12\log\overline{e^{2(v_L+e)}}.
\]
The primitive bound for $e$, Taylor expansion at the mean-zero function $v_L+e$, and local Lipschitz continuity of the logarithm give \eqref{eq:localization-007}--\eqref{eq:localization-008}. For $q$, use the exact formulas
\[
 q'=(e^v-1)\tan p+(\tan p-p)+\overline p,
 \qquad q=(e^v-1-v)+e+c,
\]
to obtain \eqref{eq:localization-009}. Interpolate $\|q\|_{H^{1/2}}^2\le C\|q\|_2\|q\|_{H^1}$ and use $\|q\|_2\le C\|q\|_\infty$ for the first bound in \eqref{eq:localization-010}. The second follows from $b=(e^v-1)\tan p+(\tan p-p)$.

The disk coefficient bounds and Parseval give \eqref{eq:localization-011}. Polarization at $\rho=v_L+q$, with $Q(v_L)=Q_N(X)$, yields \eqref{eq:localization-012}. The simultaneous extension supplies \eqref{eq:localization-013}. Domain realization and the trace-independent Hessian identify \eqref{eq:localization-014}, whose estimate is the uniform disk remainder bound.

For the regular comparator, the supporting-line formula and area constraint give \eqref{eq:localization-015}. Expanding $\sec$ and $\tan$ and applying periodic interpolation gives \eqref{eq:localization-016}; the bounds for $q_R'$ also use $\overline p_R=0$. Only multiples of $N$ occur, with $|\widehat v_{L,R}(mN)|=(mN)^{-2}$. Thus $\beta_k\le3k/2$ gives
\[
 \|M_{{\rm B}}p_R\|_\infty
 \le C\sum_{m\ge1}(mN)^{-2}\le Ch^2
\]
The pairings in \eqref{eq:curvature-coordinates-026} yield $|D_N(R_N)|\le Ch^5$, and \eqref{eq:localization-014} gives the regular remainder estimate.
\end{proof}

\subsection{The first quantitative localization}

Minimality first bounds the total quadratic energy at the regular polygon scale. That estimate can be fed back into discrepancy and reconstruction to improve the defect itself. The distinction between total energy and defect is essential here: only the latter measures departure from the regular atomic configuration.

\begin{lemma}\label{lem:coarse-polar-energy}

There are absolute constants \(C<\infty\) and \(N_0\) with the following
property. Let \(N\ge N_0\), put \(h=2\pi/N\), and let \(P=P_N^*\) be any
global minimizer translated to its polar center in the sense of
Definition~\ref{def:polar-coordinates}.
Write

\[
\partial P=\{(1+\rho(\theta))e^{i\theta}:\theta\in\mathbb T\}.
\]

Use the Bessel-normalized variables from
Definition~\ref{def:bessel-splitting}
and
Definition~\ref{def:disk-normalized-splitting};
in particular, \(Q\) uses the coefficients \(\beta_k\) fixed by
Definition~\ref{def:atomic-energies}.
Then

\[
\|\rho\|_{\dot H^{1/2}}^2+Q(\rho)
+\sum_{k=2}^{\infty}\frac{|\widehat\nu_P(k)|^2}{k^3}
\le Ch^3.
\]

The constant is independent of the smallest exterior angle, polar gap,
or side length.

\end{lemma}

\begin{proof}
Let $\eta=\|p\|_\infty$. Geometric localization gives
\begin{equation}\label{eq:localization-018}
 \|\rho\|_\infty+\|v\|_\infty\le Ch,
 \qquad \eta+\|\rho'\|_\infty\le Ch^{1/2},
 \qquad ik\widehat p(k)=-\widehat\nu_P(k).
\end{equation}
Centering removes $\widehat p(\pm1)$, so $v'=\tan p$ bounds the corresponding modes of $v$ by $C\eta^3$. Combining $\|\rho-v\|_\infty\le Ch^2$ with the area identity $2\widehat\rho(0)+\widehat{\rho^2}(0)=0$ gives
\begin{equation}\label{eq:localization-019}
 |\widehat\rho(0)|^2+|\widehat\rho(1)|^2
 +|\widehat\rho(-1)|^2\le Ch^3.
\end{equation}
The two quadratic normalizations satisfy
\begin{equation}\label{eq:localization-020}
 Q_B(f^\perp)=j^2Q(f).
\end{equation}
Their coercivity and the low-mode estimate imply
\begin{equation}\label{eq:localization-021}
 \|\rho\|_{H^{1/2}}^2\le C(Q(\rho)+h^3).
\end{equation}
At $\|\rho\|_{W^{1,\infty}}\le Ch^{1/2}$, the disk remainder is bounded by
\begin{equation}\label{eq:localization-022}
 |\widetilde{R}_N(P)|
 \le Ch^{1/2}(Q(\rho)+h^3).
\end{equation}

The regular comparison has normalized spectral deficit at most $Ch^3$. Minimality and the exact decomposition therefore give
\[
 Q(\rho)
 \le Ch^3+Ch^{1/2}(Q(\rho)+h^3).
\]
For large $N$, absorption yields
\begin{equation}\label{eq:localization-023}
 \|\rho\|_{\dot H^{1/2}}^2+Q(\rho)\le Ch^3.
\end{equation}
To recover the atomic bound, use
\[
 q'=(e^v-1)\tan p+(\tan p-p)+\overline p,
 \qquad \|q'\|_2\le Ch^{3/2}.
\]
This gives $Q(q)\le Ch^3$. The Hilbert-seminorm triangle inequality bounds $Q(v_L)$ by $Ch^3$, and then
\[
 Q(v_L)=j^{-2}Q_B(v_L)
 \ge\|v_L\|_{\dot H^{1/2}}^2
 =2\sum_{k=2}^\infty\frac{|\widehat\nu_P(k)|^2}{k^3},
\]
gives the assertion. None of these estimates requires a lower bound for a mass, a gap, or an edge.
\end{proof}

\begin{proposition}\label{prop:critical-localization}

There are absolute constants \(C<\infty\) and \(N_0\) such that, for
every \(N\ge N_0\) and every global minimizer \(P\), translated to its
polar center and written in the near-disk chart of
Definition~\ref{def:polar-coordinates},
putting

\[
h=\frac{2\pi}{N},
\qquad
 d=E_N(P)=E_N(\nu_P)=d_{3,N}(\nu_P),
\]

one has

\[
\|v_L\|_2^2\le C(d+h^4),
\]

\[
|D_N^\circ(P)|
\le
C\left(h^{9/2}+h^{3/2}d^{3/4}\right),
\]

and

\[
d\le Ch^4,
\qquad
\|p\|_\infty\le Ch.
\]

\end{proposition}

\begin{proof}
Splitting at frequency $N/2$, use low-band coercivity below the cutoff and the coarse estimate above it. This gives
\begin{equation}\label{eq:localization-024}
 \|v_L\|_2^2\le C(d+h^4).
\end{equation}
The coarse bound is $0\le d\le Ch^3$, so discrepancy yields $\eta=\|p\|_\infty\le C(h+d^{1/4})\le Ch^{3/4}$. Mean-zero interpolation gives
\[
 \|v_L\|_\infty^2
 \le C\|v_L\|_2\|v_L'\|_2
 \le C(d+h^4)^{1/2}\eta=o(1).
\]
The nonlinear reconstruction estimates, followed by Young's inequality, yield
\begin{equation}\label{eq:localization-025}
 \|q'\|_2\le C\{\eta(d+h^4)^{1/2}+\eta^3\}
 \le C(h^3+d^{3/4}).
\end{equation}
Polarizing the quadratic form at both endpoints now gives
\begin{equation}\label{eq:localization-026}
 |D_N^\circ(P)|
 \le C\left(h^{9/2}+h^{3/2}d^{3/4}\right).
\end{equation}

By the coarse bound and \eqref{eq:localization-019}, $\|\rho\|_{H^{1/2}}^2\le Ch^3$. Area normalization forces $v$ to vanish at some point; integrate $v'=\tan p$ from that point to get $\|\rho\|_{W^{1,\infty}}\le C\eta$. Hence
\[
 |\widetilde{R}_N(P)|\le C\eta h^3,
 \qquad |\widetilde{R}_N(R_N)|\le Ch^4,
\]
and discrepancy implies
\begin{equation}\label{eq:localization-027}
 \widetilde{R}_N^\circ(P)
 \ge-C(h^4+h^3d^{1/4}).
\end{equation}
Inserting these estimates into the spectral decomposition and using minimality and Bessel coercivity gives
\[
 d\le C\left(h^4+h^3d^{1/4}
 +h^{9/2}+h^{3/2}d^{3/4}\right).
\]
Young's inequality absorbs both nonlinear powers of $d$. Thus $d\le Ch^4$, and the discrepancy estimate becomes $\|p\|_\infty\le Ch$.
\end{proof}

\section{The cubic coefficient in curvature variables}\label{sec:cubic-coefficient}

At the critical scale the cubic spectral contribution is as large as the defect that should determine the sign. Its structure must therefore be identified before it can be estimated effectively. We separate the cubic coefficient into a principal Fourier interaction and a lower-order term controlled by the boundary energy. The principal interaction is the one whose cyclic structure will recover the regular value in the next section. Smooth traces permit the coefficient calculation; the endpoint bounds then identify the same coefficient for the Lipschitz graphs supplied by the global reduction.

\subsection{The lower Bessel multiplier at the endpoint}

\begin{lemma}\label{lem:bessel-endpoint-trilinear}

Let \(j\) be the first positive zero of \(J_0\), as fixed in
Definition~\ref{def:fixed-disk},
and put \(x=j^2/4\). For \(n\ge1\), define

\begin{equation}\label{eq:cubic-coefficient-010}
A_n:=\sum_{\ell=0}^{\infty}
 \frac{(-x)^\ell}{\ell!(n+1)(n+2)\cdots(n+\ell)}.
\end{equation}

The product is understood to be one when \(\ell=0\). Then

\begin{equation}\label{eq:cubic-coefficient-011}
J_n(j)=\frac{(j/2)^n}{n!}A_n,
\qquad
q_n=q_{-n}:=\frac{J_n'(j)}{J_n(j)}\quad(n\ge1).
\end{equation}

With Fourier normalization

\[
\widehat f(n)=\frac1{2\pi}\int_0^{2\pi}
 f(\theta)e^{-in\theta}\,d\theta,
\]

let \(\Lambda_\theta=|D_\theta|\) be the multiplier \(|n|\). Define the even
multiplier \(S\) by

\[
\widehat{Sf}(n)=s_n\widehat f(n),\qquad
s_n=q_n-\frac{|n|}{j}\quad(|n|\ge2),
\]

and fix any finite values \(s_0\) and \(s_{-1}=s_1\) on the remaining
modes. The constants below may depend on \(K_0:=\max(|s_0|,|s_1|)\), but
not on the functions to which the estimates are applied. Then

\begin{equation}\label{eq:cubic-coefficient-012}
(Sf)(\theta)=\int_{\mathbb S^1}K_S(\theta-\phi)f(\phi)
 \frac{d\phi}{2\pi}
\quad\hbox{with a periodic }K_S\in L^1(\mathbb S^1),
\end{equation}

where the identity first holds for trigonometric polynomials and then
for every \(f\in L^\infty\) by normalized \(L^1\) convolution. Moreover,

\[
\|Sf\|_\infty\le C\|f\|_\infty,\qquad
\|Sf\|_{H^{1/2}}\le C\|f\|_{H^{1/2}}.
\]

Put \(\langle F\rangle=(2\pi)^{-1}\int_0^{2\pi}F\). If
\(f,g,h\in W^{1,\infty}(\mathbb S^1)\) have zero mean, then every fixed
linear combination \(\Lambda\) of the five generators

\[
\langle fgh\rangle,\quad
\langle (Sf)gh\rangle,\quad
\langle (Sf)(Sg)h\rangle,\quad
\langle (\Lambda_\theta f)gh\rangle,\quad
\langle (\Lambda_\theta f)(Sg)h\rangle,
\]

and their permutations satisfies

\[
|\Lambda(f,g,h)|
\le
C_\Lambda\sum_{\rm cyc}
\|f\|_\infty
\|g\|_{\dot H^{1/2}}\|h\|_{\dot H^{1/2}}.
\]

The two expressions containing \(\Lambda_\theta\) are understood as the
\(H^{-1/2}\)--\(H^{1/2}\) pairing. This pairing agrees with the
displayed integral for smooth functions and is obtained for Lipschitz
functions by smooth density.

\end{lemma}

\begin{proof}
The kernel estimate follows by separating the leading Bessel-ratio asymptotic from a summable error. Since $x=j^2/4<3/2$, the alternating series gives
\begin{equation}\label{eq:cubic-coefficient-013}
1-\frac{x}{n+1}\le A_n\le1,
\qquad A_n>0.
\end{equation}
The first harmonic can be treated directly: differentiating
\[
J_0(z)=\sum_{m=0}^{\infty}
 \frac{(-1)^m(z^2/4)^m}{(m!)^2},
\]
gives $J_1=-J_0'$ and
\begin{equation}\label{eq:cubic-coefficient-014}
J_1(z)=\frac z2\sum_{\ell=0}^{\infty}
 \frac{(-z^2/4)^\ell}{\ell!(\ell+1)!}
=\frac z2\sum_{\ell=0}^{\infty}
 \frac{(-z^2/4)^\ell}{\ell!\,2\cdot3\cdots(1+\ell)}.
\end{equation}
Combining this with
\[
jJ_n'(j)=nJ_n(j)-jJ_{n+1}(j),
\]
yields, for every $n\ge1$,
\begin{equation}\label{eq:cubic-coefficient-015}
s_n=-\frac{J_{n+1}(j)}{J_n(j)}
=-\frac{j}{2(n+1)}\frac{A_{n+1}}{A_n}.
\end{equation}
The terms of degree at least two have the bound
\[
\sum_{\ell\ge2}\frac{x^\ell}{\ell!(n+1)^\ell}
\le \frac{C}{(n+1)^2}.
\]
Thus
\[
A_n=1-\frac{x}{n+1}+O(n^{-2}),
\qquad
\frac{A_{n+1}}{A_n}=1+O(n^{-2}),
\]
and
\begin{equation}\label{eq:cubic-coefficient-016}
s_n=-\frac{j}{2n}+O(n^{-2})\qquad(n\ge2).
\end{equation}
After even extension and inclusion of the fixed low modes in the error, write
\begin{equation}\label{eq:cubic-coefficient-017}
s_n=-\frac{j}{2|n|}+e_n,
\qquad \sum_{n\in\mathbb Z}|e_n|<\infty.
\end{equation}
The leading part corresponds to
\[
 \sum_{n\ne0}\frac{e^{in\theta}}{|n|}
 =-2\log\!\left(2|\sin(\theta/2)|\right)
 \quad\text{in }L^1(\mathbb S^1),
\]
while the error defines a bounded kernel. Young's inequality now gives
\begin{equation}\label{eq:cubic-coefficient-018}
 \|Sf\|_\infty\le C\|f\|_\infty,
 \qquad \widehat{Sf}(n)=s_n\widehat f(n).
\end{equation}
The boundedness of the symbol also gives
\begin{equation}\label{eq:cubic-coefficient-019}
\|Sf\|_{H^{1/2}}^2
=\sum_n(1+|n|)|s_n|^2|\widehat f(n)|^2
\le C\|f\|_{H^{1/2}}^2,
\end{equation}
and its decay supplies the order-one smoothing estimate.

For the multilinear assertion, use
\begin{equation}\label{eq:cubic-coefficient-020}
 \|uv\|_{\dot H^{1/2}}
 \le C\bigl(\|u\|_\infty\|v\|_{\dot H^{1/2}}
          +\|v\|_\infty\|u\|_{\dot H^{1/2}}\bigr).
\end{equation}
Together with the bounds for $S$ and Poincar\'e's inequality on the mean-zero factors, this gives
\begin{equation}\label{eq:cubic-coefficient-021}
|\langle fgh\rangle|+|\langle(Sf)gh\rangle|
 +|\langle(Sf)(Sg)h\rangle|
\le C\sum_{\rm cyc}\|f\|_\infty
 \|g\|_{\dot H^{1/2}}\|h\|_{\dot H^{1/2}}.
\end{equation}
In the two remaining generators, regard $\Lambda_\theta f$ as an element of $H^{-1/2}$ and the product of the other factors as an element of $H^{1/2}$. The same product estimate yields
\begin{equation}\label{eq:cubic-coefficient-022}
|\langle(\Lambda_\theta f)gh\rangle|+|\langle(\Lambda_\theta f)(Sg)h\rangle|
\le C\sum_{\rm cyc}\|f\|_\infty
 \|g\|_{\dot H^{1/2}}\|h\|_{\dot H^{1/2}}.
\end{equation}
Smooth approximation gives the stated interpretation for Lipschitz arguments. The estimates are preserved under permutations and fixed linear combinations.
\end{proof}

\subsection{The principal cubic symbol and its Lipschitz extension}

The principal part of the cubic coefficient couples two positive frequencies with their sum. After the two integrations in the turning potential, this interaction becomes an absolutely convergent trilinear expression in the atomic Fourier coefficients. The endpoint argument below ensures that its identification does not require smoothing the polygonal corners.

\begin{lemma}\label{lem:bessel-cubic-symbol}

Let \(f\in W^{1,\infty}(\mathbb S^1;\mathbb R)\), with the Fourier
normalization of
Lemma~\ref{lem:bessel-endpoint-trilinear},
and suppose

\[
\widehat f(0)=\widehat f(1)=\widehat f(-1)=0,
\]

so that \(f\) lies in the similarity slice. Let \(\Pi_+\) be the
projection onto positive Fourier modes, let \(\Lambda_\theta=|D_\theta|\), and let
\(P_{{\rm Pois},\rho}\) be the Poisson multiplier
\(\widehat{P_{{\rm Pois},\rho} f}(n)=\rho^{|n|}\widehat f(n)\). For
\(0<\rho<1\) put

\begin{equation}\label{eq:cubic-coefficient-023}
P^{\mathrm{cub}}_\rho(f):=
\operatorname{Re}\sum_{a,b\ge1}
ab\,\rho^{a+b}\widehat f(a)\widehat f(b)
\widehat f(-(a+b)).
\end{equation}

This series is absolutely convergent for each \(\rho<1\), and its Abel limit satisfies

\begin{equation}\label{eq:cubic-coefficient-024}
P^{\mathrm{cub}}(f):=\lim_{\rho\uparrow1}P^{\mathrm{cub}}_\rho(f)
=\lim_{\rho\uparrow1}\operatorname{Re}\left\langle
 f\,[\Lambda_\theta\Pi_+(P_{{\rm Pois},\rho} f)]^2\right\rangle
=\operatorname{Re}\langle f(\Lambda_\theta\Pi_+f)^2\rangle.
\end{equation}

The last pairing is well-defined because \(\Lambda_\theta\Pi_+f\in L^2\). Thus the
double sum below means this Abel limit. If \(f\) is smooth, the
unsmoothed double series is absolutely convergent and has the same
value.

For the cubic functional \(C_3\) defined in
Lemma~\ref{lem:analytic-coefficient-chain},
one has

\[
C_3(f)
=
-8P^{\mathrm{cub}}(f)
+C_{{\rm low}}(f),
\]

where

\[
|C_{{\rm low}}(f)|
\le C\|f\|_\infty\|f\|_{\dot H^{1/2}}^2.
\]

More precisely, with the multipliers \(\Lambda_\theta,S\) of
Lemma~\ref{lem:bessel-endpoint-trilinear},
define for every real
\(f\in W^{1,\infty}(\mathbb S^1)\cap H^{1/2}(\mathbb S^1)\)

\begin{equation}\label{eq:cubic-coefficient-025}
\begin{aligned}
B_{\rm low}^{\rm cmp}(f):={}&
j\langle f^2\Lambda_\theta f\rangle
+2j^2\langle f(\Lambda_\theta f)(Sf)\rangle
+j^2\langle f^2Sf\rangle\\
&+j^3\langle f(Sf)^2\rangle
+\frac{j(1+j^2)}3\langle f^3\rangle ,
\qquad
 C_{{\rm low}}^{\rm cmp}(f):=-\frac2jB_{\rm low}^{\rm cmp}(f).
\end{aligned}
\end{equation}

On the similarity slice, \(C_{{\rm low}}^{\rm cmp}(f)=C_{{\rm low}}(f)\).
Thus \eqref{eq:cubic-coefficient-025} is a canonical ambient completion, and its symmetric
polarization is a fixed real linear combination of the five generators
in
Lemma~\ref{lem:bessel-endpoint-trilinear}
and their permutations. Since \(\Lambda_\theta\) and \(S\) are Fourier multipliers,
this completion is equivariant under rotations of the angular variable.

\end{lemma}

\begin{proof}
For a smooth real $f$ in the similarity slice, consider
\[
 \Omega_t=\{re^{i\theta}:0\le r<1+tf(\theta)\},
\]
and write $\lambda_1(\Omega_t)=k(t)^2$, where $k(0)=j$. Domain realization gives
\[
 J_{{\rm sp}}(tU_f)=\frac{|\Omega_t|}{\pi}k(t)^2.
\]
Since
\[
 \frac{|\Omega_t|}{\pi}=\langle(1+tf)^2\rangle
 =1+t^2\langle f^2\rangle,
\]
analyticity and stationarity imply
\begin{equation}\label{eq:cubic-coefficient-026}
 k(t)=j+k_2t^2+k_3t^3+O(t^4).
\end{equation}
The constant-mode calculation will use the identities
\begin{equation}\label{eq:cubic-coefficient-027}
 J_0''(j)=-\frac{J_0'(j)}j,
 \qquad
 J_0'''(j)=\frac{2-j^2}{j^2}J_0'(j),
\end{equation}
obtained from the Bessel equation at its first zero.

We identify the analytic Taylor coefficient through a formal boundary calculation, justified on the fixed disk. In $\mathbb R[t]/(t^4)$, take
\begin{equation}\label{eq:cubic-coefficient-028}
 v^{\le3}(t,r,\theta)=J_0(k(t)r)
 +\sum_{n\ne0}a_n(t)J_{|n|}(k(t)r)e^{in\theta},
 \qquad
 a_n(t)=\sum_{p=1}^3t^pa_n^{(p)}.
\end{equation}
If $g_p$ is the boundary term already determined at lower orders, cancellation of its nonconstant part and its average prescribes
\begin{equation}\label{eq:cubic-coefficient-029}
 a_n^{(p)}=-\frac{\widehat g_p(n)}{J_{|n|}(j)}\quad(n\ne0),
 \qquad
 k_p=-\frac{\widehat g_p(0)}{J_0'(j)}.
\end{equation}
At first order the mean condition is $jJ_0'(j)\widehat f(0)=0$, hence $k_1=0$. The Bessel bounds ensure nonzero denominators throughout this recursion.

The formal coefficients define genuine smooth spatial functions on the disk. Indeed,
\begin{equation}\label{eq:cubic-coefficient-030}
 \frac{J_n(jr)}{J_n(j)}=r^n\frac{A_n(xr^2)}{A_n(x)},
 \qquad
 \sup_{0\le r\le1}\left|
 \partial_r^q\frac{J_n(jr)}{J_n(j)}\right|
 \le C_q(1+n)^q,
 \quad 0\le q\le3.
\end{equation}
The series in \eqref{eq:cubic-coefficient-010} obeys $\sup_{0\le y\le x}|A_n^{(q)}(y)|\le e^{x/(n+1)}(n+1)^{-q}$, with $A_n(x)$ bounded below. Angular differentiation costs powers of $|n|$; differentiation in $k$ at $j$ reduces to radial differentiation. The convolutions and polynomially bounded multipliers in the recursion therefore preserve rapid Fourier decay. All derivatives needed through order three converge absolutely and uniformly, and
\begin{equation}\label{eq:cubic-coefficient-031}
 T_{\le3}\bigl[v^{\le3}(t,1+tf(\theta),\theta)\bigr]=0,
\end{equation}
holds coefficientwise.

Choose a smooth field $W$ with trace $fe_r$, extended across a collar and corrected in the interior to have zero mean. Then $\Phi_t=I+tW$ realizes $\Omega_t$ for small $t$. Pull back the formal solution and retain degrees at most three. Its coefficients have zero trace and lie in $H_0^1(\mathbb D)$. With $F_t=D\Phi_t$ and $\delta_t=\det F_t$, Piola's identity is
\[
 \operatorname{div}_x\!\left(\delta_tF_t^{-1}q\right)
 =\delta_t(\operatorname{div}_y q)\circ\Phi_t,
 \qquad q=(\nabla_yv^{\le3})\circ\Phi_t.
\]
The row divergences of $\operatorname{adj}D\Phi_t$ vanish, giving the weak equation
\begin{equation}\label{eq:cubic-coefficient-032}
 \int_{\mathbb D}\delta_tF_t^{-1}F_t^{-T}
 \nabla\widetilde v^{\le3}\cdot\nabla\varphi
 =k(t)^2\int_{\mathbb D}\delta_t\widetilde v^{\le3}\varphi
 \pmod {t^4},
\end{equation}
for $\varphi\in H_0^1(\mathbb D)$. Normalize by pairing with the fixed ground state $u_0$. At the first possible disagreement between formal coefficients $(w_p,\nu_p)$ and analytic coefficients $(u_p,\ell_p)$, one has
\[
 (A_0-j^2M_0)(w_p-u_p)=(\nu_p-\ell_p)M_0u_0.
\]
Pairing with $u_0$ identifies the eigenvalue coefficient; normalization and the saddle inverse identify the state coefficient. Induction through $p=3$ proves that the formal jet computes the analytic one, without requiring boundary regularity of the moving eigenfunction.

For the remaining scalar calculation set
\[
 q_n=\frac{J_{|n|}'(j)}{J_{|n|}(j)},
 \qquad
 \tau_n=\frac{J_{|n|}''(j)}{J_{|n|}(j)}.
\]
The first two nonconstant equations and the second mean equation are
\begin{equation}\label{eq:cubic-coefficient-033}
\begin{aligned}
 a_n^{(1)}&=-\frac{jJ_0'(j)}{J_{|n|}(j)}\widehat f(n),\\
 a_n^{(2)}&=\frac{jJ_0'(j)}{J_{|n|}(j)}
 \left\{\frac12\widehat{f^2}(n)
 +j\sum_{\ell\ne0}q_\ell\widehat f(\ell)
                  \widehat f(n-\ell)\right\},\\
 k_2&=\frac j2\widehat{f^2}(0)
 +j^2\sum_{n\ne0}q_n\widehat f(n)\widehat f(-n).
\end{aligned}
\end{equation}
At order three, the nonconstant coefficients do not enter the average equation. The $k_2f$ terms cancel because $J_0'(j)+jJ_0''(j)=0$, and the term with $k_2a_n^{(1)}$ has zero average. Consequently,
\[
\begin{aligned}
0={}&J_0'(j)k_3+\frac{j^3}{6}J_0'''(j)\widehat{f^3}(0)\\
&+j\sum_{n\ne0}a_n^{(2)}J_{|n|}'(j)\widehat f(-n)
 +\frac{j^2}{2}\sum_{n\ne0}a_n^{(1)}J_{|n|}''(j)
                    \widehat{f^2}(-n).
\end{aligned}
\]
Substitution gives $k_3=-B(f)$, with
\begin{equation}\label{eq:cubic-coefficient-034}
\begin{aligned}
B(f)={}&\frac{j(2-j^2)}6\widehat{f^3}(0)
 +\frac{j^2}{2}\sum_{n\ne0}q_n\widehat f(-n)
                         \widehat{f^2}(n)\\
&+j^3\sum_{\substack{n\ne0\\\ell\ne0}}
 q_nq_\ell\widehat f(-n)\widehat f(\ell)
                    \widehat f(n-\ell)
 -\frac{j^3}{2}\sum_{n\ne0}\tau_n\widehat f(n)
                         \widehat{f^2}(-n).
\end{aligned}
\end{equation}
All these series are absolutely convergent for smooth $f$. Reintroducing the exact area factor gives
\[
 J_{{\rm sp}}(tU_f)
 =j^2+t^2\bigl(2jk_2+j^2\langle f^2\rangle\bigr)
 +2jk_3t^3+O(t^4),
\]
and hence
\begin{equation}\label{eq:cubic-coefficient-035}
 C_3(f)=\frac{2k_3}{j}=-\frac2jB(f).
\end{equation}

To separate the principal order, let $T_q$ and $T_\tau$ have even symbols $q_n$ and $\tau_n$ on the nonzero modes. Parseval rewrites the preceding expression as
\begin{equation}\label{eq:cubic-coefficient-036}
 B(f)=\frac{j(2-j^2)}6\langle f^3\rangle
 +\frac{j^2}{2}\langle f^2T_q f\rangle
 +j^3\langle f(T_q f)^2\rangle
 -\frac{j^3}{2}\langle f^2T_\tau f\rangle.
\end{equation}
The Bessel equation and the definition of $S$ give
\[
 \tau_n=\frac{n^2}{j^2}-1-\frac{q_n}{j},
 \qquad q_n=\frac{|n|}{j}+s_n.
\]
Thus $T_q=\Lambda_\theta/j+S$ and $T_\tau=\Lambda_\theta^2/j^2-I-\Lambda_\theta/j^2-S/j$ on the similarity slice. Inserting these decompositions gives
\begin{equation}\label{eq:cubic-coefficient-037}
\begin{aligned}
B(f)={}&j\langle f(\Lambda_\theta f)^2\rangle
 -\frac j2\langle f^2\Lambda_\theta^2f\rangle\\
&+j\langle f^2\Lambda_\theta f\rangle
 +2j^2\langle f(\Lambda_\theta f)(Sf)\rangle
 +j^2\langle f^2Sf\rangle\\
&+j^3\langle f(Sf)^2\rangle
 +\frac{j(1+j^2)}3\langle f^3\rangle.
\end{aligned}
\end{equation}
The order-two contribution is its first line. On an ordered triple with sum zero, the symmetrized coefficient is
\[
 j\left\{
 \frac{|a||b|+|b||c|+|c||a|}{3}
 -\frac{a^2+b^2+c^2}{6}\right\}.
\]
For $a,b>0$ and $c=-(a+b)$ it equals $2jab/3$. Summing over the positions of the negative mode and over the conjugate sector yields
\begin{equation}\label{eq:cubic-coefficient-038}
 j\langle f(\Lambda_\theta f)^2\rangle-\frac j2\langle f^2\Lambda_\theta^2f\rangle
 =4j\operatorname{Re}\sum_{a,b\ge1}
 ab\widehat f(a)\widehat f(b)\widehat f(-(a+b)).
\end{equation}
The case $a=b$ has three distinct positions rather than six, but $(a,a)$ also occurs only once in the ordered double sum. The same formula therefore covers repeated indices.

The five other terms are the completed lower-order generators. The endpoint estimate gives
\begin{equation}\label{eq:cubic-coefficient-039}
 |B_{\rm low}^{\rm cmp}(f)|
 \le C\|f\|_\infty\|f\|_{\dot H^{1/2}}^2,
\end{equation}
so that
\begin{equation}\label{eq:cubic-coefficient-040}
 C_3(f)=-8P^{\mathrm{cub}}(f)
 -\frac2jB_{\rm low}^{\rm cmp}(f).
\end{equation}
On the slice, $C_{{\rm low}}=C_{{\rm low}}^{\rm cmp}=-2B_{\rm low}^{\rm cmp}/j$. Its polarization is a fixed linear combination of the polarized generators. Each multiplier and the normalized integral commute with rotation, so the completion is equivariant.

For Lipschitz traces, interpret the principal expression by Abel regularization:
\begin{equation}\label{eq:cubic-coefficient-041}
 P^{\mathrm{cub}}_\rho(f)=\operatorname{Re}\left\langle
 f[\Lambda_\theta\Pi_+(P_{{\rm Pois},\rho} f)]^2\right\rangle.
\end{equation}
For $0<\rho<1$ the series is absolutely convergent. As $\rho\uparrow1$, the differentiated positive-frequency factor converges in $L^2$, since $f\in H^1$. Its square converges in $L^1$, and multiplication by the bounded factor $f$ gives the Abel limit. On smooth traces this is the usual absolutely convergent sum.

Take periodic mollifications preserving modes $0,\pm1$:
\begin{equation}\label{eq:cubic-coefficient-042}
 f_m\to f\quad\hbox{uniformly and in }H^1,
 \qquad \sup_m\|f_m\|_{W^{1,\infty}}<\infty.
\end{equation}
The Abel representation gives $P^{\mathrm{cub}}(f_m)\to P^{\mathrm{cub}}(f)$. Polarizing the difference of the lower-order forms and using convergence in $L^\infty$ and $H^{1/2}$ gives
\begin{equation}\label{eq:cubic-coefficient-043}
 B_{\rm low}^{\rm cmp}(f_m)\longrightarrow
 B_{\rm low}^{\rm cmp}(f).
\end{equation}
For the analytic coefficient put $G_m=DU_{f_m}$ and $G=DU_f$. The extension satisfies
\begin{equation}\label{eq:cubic-coefficient-044}
 G_m\to G\quad\hbox{in }L^2,
 \qquad \sup_m\|G_m\|_\infty<\infty.
\end{equation}
For $r\le3$, the pulled matrices are finite polynomials in these gradients and the first two area integrals. They are uniformly bounded and converge in measure; the area integrals converge as well. If $w_m\to w$ in $H_0^1$, it follows that
\begin{equation}\label{eq:cubic-coefficient-045}
 A_r^{(m)}w_m\to A_rw,
 \qquad M_r^{(m)}w_m\to M_rw
 \quad\hbox{in }H^{-1}.
\end{equation}
For the gradient terms, decompose the difference as $P_r^{(m)}\nabla(w_m-w)+(P_r^{(m)}-P_r)\nabla w$. Boundedness controls the first term. For the second, split into $|P_r^{(m)}-P_r|\le\varepsilon$ and its complement, and use convergence in measure and absolute continuity of the $L^2$ integral. Let $m\to\infty$ and then $\varepsilon\downarrow0$. The mass terms are treated identically with functions in place of gradients.

The coefficient recurrence uses the same disk saddle inverse for every $m$. Applying it at orders one, two, and three gives
\begin{equation}\label{eq:cubic-coefficient-046}
 u_r^{(m)}\to u_r\quad\hbox{in }H_0^1,
 \qquad \lambda_r^{(m)}\to\lambda_r,
 \qquad r=1,2,3.
\end{equation}
Thus $C_3(f_m)=\lambda_3^{(m)}/j^2\to\lambda_3/j^2=C_3(f)$. Passing to the limit in \eqref{eq:cubic-coefficient-040} establishes the identity and its equivariant completion for Lipschitz traces.
\end{proof}

\section{Cyclic quadrature and the recovery of a regular mesh}\label{sec:cyclic-quadrature}

The estimate $E_N=O(N^{-4})$ leaves $N^4E_N$ bounded, while atomic mesh rigidity requires it to tend to zero. The missing gain comes from identifying the leading cubic value across the critical class. Frequencies near multiples of $N$ are organized through a unit-modulus phase; its three-factor correlation recovers the regular value without a lower gap bound. Once that common value is subtracted in the minimizing inequality, the defect becomes $o(N^{-4})$ and the atomic certificate yields a nearly regular mesh.

\subsection{Finite phase identities and their sampling bounds}

The sampling bounds below need control of local node multiplicity, not a minimum gap. Critical-scale data supply that weaker geometric information. Their Fourier columns can then be compared with a periodic unit-modulus phase. We distinguish two algebraic facts: multiplication of phases gives the unit three-column correlation, whereas cancellation of the first moments also uses the degree-zero condition stated in the lemma. The latter condition is explicit in the near-regular phase construction used for the relative comparison.

\begin{definition}\label{def:critical-ward}
Fix $C_0<\infty$.  A critical-scale datum consists of exactly $N$ distinct
cyclically ordered points $0\le x_1<\cdots<x_N<x_1+1$, positive weights
$w_j$ of total mass one, and
\[
 \nu_N=\sum_{j=1}^Nw_j\delta_{2\pi x_j},\qquad
 H_N(k)=\overline{\widehat\nu_N(k)},\qquad
 a_j=Nw_j,\qquad g_j=N(x_{j+1}-x_j).
\]
It also contains the real turning representative $p_N$ determined by
\[
 Dp_N=d\theta-2\pi\nu_N,\qquad
 \int_0^1\tan p_N(x)\,dx=0,
\]
with values in $(-\pi/2,\pi/2)$.  The critical hypotheses are
\begin{equation}\label{eq:cyclic-quadrature-001}
 H_N(1)=0,\qquad
 d_{3,N}(\nu_N)\le C_0N^{-4},\qquad
 \|p_N\|_\infty\le C_0N^{-1}.
\end{equation}
Put
\[
 I_N=\{-\lfloor(N-1)/2\rfloor,\ldots,\lfloor N/2\rfloor\},
 \qquad C_m(r)=H_N(mN+r),
\]
and
\begin{equation}\label{eq:cyclic-quadrature-002}
 T_N(\nu_N)
 =\sum_{a,b\ge1}
 \frac{\widehat\nu_N(a)\widehat\nu_N(b)
       \overline{\widehat\nu_N(a+b)}}{ab(a+b)^2}.
\end{equation}
The series is absolutely convergent, since grouping by $s=a+b$ gives
$2\sum_{s\ge2}H_{s-1}s^{-3}<\infty$.  The column series formed from $C_m(r)$ is the complex conjugate of
$T_N(\nu_N)$. Alias blocks and sector sums below carry the
canonical Fourier convention of \eqref{eq:cyclic-quadrature-002}; an
expression written directly in the conjugate columns is conjugated
accordingly. Real parts and absolute values are unchanged.
\end{definition}

\begin{lemma}\label{lem:cyclic-ward}
Let $\nu=\sum_jw_j\delta_{x_j}$ be a probability measure on $\Tunit$ with
distinct nodes.  Assume
\begin{equation}\label{eq:cyclic-quadrature-003}
 Nw_j\le A_0,
 \qquad
 \sup_{I:\,|I|=(4N)^{-1}}\#\{j:x_j\in I\}\le M_0.
\end{equation}
All constants below depend only on $A_0,M_0$.

\smallskip
\noindent\textup{\bfseries (i) Sampling and phase columns.}
For arbitrary $z_j$ and every $F\in H^1(\Tunit)$,
\begin{equation}\label{eq:cyclic-quadrature-004}
 \sum_{r\in I_N}\left|\sum_jw_jz_je^{2\pi irx_j}\right|^2
 \le C\sum_jw_j|z_j|^2,
 \qquad
 \sum_jw_j|F(x_j)|^2
 \le C\bigl(\|F\|_2^2+N^{-2}\|F'\|_2^2\bigr).
\end{equation}
Let $\varsigma\in\{-1,1\}$, let $f_1\in H^1(\Tunit;\mathbb S^1)$ have
degree zero, put $f_k=f_1^k$, and define
\[
 c_k^\varsigma(r)=\int_{\Tunit}f_k(x)e^{-2\pi i\varsigma rx}\,dx.
\]
Convolution by $c_k^\varsigma$ is unitary,
$c_m^\varsigma*c_n^\varsigma=c_{m+n}^\varsigma$, and
\begin{equation}\label{eq:cyclic-quadrature-005}
 \langle c_m^\varsigma*c_n^\varsigma,c_{m+n}^\varsigma\rangle=1,
 \qquad
 \left\|\frac rN c_k^\varsigma(r)\right\|_2
 =\frac{\|f_k'\|_2}{2\pi N}.
\end{equation}
The three first moments obtained by placing $r/N$, $s/N$, or $(r+s)/N$
in this correlation vanish.  If
$f_1(x_j)=e^{-2\pi i\varsigma Nx_j}$ and
\[
 C_k^\varsigma(r)=\int e^{-2\pi i\varsigma(kN+r)x}\,d\nu(x),
\]
then
\begin{equation}\label{eq:cyclic-quadrature-006}
 C_k^\varsigma(r)-c_k^\varsigma(r)
 =\int f_k(x)e^{-2\pi i\varsigma rx}\,d(\nu-dx)(x),
\end{equation}
and every finite Toeplitz block cut from $C_k^\varsigma$ has uniformly
bounded $2\to2$ norm; the corresponding $c_k^\varsigma$ block is a
contraction.

\smallskip
\noindent\textup{\bfseries (ii) Finite interpolation.}
Write
\[
 \rho(u)=C_0^\varsigma(u)-\delta_0(u),\quad
 R_2=\|\rho\|_{\ell^2(|u|\le3N/4)},\quad
 R_1(L)=\sum_{|u|\le L}|\rho(u)|,
 \quad \Lambda_k=(2\pi)^{-1}\|f_k'\|_2.
\]
For $1\le S\le N/8$,
\begin{equation}\label{eq:cyclic-quadrature-007}
 \|\Pi_{I_N}(C_k^\varsigma-c_k^\varsigma)\|_2
 \le C\left(\sqrt S\,R_2+\frac{\Lambda_k}{S}
                         +\frac{\Lambda_k}{N}\right).
\end{equation}
For $0\le R<K<N/8$, if $Q_{{\rm col},k}^R$ and $M_{{\rm ph},k}^R$ are the
actual and phase convolution blocks on $[-R,R]$, then
\begin{equation}\label{eq:cyclic-quadrature-008}
 \|Q_{{\rm col},k}^R-M_{{\rm ph},k}^R\|_{2\to2}
 \le C\left[R_1(K+R)+
 (\Lambda_k\sqrt{2R+1}+R)(K^{-1}+N^{-1})\right].
\end{equation}

\smallskip
\noindent\textup{\bfseries (iii) Exact cyclic return.}
Suppose $C_k=c_k+e_k$ on $I_N$, the phase identities above hold, all
actual Toeplitz blocks are uniformly bounded, and for some $d_*\ge0$
\begin{equation}\label{eq:cyclic-quadrature-009}
 \|(r/N)c_k(r)\|_2\le Ck d_*,\qquad
 \|N^{-1}f_k'\|_\infty\le Ck.
\end{equation}
Put $\varepsilon_k=\|e_k\|_{\ell^2(I_N)}$ and let
$T_k=\|\mathbf1_{\{|r|>N/8\}}c_k\|_2$.  There are nonnegative universal completion
weights $\beta_{k,M}$ with
\begin{equation}\label{eq:cyclic-quadrature-010}
 \beta_{k,M}\le C\frac{\log(2k)}{k^3},\qquad
 \sum_{k\le2M}k\beta_{k,M}\le C,
 \qquad
 \sum_{k\le2M}k^2\beta_{k,M}\le C(1+\log(2M))^2.
\end{equation}
Define
\begin{equation}\label{eq:cyclic-quadrature-011}
 E_M^{\rm col}=\sum_{k\le2M+2}\frac{\log(2k)}{k^3}\varepsilon_k^2,
 \qquad
 E_M^{\rm tail}=\sum_{k\le2M+2}\frac{\log(2k)}{k^3}T_k^2,
\end{equation}
\begin{equation}\label{eq:cyclic-quadrature-012}
 U_M=\sum_{k\le2M}\beta_{k,M}
 (\|\Pi_{I_N}C_k\|_2^2-\|\Pi_{I_N}c_k\|_2^2).
\end{equation}
For $r,s\in I_N$ let $t\in I_N$ and $\sigma\in\{-1,0,1\}$ be the
unique return $r+s=t+\sigma N$.  With the returned exact denominator,
let $V_M$ denote twice the real part of the $N^4$-normalized
sum over $m,n\le M$, after subtracting the regular constant
$[mn(m+n)^2]^{-1}$.  Then, for an absolute $A$,
\begin{equation}\label{eq:cyclic-quadrature-013}
 \ |
 V_M-U_M|
 \le C\log^A(2N)\left[
 d_*^2\log^2(2M)+d_*\log(2M)(E_M^{\rm col})^{1/2}
 +d_*(E_M^{\rm tail})^{1/2}+E_M^{\rm col}+E_M^{\rm tail}\right].
\end{equation}
Moreover
\begin{equation}\label{eq:cyclic-quadrature-014}
 N^4\sum_{m>M\ \mathrm{or}\ n>M}|T_{m,n,N}|
 \le C\log^A(2N)\frac{\log(2M)}{M^2}.
\end{equation}
For a fixed pair $m,n$, if the columns are restricted to $|r|\le R<N/8$
and the actual/phase block error is $\epsilon$, while the phase tails are
$\tau$, then
\begin{equation}\label{eq:cyclic-quadrature-015}
 \left|N^4T_{m,n,N}-\frac1{mn(m+n)^2}\right|
 \le \frac{C}{mn(m+n)^2}
 \bigl(\epsilon+\tau+R/N\bigr)\log(2N),
\end{equation}
up to the same bound on the complementary carry/tail chambers.

Finally, for the shifted return $t=r+q-\sigma N\in I_N$,
$2\le q\le\lfloor N/2\rfloor$, expand the exact correlation as
$H_{m,q}=P_{m,q}+L^{\rm in}_{m,q}+L^{\rm out}_{m,q}+Q_{m,q}$ according
to $C=c+e$.  Uniformly in $m$,
\begin{equation}\label{eq:cyclic-quadrature-016}
 \|P_{m,\cdot}\|_2
 \le C\log^2(2N)\min\{1,(m+1)d_*\},
\end{equation}
\begin{equation}\label{eq:cyclic-quadrature-017}
 \|L^{\rm in}_{m,\cdot}\|_2+\|L^{\rm out}_{m,\cdot}\|_2
 \le C\log^2(2N)(\varepsilon_m+\varepsilon_{m+1}),
 \quad
 \sup_q|Q_{m,q}|\le C(\varepsilon_m^2+\varepsilon_{m+1}^2),
\end{equation}
and
\begin{equation}\label{eq:cyclic-quadrature-018}
 \sup_q\sum_{m>M}m^{-3}|H_{m,q}|\le CM^{-2}.
\end{equation}
The no-carry chamber, both nonzero carries, and the even-$N$ Nyquist row
are disjoint and exhaustive; the latter is counted once by the half-open
choice of $I_N$.
\end{lemma}

\begin{proof}
The multiplicity hypothesis allows a partition of the nodes into finitely many $(4N)^{-1}$-separated families. On each family, the large sieve over $N$ consecutive frequencies and $Nw_j\le A_0$ give the first sampling estimate. For the second, average the one-dimensional Sobolev inequality on intervals of length $2/N$.

Fourier transformation turns convolution by $c_k^\varsigma$ into multiplication by $f_k$. Unimodularity gives unitarity, and $f_mf_n=f_{m+n}$ gives the unit correlation. The degree-zero condition gives its three vanishing first moments. Sampling the two trigonometric factors in the nodal phase identity bounds the finite Toeplitz blocks. To compare columns, truncate $f_k$ at frequency $S$. The low part contributes $C\sqrt S R_2$, while sampling and Bessel's inequality give $C(\Lambda_k/S+\Lambda_k/N)$ for the tail. This proves \eqref{eq:cyclic-quadrature-007}. Apply this argument to $f_k\phi$ for $\phi$ supported on $[-R,R]$. The low contribution becomes $R_1(K+R)$, and the derivative is bounded by $(2\pi)^{-1}\|(f_k\phi)'\|_2\le\Lambda_k\sqrt{2R+1}+R$. The compressed-block estimate follows.

For the cubic expression, retain the exact return $r+s=t+\sigma N$. On the no-carry chamber its normalized denominator expands as
\begin{equation}\label{eq:cyclic-quadrature-019}
 1-\frac r{mN}-\frac s{nN}-\frac{2t}{(m+n)N}
 +\sum_{\alpha<\beta}\frac{X_\alpha X_\beta}{N^2}
 A_{\alpha\beta},
\end{equation}
with uniformly $C^2$ coefficients. The first-moment identities cancel the three linear phase contributions. After $C=c+e$ is substituted, combine the one-error terms through
\begin{equation}\label{eq:cyclic-quadrature-020}
 2\operatorname {Re}\langle e_k,\Pi c_k\rangle
 =\|\Pi C_k\|_2^2-\|\Pi c_k\|_2^2-\|e_k\|_2^2.
\end{equation}
Their sum is the completed expression \eqref{eq:cyclic-quadrature-012}. Every remaining term contains two moments, one moment and one error, or two errors. Weighted Cauchy--Schwarz with \eqref{eq:cyclic-quadrature-010}--\eqref{eq:cyclic-quadrature-011} gives the five terms in \eqref{eq:cyclic-quadrature-013}.

On a nonzero-carry chamber, at least two of $|r|,|s|,|t|$ are $N/8$ or larger. Schur's test gives the same estimate with two tails, a tail and an error, or two errors. The masks have Fourier $\ell^1$ norm $O(\log N)$, and the rational symbols remain uniformly smooth on each chamber. Summing $C\log^A(2N)[mn(m+n)^2]^{-1}$ gives \eqref{eq:cyclic-quadrature-014}; retaining one alias pair gives \eqref{eq:cyclic-quadrature-015}.

For the shifted return, the phase correlation is an isometry in the shift variable and the masks cost $\log^2(2N)$. Subtract the constant denominator. In the no-carry chamber the difference contains an input or output moment; in a carry chamber use $|r-t|\ge N/2$. This gives \eqref{eq:cyclic-quadrature-016}. Replacing phase columns by errors gives \eqref{eq:cyclic-quadrature-017}, and the partial-bijection property of the return gives the $m^{-3}$ tail. The half-open residue set counts the even Nyquist coordinate once and makes the cases disjoint.
\end{proof}

\begin{proposition}
For every fixed $C_0$, critical-scale data with sufficiently large $N$ satisfy
\begin{equation}\label{eq:cyclic-quadrature-021}
 0<a_j\le C,
 \qquad 0<g_j\le C,
 \qquad \sum_j|a_j-1|^2\le C,
\end{equation}
and every interval of length $(4N)^{-1}$ contains at most $M_0(C_0)$ nodes.
There is a piecewise affine lift $\Phi_N$ such that
\begin{equation}\label{eq:cyclic-quadrature-022}
 e^{i\Phi_N(x_j)}=e^{2\pi iNx_j},\qquad
 f_m=e^{im\Phi_N},\qquad
 \|f_m'\|_2^2\le Cm^2N,
\end{equation}
and, for
\[
 c_m(r)=\int_{\Tunit}f_m(x)e^{2\pi irx}\,dx,
\]
we have $c_m*c_n=c_{m+n}$, $\|c_m\|_2=1$, and unit three-column
correlation.  The sampling estimates \eqref{eq:cyclic-quadrature-004} hold.  Moreover, for
$1\le R\le N/8$,
\begin{equation}\label{eq:cyclic-quadrature-023}
 \|C_m-c_m\|_{\ell^2(I_N)}
 \le C\left(\sqrt{R/N}+m\sqrt N/R+m/\sqrt N\right),
\end{equation}
and, if $R<K<N/8$,
\begin{equation}\label{eq:cyclic-quadrature-024}
 \|Q_{{\rm col},m,N}^R-M_{{\rm ph},m,N}^R\|_{2\to2}
 \le C\left[\frac{(K+R)^2}{N^2}
 +(m\sqrt{NR}+R)(K^{-1}+N^{-1})\right].
\end{equation}
All constants are independent of the smallest weight and gap.
\end{proposition}

\begin{proof}
The weight estimate gives $a_j\le C$ and $\sum_j|a_j-1|^2\le C$; the turning bound gives $g_j\le C$. Sampling does not require a positive minimum gap, only a bound on local multiplicity. To obtain it, use
\begin{equation}\label{eq:cyclic-quadrature-025}
 \frac12\int_0^\infty s^2G_N(e^{-s})\,ds=d_{3,N}(\nu_N),
\end{equation}
and retain adjacent pairs at $s=t/N$, $1\le t\le2$, whose normalized masses are at least $1/2$. Then
\begin{equation}\label{eq:cyclic-quadrature-026}
 G_N(e^{-t/N})\ge\frac cN
 \sum_{j,j+1\notin B}\left|\frac{\sin(\pi g_j)}{g_j}\right|^2,
 \qquad B=\{j:a_j<1/2\}.
\end{equation}
The exceptional set $B$ has bounded cardinality. Integration controls the sine sum; because $\sin(\pi g)\ge2g$ for $0<g<1/2$, the number of such short gaps is bounded. Hence every interval of length $(4N)^{-1}$ contains boundedly many nodes.

For each gap choose a nearest integer $n_j$ and write $\delta_j=g_j-n_j$. The same estimate gives
\begin{equation}\label{eq:cyclic-quadrature-027}
 \sum_j\frac{|\delta_j|^2}{g_j}\le C.
\end{equation}
Define the affine phase lift by $\Phi_N(x_{j+1})-\Phi_N(x_j)=2\pi\delta_j$. The sum $\sum_j\delta_j$ is an integer, so the exponential phase is periodic and has the prescribed nodal values $e^{2\pi iNx_j}$. Its derivative obeys
\[
 \|f_m'\|_2^2=4\pi^2m^2N\sum_j|\delta_j|^2/g_j\le Cm^2N.
\]
This bound, together with sampling, gives the phase estimates.

The critical defect bound and centering yield, by low-band coercivity,
\begin{equation}\label{eq:cyclic-quadrature-028}
 R_2:=\|H_N-\delta_0\|_{\ell^2(|u|\le3N/4)}\le CN^{-1/2},
 \qquad
 R_1(L)\le C L^2N^{-2}\quad(L\le3N/4).
\end{equation}
With $S=R$ and $\Lambda_m\le Cm\sqrt N$, \eqref{eq:cyclic-quadrature-007} gives \eqref{eq:cyclic-quadrature-023}. The same low-frequency bounds in \eqref{eq:cyclic-quadrature-008} give \eqref{eq:cyclic-quadrature-024}.
\end{proof}

\begin{theorem}
\label{thm:critical-cubic-ward}
\label{thm:critical-cubic-quadrature}
For every $C_0<\infty$ there are $C,N_0$ such that every critical-scale datum
with $N\ge N_0$ satisfies
\begin{equation}\label{eq:cyclic-quadrature-029}
 \left|N^4T_N(\nu_N)-\frac{\zeta(4)}2\right|
 \le CN^{-1/16}\log^2(2N).
\end{equation}
More precisely, with
\[
 R=N^{3/4}+O(1),\qquad K=N^{31/32}+O(1),\qquad M=N^{1/32}+O(1),
\]
the positive-alias sector differs from its regular value by
$O(N^{-1/16}\log^2N)$ after multiplication by $N^4$; the sector with one
central input is $O(N^{-1/16})$, and the sector with two central inputs is
$O(N^{-1/2})$ after the same normalization.
\end{theorem}

\begin{proof}
Decompose the two positive inputs into central frequencies and positive aliases. The low-frequency estimates are
\begin{equation}\label{eq:cyclic-quadrature-030}
 \sum_{2\le k\le3N/4}\frac{|H_N(k)|^2}{k^3}\le CN^{-4},
 \qquad
 \sum_{2\le k\le N/2}\frac{|H_N(k)|}{k}\le CN^{-1}.
\end{equation}
Choose $R=N^{3/4}+O(1)$, $K=N^{31/32}+O(1)$, and $M=N^{1/32}+O(1)$. Uniformly for $m\le3M$, phase approximation gives
\begin{equation}\label{eq:cyclic-quadrature-031}
 \|C_m-c_m\|_2+\|(I-\Pi_R)c_m\|_2\le CN^{-1/8},
 \qquad
 \|Q_{{\rm col},m,N}^R-M_{{\rm ph},m,N}^R\|\le CN^{-1/16}.
\end{equation}
The unit phase correlation and $R/N=N^{-1/4}$ imply, for $m,n\le M$,
\begin{equation}\label{eq:cyclic-quadrature-032}
 N^4T_{m,n,N}
 =\frac{1+O(N^{-1/16}\log(2N))}{mn(m+n)^2},
\end{equation}
with both carry chambers included. The unrestricted alias blocks satisfy
\begin{equation}\label{eq:cyclic-quadrature-033}
 N^4|T_{m,n,N}|
 \le \frac{C\log(2N)}{mn(m+n)^2},
\end{equation}
by the one-mask Schur estimate and the Toeplitz bound. Since the tail of $[mn(m+n)^2]^{-1}$ outside $m,n\le M$ is at most $C\log(2M)M^{-2}$, this proves the positive-alias estimate.

For exactly one central input, $2\le k\le N/2$, the conjugate-column identity is
\begin{equation}\label{eq:cyclic-quadrature-034}
 \overline{T_N^{1\mathrm c}}
 =\frac2{N^3}\sum_k\frac{H_N(k)}k
   \sum_{m\ge1}\frac{H_{m,k}}{m^3}.
\end{equation}
Here $d_*\le CN^{-1/2}$. Split the aliases at $m=M$ and use the shifted-return and tail estimates to get
\begin{equation}\label{eq:cyclic-quadrature-035}
 N^4|T_N^{1\mathrm c}|
 \le C\{N^{-1/8}\log^2(2N)+M^{-2}\}
 \le CN^{-1/16}.
\end{equation}

For two central inputs put $x_k=|H_N(k)|k^{-3/2}$, so that $\|x\|_2^2\le CN^{-4}$. On $a+b=d\le\alpha N$,
\[
 \sum_{a+b=d}\frac{|H_N(a)H_N(b)|}{ab}
 \le \frac d2\|x\|_2^2.
\]
The second low-band bound makes this contribution $O(N^{-5})$. At output $d=N+r$, the bound is $CN^{-5}|C_1(r)|$. Sampling gives $\sum_{r\in I_N}|C_1(r)|\le C\sqrt N$, and hence
\begin{equation}\label{eq:cyclic-quadrature-036}
 |T_N^{2\mathrm c}|\le CN^{-9/2}.
\end{equation}

These sectors exhaust the series. Its regular value is fixed by
\begin{equation}\label{eq:cyclic-quadrature-037}
 \sum_{m,n\ge1}\frac1{mn(m+n)^2}=\frac{\zeta(4)}2,
\end{equation}
obtained from the decompositions of $\zeta(2)^2$ using $\sum[m^2(m+n)^2]^{-1}$ and the diagonal. Summing the three estimates proves the result.
\end{proof}

\subsection{From cubic quadrature to lattice localization}\label{sec:localization-cubic}

The quadrature result must still be connected to the actual spectral remainder. We first justify substituting the linear turning potential into the cubic coefficient at the accuracy available in the critical class. The resulting normal form has the same leading term at a minimizer and at the regular polygon. Their subtraction is what improves the defect and makes atomic rigidity applicable.

\begin{lemma}
\label{lem:unified-fixed-disk-substitution}
Let $P$ be a polygon of area $\pi$ in the centered polar chart, with radial
function $\rho$, polar measure $\nu$, linear potential $v_L$, and
$q=\rho-v_L$.  Assume $\widehat\nu(1)=0$, the Bessel quadratic
normalization
\begin{equation}\label{eq:localization-028}
 Q(f)=j^{-2}Q_B(f^\perp),
\end{equation}
and the smallness required by the analytic coefficient chain for the extension of $\rho e_r$.  Put
\begin{equation}\label{eq:localization-029}
 A=\|v_L\|_{W^{1,\infty}},\quad B=\|v_L\|_{H^{1/2}},\quad
 E=\|q\|_{W^{1,\infty}},\quad F=\|q\|_{H^{1/2}}.
\end{equation}
Then
\begin{equation}\label{eq:localization-030}
 \left|\widetilde{R}_N(P)
       +8\operatorname{Re}T_N(\nu)\right|
 \le C\,S_{\rm sub}(\rho,v_L,q),
\end{equation}
where
\begin{equation}\label{eq:localization-031}
\begin{aligned}
 S_{\rm sub}(\rho,v_L,q):={}&ABF+EB^2+AF^2+EBF+EF^2\\
 &+\|v_L\|_\infty B^2
  +\|\rho\|_{W^{1,\infty}}^2\|\rho\|_{H^{1/2}}^2.
\end{aligned}
\end{equation}
The same estimate holds for the regular polygon.  In that case
\begin{equation}\label{eq:localization-032}
 S_{\rm sub}(\rho_R,v_{L,R},q_R)\le Ch^5.
\end{equation}
All constants are independent of the smallest weight and gap.
\end{lemma}

\begin{proof}
Start from the exact disk expansion
\begin{equation}\label{eq:localization-033}
 \widetilde{R}_N(P)
 =C_3(\rho)+R_{\ge4}(\rho).
\end{equation}
For the linear potential, centering removes modes $0,\pm1$ and gives
\begin{equation}\label{eq:localization-034}
 ab\,\widehat v_L(a)\widehat v_L(b)\widehat v_L(-(a+b))
 =\frac{\widehat\nu(a)\widehat\nu(b)
        \overline{\widehat\nu(a+b)}}{ab(a+b)^2}.
\end{equation}
The series on the right is absolutely convergent, so it identifies the Abel-defined principal term. Thus
\begin{equation}\label{eq:localization-035}
 C_3(v_L)
 =-8\operatorname{Re}T_N(\nu)+C_{{\rm low}}(v_L),
 \qquad
 |C_{{\rm low}}(v_L)|\le C\|v_L\|_\infty B^2.
\end{equation}
The difference at $\rho=v_L+q$ is bounded by the polarized cubic estimate:
\begin{equation}\label{eq:localization-036}
 |C_3(\rho)-C_3(v_L)|
 \le C(ABF+EB^2+AF^2+EBF+EF^2).
\end{equation}
Adding the quartic bound gives the last term of \eqref{eq:localization-031} and proves the general substitution estimate.

On a regular cell $|s|\le h/2$,
\[
 p_R(s)=s,\qquad v_{L,R}(s)=\frac{s^2}{2}-\frac{h^2}{24}.
\]
Periodic estimates yield
\[
 \|v_{L,R}\|_\infty\le Ch^2,\quad A_R\le Ch,\quad
 B_R\le Ch^{3/2}.
\]
The formula $q_R'=\rho_R\tan p_R+(\tan p_R-p_R)$ and the area constraint give
\[
 E_R\le Ch^3,\qquad F_R\le Ch^{7/2},\qquad
 \|\rho_R\|_{W^{1,\infty}}\le Ch,\qquad
 \|\rho_R\|_{H^{1/2}}\le Ch^{3/2}.
\]
Substitution in \eqref{eq:localization-031} bounds every term by $Ch^5$. The constants are independent of the smallest mass and gap.
\end{proof}

\begin{theorem}\label{thm:absolute-cubic-normal-form}

There are absolute constants \(C<\infty\) and \(N_0\) such that, for
every \(N\ge N_0\) and every global minimizer \(P_N\) with polar measure
\(\nu_N\),

\[
\left|
\widetilde{R}_N(P_N)
+8\operatorname{Re}T_N(\nu_N)
\right|
\le CN^{-9/2}.
\]

For the regular polygon one has

\[
\left|
\widetilde{R}_N(R_N)
+\frac{4\zeta(4)}{N^4}
\right|
\le CN^{-5}.
\]

The same \(C\) works for all global minimizers and is independent of the
smallest atomic weight and polar gap.

Here every minimizer is translated to its polar center and is written in
the distinguished polar chart of
Definition~\ref{def:polar-coordinates} and
Definition~\ref{def:disk-normalized-splitting}.
For any polar probability measure \(\nu\) (also when it has fewer than
\(N\) atoms), the global Fourier convention gives

\[
T_N(\nu):=
\sum_{a,b\ge1}
\frac{\widehat\nu(a)\widehat\nu(b)
      \overline{\widehat\nu(a+b)}}
     {ab(a+b)^2}.
\]

The conjugate phase-column series used in the finite cyclic calculation is the
complex conjugate of this one and therefore has the same real part.
This series is absolutely convergent, since \(|\widehat\nu(k)|\le1\) and
\((a+b)^2\ge4ab\).

\end{theorem}

\begin{proof}
The critical localization gives
\begin{equation}\label{eq:localization-037}
 \|v_L\|_2\le Ch^2,\qquad \|p\|_\infty\le Ch,
 \qquad v_L'=p-\overline p.
\end{equation}
Together with centering, periodic interpolation gives
\begin{equation}\label{eq:localization-038}
 \|v_L\|_\infty\le Ch^{3/2},\qquad
 \|v_L\|_{H^{1/2}}\le Ch^{3/2},\qquad
 \|v_L\|_{W^{1,\infty}}\le Ch.
\end{equation}
Write $v=v_L+e+c$, $\overline e=0$. The identities
\[
 e'=\tan p-p+\overline p,
 \qquad \overline p=\overline{p-\tan p},
\]
and $|\tan t-t|\le C|t|^3$ imply $\|e\|_{W^{1,\infty}}+\|e\|_{H^1}\le Ch^3$. Area normalization gives
\begin{equation}\label{eq:localization-039}
 e^{2c}\,\overline{e^{2(v_L+e)}}=1,
 \qquad |c|\le C\|v_L+e\|_2^2\le Ch^4.
\end{equation}
For $\rho=e^v-1$ and $q=\rho-v_L$, the resulting bounds are
\begin{equation}\label{eq:localization-040}
 \begin{array}{c|ccc}
 &W^{1,\infty}&H^{1/2}&L^\infty\\ \hline
 v_L&Ch&Ch^{3/2}&Ch^{3/2}\\
 \rho&Ch&Ch^{3/2}&Ch^{3/2}\\
 q&Ch^{5/2}&Ch^{11/4}&Ch^3
 \end{array}.
\end{equation}
The last row follows by estimating
\[
 q=(e^v-1-v)+e+c,
 \qquad q'=(e^v-1)\tan p+(\tan p-p)+\overline p,
\]
in $L^2$, $L^\infty$, and $H^1$, then interpolating. The extension gradient is at most $Ch$, so the analytic expansion applies. With $j^{-2}Q_B=Q$, the substitution estimate yields
\begin{equation}\label{eq:localization-041}
 \left|\widetilde{R}_N(P_N)
       +8\operatorname{Re}T_N(\nu_N)\right|
 \le Ch^{9/2}\le CN^{-9/2},
\end{equation}
without a minimum mass or gap hypothesis.

For the regular polygon the sharper substitution bound gives
\begin{equation}\label{eq:localization-042}
 \left|\widetilde{R}_N(R_N)
       +8\operatorname{Re}T_N(\nu_{R_N})\right|
 \le Ch^5.
\end{equation}
After rotation, its Fourier coefficients are $\mathbf1_{\{N\mid k\}}$. Therefore
\begin{equation}\label{eq:localization-043}
 T_N(\nu_{R_N})
 =\frac1{N^4}\sum_{m,n\ge1}\frac1{mn(m+n)^2}
 =\frac{\zeta(4)}{2N^4}.
\end{equation}
To evaluate the double sum $I$, set $U=\sum_{m,n\ge1}[m^2(m+n)^2]^{-1}$. Partial fractions and the strict-half-plane decomposition give $\zeta(2)^2=2U+2I=\zeta(4)+2U$, so $I=\zeta(4)/2$.
\end{proof}

\begin{corollary}\label{cor:first-localization}

For all sufficiently large \(N\), every global minimizer satisfies

\[
E_N(P)
\le
CN^{-4-1/16}\bigl(\log(2N)\bigr)^2,
\qquad
B_N(\nu_P)
\le
CN^{-1/16}\bigl(\log(2N)\bigr)^2.
\]

In particular, there is a numerical sequence
\(\varepsilon_N\downarrow0\) such that every global minimizer with
sufficiently large \(N\) has exactly \(N\) atoms and, in the cyclic
coordinates of
Theorem~\ref{thm:microscopic-rigidity},

\[
\max_j\bigl(|Nw_j-1|+|\eta_j|\bigr)\le\varepsilon_N.
\]

\end{corollary}

\begin{proof}
By activation and polar centering, a minimizer has exactly $N$ distinct nodes with positive weights, total mass one, and zero first moment. It also satisfies
\[
 d_{3,N}(\nu_P)\le CN^{-4},\qquad \|p\|_\infty\le CN^{-1}.
\]
The cubic quadrature theorem applies, giving
\begin{equation}\label{eq:localization-044}
 T_N(\nu_P)
 =\frac{\zeta(4)}{2N^4}
 +O\!\left(N^{-4-1/16}\log^2(2N)\right).
\end{equation}
The two absolute spectral expansions have the same leading term $-4\zeta(4)N^{-4}$. Their difference is therefore
\[
 \widetilde{R}_N^\circ(P)
 =O\!\left(N^{-4-1/16}\log^2(2N)\right).
\]
At $E_N(P)\le Ch^4$, Proposition~\ref{prop:critical-localization} gives a geometric error $O(N^{-9/2})$. Minimality and Bessel coercivity now imply
\[
 0\ge\frac14E_N(P)
 -CN^{-4-1/16}\log^2(2N).
\]
This bounds $E_N$; multiplication by $N^4$ bounds $B_N$.

Thus $B_N\to0$ uniformly over minimizers. The microscopic theorem makes the normalized masses and gaps converge to one. Any failure of uniformity would give a sequence contradicting that theorem. A decreasing tail supremum gives $\varepsilon_N\downarrow0$.
\end{proof}

\section{From absolute localization to relative cancellation}\label{sec:relative-cubic}

With the mesh now controlled, the cubic term can be estimated relative to $E_N$ itself. This requires combining the frequency columns before taking absolute values. The apparent linear phase errors become squared differences, and each weight residual is paired with a phase error. These quadratic expressions preserve the defect factor. The resulting relative bound, together with sharper absolute estimates for the other terms, gives the refined localization used in the material comparison.

\subsection{Sharper absolute remainder estimates}

The first localization gives additional low-frequency control. Using it in the reconstruction and cubic substitution improves the absolute errors. These estimates will be paired with the relative principal-cubic bound to reach the sharper defect scale needed for the geometric transport.

\begin{proposition}\label{prop:absolute-geometric-error}

Let \(P=P_N^*\) be a global minimizer and \(R_N\) the regular competitor. For all sufficiently large \(N\),

\[
 |D_N(P)|+|D_N(R_N)|
 \le Ch^5\bigl(\log(2N)\bigr)^3.
\]

Consequently,

\[
|D_N^\circ(P)|
 \le  Ch^5\bigl(\log(2N)\bigr)^3.
\]

\end{proposition}

\begin{proof}
With $d=E_N(P)$ and $\eta=\|p\|_\infty$, the first localization gives
\begin{equation}\label{eq:relative-cubic-016}
 d\le CN^{-4-1/16}\log^2(2N),\qquad \eta\le Ch.
\end{equation}
At $K=\lfloor N/10\rfloor$, the low-mode term in the Fej\'er bound is $o(h^2)$. The Hilbert estimate then gives
\begin{equation}\label{eq:relative-cubic-017}
 \|v_L\|_\infty\le Ch^2\log(2N),
 \qquad \|M_{{\rm B}}p\|_\infty\le Ch^2\log^2(2N).
\end{equation}
Combine this with $\|v_L\|_2\le Ch^2$ in radial reconstruction:
\begin{equation}\label{eq:relative-cubic-018}
 \|v\|_\infty\le Ch^2\log(2N),
 \qquad
 \|b\|_1+\|b\|_2+\|b\|_\infty
 \le Ch^3\log(2N).
\end{equation}
The exact correction and its pairing estimates yield
\[
 |D_N(P)|
 \le C\|M_{{\rm B}}p\|_\infty\|b\|_1
   +C\|b\|_2^2
 \le Ch^5\log^3(2N).
\]
Use \eqref{eq:localization-017} for the regular endpoint, and apply the triangle inequality for the difference.
\end{proof}

\begin{lemma}\label{lem:refined-cubic-error}

For every sufficiently large \(N\) and every global minimizer \(P\),

\begin{equation}\label{eq:relative-cubic-019}
\left|
\widetilde{R}_N^\circ(P)
+8\operatorname{Re}\!\left[
T_N(\nu_P)-T_N(\nu_{R_N})
\right]
\right|
\le Ch^5\log(2N).
\end{equation}

The constant is uniform in the smallest polar weight and gap.

\end{lemma}

\begin{proof}
The localization available here is
\[
 d\le CN^{-4-1/16}\log^2(2N),\qquad \|p\|_\infty\le Ch.
\]
The bounds $\|v_L\|_2\le Ch^2$ and $\|v_L\|_\infty\le Ch^2\log(2N)$ from Fej\'er--Hilbert control give, through reconstruction,
\begin{equation}\label{eq:relative-cubic-020}
 \begin{array}{c|ccc}
 &W^{1,\infty}&H^{1/2}&L^\infty\\ \hline
 v_L&Ch&Ch^{3/2}&Ch^2\log(2N)\\
 \rho&Ch&Ch^{3/2}&Ch^2\log(2N)\\
 q&Ch^3\log(2N)&Ch^3\log(2N)&Ch^3\log(2N)
 \end{array}.
\end{equation}
The substitution parameters therefore satisfy $A\le Ch$, $B\le Ch^{3/2}$, and $E+F\le Ch^3\log(2N)$. Hence
\[
 \left|\widetilde{R}_N(P)
       +8\operatorname{Re}T_N(\nu_P)\right|
 \le Ch^5\log(2N).
\]
Subtract the regular-endpoint formula, whose error is at most $Ch^5$, to obtain \eqref{eq:relative-cubic-019}. These bounds remain uniform in the smallest polar mass and gap.
\end{proof}

\subsection{Comparison on a defect annulus}

We compare the principal cubic term on a defect annulus, keeping the lower cutoff explicit. In addition to the finite cyclic identities, the proof needs a quantitative estimate for their completed linear contribution. This is obtained by separating the composite-trapezoid weights from the weight residual: the first part is a quadrature error of squared phase differences, and the second pairs a weight residual with a phase error. The exact constraint chart will later remove the cutoff.

\begin{proposition}
\label{prop:principal-cubic-relative}
Fix $0<b<3/8$ and $C_b<\infty$.  For each $N$, let
\[
 \nu_N=\sum_{j=1}^Nw_j\delta_{\psi_j},\qquad
 a_j=Nw_j,
 \qquad
 \psi_{j+1}-\psi_j=\frac{2\pi}{N}(1+\eta_j),
\]
have exactly $N$ distinct cyclically ordered nodes, positive weights of total
mass one, and $\widehat\nu_N(1)=0$.  Assume
\begin{equation}\label{eq:relative-cubic-001}
 \max_j(|a_j-1|+|\eta_j|)\le\frac14,
\end{equation}
and, with $B_N=N^4d_{3,N}(\nu_N)$,
\begin{equation}\label{eq:relative-cubic-002}
 e^{-N^{1/3}}\le B_N
 \le C_bN^{-b}\bigl(\log(2N)\bigr)^{C_b}.
\end{equation}
Then, for
$\sigma_N=N^{-1}\sum_{j=0}^{N-1}\delta_{2\pi j/N}$, there is a numerical
sequence $\epsilon_N=\epsilon_N(b,C_b)\to0$ such that
\begin{equation}\label{eq:relative-cubic-003}
 N^4\left|\operatorname {Re}\bigl[
 T_N(\nu_N)-T_N(\sigma_N)
 \bigr]\right|\le\epsilon_N B_N.
\end{equation}
The estimate is uniform in the individual weights and gaps and uses no
minimality, Euler equation, or turning representative.
\end{proposition}

\begin{proof}
Write $x_j=\psi_j/(2\pi)$ and $H_N(k)=\overline{\widehat\nu_N(k)}$. The assumed upper bound implies $B_N\to0$, so microscopic rigidity gives
\begin{equation}\label{eq:relative-cubic-004}
 \sum_j\bigl((a_j-1)^2+\eta_j^2\bigr)\le CB_N,
 \qquad
 \max_j(|a_j-1|+|\eta_j|)\le C\sqrt{B_N},
\end{equation}
uniformly; otherwise a violating sequence would contradict its sequential version. The linear column terms must be combined before estimation to preserve this factor of defect.

Choose the degree-zero phase lift by
\[
 \Phi_N(x_j)=2\pi Nx_j\pmod{2\pi},\qquad
 \Phi_N'=2\pi N\frac{\eta_j}{1+\eta_j},
\]
and set
\[
 f_k=e^{ik\Phi_N},\quad
 c_k(r)=\int f_k(x)e^{2\pi irx}\,dx,\quad
 C_k(r)=H_N(kN+r),\quad
 e_k=\Pi_{I_N}(C_k-c_k).
\]
For
\[
 \alpha_N=(2\pi N)^{-1}\|f_1'\|_2,
 \qquad
 q_j=a_j-\frac{(1+\eta_{j-1})+(1+\eta_j)}2,
\]
rigidity gives
\begin{equation}\label{eq:relative-cubic-005}
 \alpha_N^2+\|q\|_{2,N}^2\le C\frac{B_N}{N},
 \qquad d_*\le C\sqrt{B_N/N}.
\end{equation}
The parameter $d_*$ uses the opposite orientation of the conjugate phase. Quasi-uniformity supplies sampling, while low-band coercivity gives
\[
 R_2^2\le C\frac{B_N}{N},
 \qquad
 R_1(L)\le CB_N^{1/2}\frac{L^2}{N^2}
 \qquad(L\le3N/4).
\]
Use $S=\lceil(kN)^{2/3}\rceil$ for $k\le N^{1/4}$ and $S=\lfloor N/8\rfloor$ otherwise. The column bound becomes
\begin{equation}\label{eq:relative-cubic-006}
 \varepsilon_k^2:=\|e_k\|_2^2
 \le CB_N\begin{cases}
 k^{2/3}N^{-1/3}+k^2N^{-1},&k\le N^{1/4},\\
 1+k^2/N,&k>N^{1/4},
 \end{cases}
 \qquad
 T_k^2\le C\min\{1,k^2B_N/N\}.
\end{equation}
With $\ell_M=\log(2M)$, summation yields
\begin{equation}\label{eq:relative-cubic-007}
 E_M^{\rm col}\le B_N\,a^{\mathrm{err}}_N(M),\qquad
 E_M^{\rm tail}\le B_N\,p^{\mathrm{tail}}_N,
\end{equation}
where
\[
 a^{\mathrm{err}}_N(M)
 =C\left(N^{-1/3}+N^{-1/2}\log(2N)+\frac{\ell_M^2}{N}\right),
 \qquad
 p^{\mathrm{tail}}_N
 =\frac{C}{N}\left[1+\log^2\!\left(2\sqrt{N/B_N}\right)\right].
\]
These estimates control the quadratic errors and phase tails, but not yet the completed linear term.

For that term separate the composite-trapezoid measure:
\begin{equation}\label{eq:relative-cubic-008}
 \nu_N=\nu_N^\gamma+\mu_q,
 \quad
 \nu_N^\gamma=\frac1N\sum_j\frac{2+\eta_{j-1}+\eta_j}{2}\,\delta_{x_j},
 \quad
 \mu_q=\frac1N\sum_jq_j\delta_{x_j}.
\end{equation}
Take a smooth taper equal to one on the central residue block, with
\[
 |K_N(x)|\le C\min\{N,N^{-2}\|x\|_{\Tunit}^{-3}\},
 \qquad
 \int\|x\|_{\Tunit}|K_N(x)|\,dx\le C/N.
\]
Unimodularity gives, for $g_{k,N}=1-\operatorname{Re}(\overline f_kA_{{\rm cut},N}f_k)$,
\[
 g_{k,N}(x)=\frac12\int K_N(z)|f_k(x)-f_k(x-z)|^2\,dz,
 \qquad
 \operatorname{Var}(g_{k,N})\le CN^{-1}\|f_k'\|_2^2.
\]
Thus its trapezoid error is $O(k^2\alpha_N^2)$. For the residual measure put $h_k=(I-A_{{\rm cut},N})f_k$; then
\[
 \|h_k\|_2\le Ck\alpha_N,
 \qquad \|h_k'\|_2\le CkN\alpha_N,
 \qquad |g_{k,N}|\le|h_k|.
\]
Sampling and Cauchy--Schwarz give $|\int g_{k,N}\,d\mu_q|\le Ck\alpha_N\|q\|_{2,N}$. Returning to the sharp block costs at most $2\varepsilon_k^2+Ck^2\alpha_N^2$. The completed identity consequently gives
\begin{equation}\label{eq:relative-cubic-009}
 |U_M|
 \le C\left[
 \sum_{k\le2M}\beta_{k,M}\varepsilon_k^2
 +\alpha_N^2(1+\ell_M)^2+\alpha_N\|q\|_{2,N}\right]
 \le B_N\,a^{\mathrm{err}}_N(M).
\end{equation}
The linear contribution now has the same full factor $B_N$ as the quadratic errors.

Choose the alias cutoff by
\begin{equation}\label{eq:relative-cubic-010}
 L_N^{\rm ann}=1+\log N+\log B_N^{-1},
 \qquad
 M=\left\lceil B_N^{-1/2}(L_N^{\rm ann})^{A_*}\right\rceil,
\end{equation}
where $A_*$ dominates the fixed logarithmic losses. The lower defect bound gives $L_N^{\rm ann}+\ell_M=O(N^{1/3})$, hence $a_N^{\mathrm{err}}(M)+p_N^{\mathrm{tail}}=o(1)$. After division by $B_N$ the contributions are
\begin{equation}\label{eq:relative-cubic-011}
\begin{array}{c|c}
\text{term} & \text{bound after division by }B_N\\ \hline
 d_*^2\ell_M^2 & C\ell_M^2/N\\
 d_*\ell_M(E_M^{\rm col})^{1/2}
   & C\ell_M\sqrt{a^{\mathrm{err}}_N(M)/N}\\
 d_*(E_M^{\rm tail})^{1/2}
   & C\sqrt{p^{\mathrm{tail}}_N/N}\\
 E_M^{\rm col},\ E_M^{\rm tail},\ |U_M|
   & C\bigl(a^{\mathrm{err}}_N(M)+p^{\mathrm{tail}}_N\bigr)
\end{array}.
\end{equation}
Each remains $o(1)$ after the fixed logarithmic loss. Also $M^{-2}\le B_N(L_N^{\rm ann})^{-2A_*}$, so the alias tail is $o(B_N)$. It follows that
\begin{equation}\label{eq:relative-cubic-012}
 2N^4\operatorname{Re}\sum_{m,n\ge1}
 \left(T_{m,n,N}-\frac1{N^4mn(m+n)^2}\right)
 =o(B_N).
\end{equation}

For central inputs let $\mathcal L_N=\{2,\ldots,\lfloor N/2\rfloor\}$. Coercivity and Cauchy--Schwarz give
\begin{equation}\label{eq:relative-cubic-013}
 \sum_{k\in\mathcal L_N}\frac{|H_N(k)|^2}{k^3}
 \le C\frac{B_N}{N^4},
 \qquad
 \sum_{k\in\mathcal L_N}\frac{|H_N(k)|}{k}
 \le C\frac{B_N^{1/2}}N.
\end{equation}
With one central input the exact conjugate-column expression is
\[
 \overline{T_N^{1\mathrm c}}
 =\frac2{N^3}\sum_{k\in\mathcal L_N}\frac{H_N(k)}k
   \sum_{m\ge1}\frac{H_{m,k}}{m^3}.
\]
The shifted-return estimates, with the same cutoff, give
\begin{equation}\label{eq:relative-cubic-014}
 \begin{aligned}
 \frac{N^4|T_N^{1\mathrm c}|}{B_N}
 \le C\Bigg[&
 \frac{\log^2(2N)}N+N^{-2/3}\log^C(2N)\\
 &+B_N^{1/2}\left(N^{-1/3}+\frac{\ell_M}{N}\right)\log^C(2N)
 +B_N^{1/2}(L_N^{\rm ann})^{-2A_*}\Bigg]=o(1).
 \end{aligned}
\end{equation}
The terms come from the principal phase part, the linear errors, the quadratic error, and the alias tail; the carry chamber and even central row are included.

With two central inputs put $x_k=|H_N(k)|k^{-3/2}$, so $\|x\|_2^2\le CB_NN^{-4}$. For $a+b=d\le\beta N$, $1/2<\beta<1$ fixed,
\[
 \sum_{a+b=d}\frac{|H_N(a)H_N(b)|}{ab}
 \le Cd\frac{B_N}{N^4},
 \qquad
 \sum_{d\le\beta N}\frac{|H_N(d)|}{d}
 \le C_\beta\frac{B_N^{1/2}}N.
\]
At $d=N+r$, the contribution is at most $CB_NN^{-5}|C_1(r)|$. Sampling gives $\sum_r|C_1(r)|\le C\sqrt N$, and therefore
\begin{equation}\label{eq:relative-cubic-015}
 \frac{N^4|T_N^{2\mathrm c}|}{B_N}
 \le C\left(B_N^{1/2}N^{-1}+N^{-1/2}\right)=o(1).
\end{equation}
The first input harmonic vanishes by centering. The half-open convention counts $d=N$ and the even Nyquist input once. The sectors exhaust the series, and all central sectors vanish at the regular measure. Combine \eqref{eq:relative-cubic-012}, \eqref{eq:relative-cubic-014}, and \eqref{eq:relative-cubic-015}. Continuity and compactness give finite bounds for the remaining finitely many $N$; a monotone tail envelope supplies $\epsilon_N\to0$.
\end{proof}

\subsection{Refined localization of minimizers}\label{sec:refined-localization}

The relative cubic estimate now enters the minimizing inequality together with the sharper absolute errors. The cubic difference can be absorbed into the positive defect, leaving the absolute errors to set a smaller localization scale. This is the quantitative regime assumed in the material comparison; no critical-point equation is needed for that later comparison.

\begin{corollary}\label{cor:refined-localization}

For every sufficiently large \(N\), every global minimizer satisfies

\[
E_N(P)
\le
Ch^5\bigl(\log(2N)\bigr)^3,
\qquad
B_N(\nu_P)\le Ch\bigl(\log(2N)\bigr)^3.
\]

\end{corollary}

\begin{proof}
Set $b=1/32$. The first localization gives
\begin{equation}\label{eq:relative-cubic-021}
 B_N(\nu_P)\le C_bN^{-b}\log^{C_b}(2N).
\end{equation}
Activation, polar centering, and microscopic localization supply respectively $N$ positive atoms, the exact first-moment condition, and quasi-uniformity. The radial estimate supplies the required chart. The regular measure is a rotation of $\sigma_N$; its cubic value is unchanged because every summand has total phase zero. The Bessel normalization and polar transfer give the exact decomposition and its convergence.

For $B_N<e^{-N^{1/3}}$, the assertion follows from $B_N=N^4E_N$ and $e^{-N^{1/3}}\le Ch\log^3(2N)$. In the complementary range the relative cubic estimate gives
\begin{equation}\label{eq:relative-cubic-022}
 \left|\operatorname {Re}\left[
 T_N(\nu_P)-T_N(\nu_{R_N})\right]\right|
 =o(E_N(P)).
\end{equation}
Together with the refined absolute bound, this yields
\begin{equation}\label{eq:relative-cubic-023}
 |\widetilde{R}_N^\circ(P)|
 \le o(E_N(P))+Ch^5\log(2N).
\end{equation}
Minimality, Bessel coercivity, and the geometric estimate imply
\[
 0\ge\left(\frac14-o(1)\right)E_N(P)
      -Ch^5\log^3(2N).
\]
Absorb the relative error. Multiplying the resulting estimate for $E_N$ by $N^4$ and using $N^4h^5=(2\pi)^4h$ gives the estimate for $B_N$.
\end{proof}

\section{Comparison along a common material coordinate}\label{sec:polygonal-transport}

For the remaining terms, vertex motion must be included in the differentiation. We attach fixed labels to the cells and pull both the geometric and spectral expressions back along the moving-mesh map. The atomic defect controls the quadratic size of this deformation. Cyclic symmetry cancels the first variations at the regular endpoint, so second material derivatives estimate the endpoint differences. The periodic estimates in Appendix~\ref{sec:mesh-tools} keep the constants uniform as the number of cells grows.

\subsection{Transport of the polygonal turning data}

Fix $C_0<\infty$, put $h=2\pi/N$, $\theta_j=jh$, and let
\begin{equation}\label{eq:polygonal-transport-001}
 \nu_P=\frac1N\sum_ja_j\delta_{\psi_j},\qquad
 \psi_j=\theta_j+y_j,\qquad y_{j+1}-y_j=h\eta_j,
\end{equation}
be exact-centered near-regular polar data, in the canonical zero-mean gauge, with
\begin{equation}\label{eq:polygonal-transport-002}
 \langle a\rangle_N=1,\quad \langle ae^{i\psi}\rangle_N=0,
 \quad \sum_jy_j=\sum_j\eta_j=0,
 \quad \max_j(|a_j-1|+|\eta_j|)\le\frac14,
\end{equation}
and
\begin{equation}\label{eq:polygonal-transport-003}
 0\le E_N(P)\le C_0h^5\log^3(2N),\qquad
 B_N(P)=N^4E_N(P)\le C_0h\log^3(2N).
\end{equation}
Write $L_N=(\log(2N))^{A_0(C_0)}$, where the exponent is enlarged a finite
number of times below, and put
\begin{equation}\label{eq:polygonal-transport-004}
 \chi_j=\frac{\eta_{j-1}+\eta_j}{2},\qquad \gamma_j=1+\chi_j,
 \qquad \sigma^2=\frac1N\sum_j\eta_j^2.
\end{equation}
Let $\bar\gamma(\psi)=\gamma+U_{{\rm cen},1}(c_X)$ be the unique positive
mass--center correction in the first Fourier weight block, and decompose
\begin{equation}\label{eq:polygonal-transport-005}
 a-1=\chi+q^\sharp+r^{\rm geo},\qquad
 r^{\rm geo}=U_{{\rm cen},1}(c_X+c_Z),
\end{equation}
where $q^\sharp$ has no regular modes $0,\pm1$.  With
\begin{equation}\label{eq:polygonal-transport-006}
 G_X=N^{-4}\sum_j\eta_j^2,
 \qquad
 G_q=\sum_{r\notin\{0,\pm1\}}w_{N,r}|\widehat \vartheta_N^\sharp(r)|^2,
 \qquad G=G_X+G_q,
\end{equation}
let $Q_{\rm st}$ be the zero-mean regular-step primitive
\begin{equation}\label{eq:polygonal-transport-007}
 \partial_xQ_{\rm st}=-h\sum_j(q_j^\sharp+r_j^{\rm geo})\delta_{\theta_j}.
\end{equation}
On $I_j=(\theta_j,\theta_{j+1})$ set $P_0=x-\theta_j-h/2$ and define
\begin{equation}\label{eq:polygonal-transport-008}
 T_t(x)=x+tY(x),\quad Y'=\eta,\quad J_t=T_t'=1+t\eta,
\end{equation}
\begin{equation}\label{eq:polygonal-transport-009}
 P_t=J_tP_0+tQ_{\rm st}+c_{\rm turn}(t),\qquad z_t=\dot P_t=\eta P_0+Q_{\rm st}+c_{\rm turn}'(t),
\end{equation}
where $c_{\rm turn}(t)$ is uniquely chosen by
\begin{equation}\label{eq:polygonal-transport-010}
 \int_{\mathbb T}J_t\tan P_t\,dx=0.
\end{equation}
Let $V_t$ be the normalized primitive
\begin{equation}\label{eq:polygonal-transport-011}
 \partial_xV_t=J_t\tan P_t,
 \qquad \int_{\mathbb T}J_te^{2V_t}\,dx=2\pi,
 \qquad b_t=e^{V_t}\tan P_t-P_t.
\end{equation}
For the moving Bessel form, let $K_{{\rm B}}$ be the fixed form with Fourier multiplier
\begin{equation}\label{eq:polygonal-transport-012}
 \Gamma(k)=\begin{cases}4\beta_{|k|}/k^2,&|k|\ge2,\\0,&k=0,\pm1,
 \end{cases}
 \qquad \beta_k=k+1-j\frac{J_{k+1}(j)}{J_k(j)},
\end{equation}
write $K_{{\rm B},t}(F,H)=K_{{\rm B}}(F\circ T_t^{-1},H\circ T_t^{-1})$ and
\begin{equation}\label{eq:polygonal-transport-013}
 D_N^{\rm raw}(t)
 =2K_{{\rm B},t}(P_t,b_t)
  +K_{{\rm B},t}(b_t,b_t).
\end{equation}

\begin{proposition}
\label{prop:geometric-comparison}
Uniformly for the near-regular data above,
\begin{equation}\label{eq:polygonal-transport-014}
 G\le CE_N(P),\qquad
 \sigma^2\le C\frac{B_N(P)}N,
\end{equation}
and the path in \textup{\eqref{eq:polygonal-transport-008}--\eqref{eq:polygonal-transport-013}} is $C^2$, has the exact endpoints
\begin{equation}\label{eq:polygonal-transport-015}
 D_N^{\rm raw}(0)=D_N(R_N),\qquad
 D_N^{\rm raw}(1)=D_N(P),
\end{equation}
and satisfies
\begin{equation}\label{eq:polygonal-transport-016}
 (D_N^{\rm raw})'(0)=0,
 \qquad
 \sup_{0\le t\le1}|(D_N^{\rm raw})''(t)|
 \le Ch^{1/2}L_N^C G.
\end{equation}
Consequently there is $\varepsilon_N(C_0)\to0$ such that
\begin{equation}\label{eq:polygonal-transport-017}
 |D_N^\circ(P)|
 \le \varepsilon_N(C_0)E_N(P).
\end{equation}
No minimizing or critical-point equation is used in this comparison.
\end{proposition}

\begin{proof}
We first bound the geometric transport by the atomic defect, and then estimate the second variation along that transport. For the trapezoid weights $\gamma_j(t)=1+t\chi_j$ at $\psi_j(t)=\theta_j+ty_j$, the cell-bubble representation is
\begin{equation}\label{eq:polygonal-transport-018}
 \nu_t^\gamma-\frac{d\vartheta}{2\pi}=\frac1{2\pi}B_t'',
 \qquad
 E_{{\rm at},3}(\nu_t^\gamma)=\frac12[B_t]_{\dot H^{1/2}}^2,
 \qquad
 B_t(T_tx)=(1+t\eta(x))^2B_0(x),
\end{equation}
where $B_t(\vartheta)=\tfrac12(\vartheta-\psi_j(t))(\psi_{j+1}(t)-\vartheta)$ on cell $j$. For $\ell=1,2$,
\begin{equation}\label{eq:polygonal-transport-019}
 |B_0|\le Ch^2,
 \quad |B_0(x)-B_0(z)|\le C\min\{h|x-z|,h^2\},
 \quad [\eta^\ell B_0]_{\dot H^{1/2}}^2
 \le Ch^3\|\eta\|_\infty^{2\ell-2}\|\eta\|_2^2.
\end{equation}
Differentiating the transported Gagliardo kernel introduces $(Y(x)-Y(z))/(x-z)$. Four derivatives, split near and away from the diagonal, give
\begin{equation}\label{eq:polygonal-transport-020}
 |E^{(4)}(t)|\le C\|\eta\|_\infty^2G_X,
 \qquad E(t):=E_{{\rm at},3}(\nu_t^\gamma).
\end{equation}
The grouped Hessian gives
\begin{equation}\label{eq:polygonal-transport-021}
 0\le E''(0)
 =2\sum_{1\le r<N/2}H_{N,r}|X_r|^2
 +2\mathbf1_{\{2\mid N\}}H_{N,N/2}|X_{N/2}|^2
 \le CG_X,
\end{equation}
with $X_r=N^{-1}\sum_jy_je^{-ir\theta_j}$ and $0\le H_{N,r}\le Cr^2N^{-3}$. Global minimality of the regular roots on both directions of the path, together with the second- and fourth-derivative bounds, yields
\begin{equation}\label{eq:polygonal-transport-022}
 0\le E_{{\rm at},3}(\gamma,\psi)-\frac{\zeta(3)}{N^3}\le CG_X.
\end{equation}

The first weight block of the centering equation is
\begin{equation}\label{eq:polygonal-transport-023}
 (C_{{\rm cen},\psi} U_{{\rm cen},1})c=A_\psi c+B_\psi\bar c,
 \quad A_\psi=\langle e^{iy}\rangle_N,
 \quad B_\psi=\langle e^{i(2\theta+y)}\rangle_N.
\end{equation}
Microscopic rigidity gives
\begin{equation}\label{eq:polygonal-transport-024}
 \sum_j\bigl((a_j-1)^2+\eta_j^2\bigr)\le CB_N,
 \qquad \|y\|_\infty\le C\sqrt{B_N/N},
\end{equation}
so the block is uniformly invertible. The trapezoid error for $e^{ix}$ implies
\begin{equation}\label{eq:polygonal-transport-025}
 |c_X|\le Ch^2\sigma,
 \qquad \bar\gamma_j\ge\frac12.
\end{equation}
The Fourier coefficients of the correction satisfy
\begin{equation}\label{eq:polygonal-transport-026}
 \sum_{n\ge2}\frac{|\widehat{U_{{\rm cen},1}(c_X)}_\psi(n)|^2}{n^3}
 \le C|c_X|^2\{N^{-3}+\|y\|_\infty^2
                       \log(2/\|y\|_\infty)\},
\end{equation}
and its energy cost is $o(G_X)$. For the cross term use
\begin{equation}\label{eq:polygonal-transport-027}
 \widehat\gamma_\psi(n)-\mathbf1_{\{N\mid n\}}
 =-\frac1{2\pi}\left[D_0(e^{-inT}-e^{-inx})
       +\partial_x(\eta P_0)e^{-inT}\right],
 \quad D_0=dx-h\sum_j\delta_{\theta_j},
\end{equation}
with $|D_0(H)|\le2h\operatorname{TV}(H)$ and the $\dot H^{-1/2}$ bound for the cell-mean-zero function $(\eta P_0)\circ T^{-1}$. Hence
\begin{equation}\label{eq:polygonal-transport-028}
 0\le E_{{\rm at},3}(\bar\gamma(\psi),\psi)-\frac{\zeta(3)}{N^3}
 \le CG_X.
\end{equation}
Polarization between $a$ and $\bar\gamma(\psi)$, using the first-mode equation for $c_Z$, gives
\begin{equation}\label{eq:polygonal-transport-029}
 G_q\le CE_N(P)+C(|c_X|^2+|c_Z|^2),
 \qquad |c_Z|\le\kappa_NG_q^{1/2},
 \quad \kappa_N=o(1).
\end{equation}
After absorption, $G\le CE_N(P)$. Regular-step aliases also give
\begin{equation}\label{eq:polygonal-transport-030}
 \|Q_{\rm st}\|_{\dot H^{-1/2}}\le CG^{1/2},\qquad
 \|Q_{\rm st}\|_2\le Ch^{-1/2}G^{1/2},\qquad
 \|Q_{\rm st}\|_\infty\le Ch^{-1}G^{1/2}.
\end{equation}

The derivative of the scalar turning constraint is $\int J_t\sec^2P_t\ge2\pi$, so $c_{\rm turn}\in C^2$ and
\begin{equation}\label{eq:polygonal-transport-031}
 |c_{\rm turn}|+|c_{\rm turn}'|\le Ch^{-1/2}G^{1/2},
 \quad |c_{\rm turn}''|\le Ch^{-2}G,
 \quad \|z_t\|_2\le Ch^{-1/2}G^{1/2}.
\end{equation}
Apply the periodic primitive estimates to the first two differentiated equations for $V_t$. They give
\begin{equation}\label{eq:polygonal-transport-032}
 \|P_t\|_2+\|P_t\|_\infty\le Ch,
 \quad \|V_t\|_\infty\le Ch^2L_N,
 \quad \|\dot V_t\|_{H^{1/2}}+\|\dot V_t\|_2\le CL_NG^{1/2},
\end{equation}
and
\begin{equation}\label{eq:polygonal-transport-033}
 \|b_t\|_2+\|b_t\|_\infty\le Ch^3L_N,
 \quad
 \|\dot{b}_t\|_2
 +\|\dot{b}_t\|_{\dot H^{-1/2}}
 \le ChL_NG^{1/2}.
\end{equation}
Twice differentiating the exact turning and area constraints yields
\begin{equation}\label{eq:polygonal-transport-034}
 \partial_x\ddot V_t
 =2\eta\sec^2P_tz_t
  +2J_t\sec^2P_t\tan P_tz_t^2
  +J_t\sec^2P_tc_{\rm turn}''(t).
\end{equation}
The finite-mesh endpoint estimate applied to its three terms gives
\begin{equation}\label{eq:polygonal-transport-035}
 \|\partial_x\ddot V_t\|_{L^1}
 +\|\ddot V_t\|_{H^{1/2}}+\|\ddot V_t\|_\infty
 \le Ch^{-2}L_N^CG.
\end{equation}
Put $c_1=e^{V_t}\tan P_t$ and $c_2=e^{V_t}\sec^2P_t-1$. Then
\begin{equation}\label{eq:polygonal-transport-036}
 \ddot{b}_t
 =D_t(e^{V_t}\dot V_t^2)
 +2e^{V_t}\sec^2P_t\tan P_tz_t^2
 +c_1\ddot V_t+c_2c_{\rm turn}''-2(\eta/J_t)c_1\dot V_t.
\end{equation}
The same cell estimates imply
\begin{equation}\label{eq:polygonal-transport-037}
 |K_{{\rm B},t}(P_t,\ddot{b}_t)|
 +|K_{{\rm B},t}(\ddot{b}_t,b_t)|
 \le ChL_N^CG
\end{equation}
with one extra power of $h$ in the second pairing.

For the transported Bessel form, the symbol is
\begin{equation}\label{eq:polygonal-transport-038}
 \Gamma(k)/2=2|k|^{-1}+2k^{-2}+O(|k|^{-3}).
\end{equation}
Its kernel has logarithmic and periodic Bernoulli parts, an $H^2$ remainder, and finitely many low modes. The second derivative of the Bernoulli part has both a diagonal delta and a constant; neither is discarded. The uniform bi-Lipschitz bounds and $|Y(x)-Y(z)|\le C|x-z|$ justify differentiating the integral twice and give
\begin{equation}\label{eq:polygonal-transport-039}
 |K_{{\rm B},t}(A,F)|
 \le C\|A\circ T_t^{-1}\|_{\dot H^{-1/2}}
        \|F\circ T_t^{-1}\|_{\dot H^{-1/2}},
\end{equation}
and
\begin{equation}\label{eq:polygonal-transport-040}
 |K_{{\rm B},t}'(f,g)|\le C\|\eta\|_2\|f\|_2\|g\|_2,
 \qquad
 |K_{{\rm B},t}''(f,g)|\le C\|\eta\|_2^2
                             \|f\|_\infty\|g\|_\infty.
\end{equation}
The corresponding multiplier annihilates constants and satisfies
\begin{equation}\label{eq:polygonal-transport-041}
 \|M_{{\rm B},t}P_t\|_\infty\le Ch^2L_N,
 \qquad
 \|M_{{\rm B},t}b_t\|_\infty\le Ch^3L_N.
\end{equation}
These estimates give the asserted path regularity. The turning jumps at the endpoints are $-h$ on the regular mesh and $-ha_j$ at $\psi_j$. The turning, area, and centering normalizations therefore identify the endpoints with $R_N$ and $P$.

Since $\ddot P_t=c_{\rm turn}''(t)$, the constant arguments vanish in the symmetric form. Differentiation gives
\begin{equation}\label{eq:polygonal-transport-042}
\begin{aligned}
(D_N^{\rm raw})''={}&
 2K_{{\rm B},t}''(P,b)
 +4K_{{\rm B},t}'(z,b)
 +4K_{{\rm B},t}'(P,\dot{b})
 +4K_{{\rm B},t}(z,\dot{b})\\
&+2K_{{\rm B},t}(P,\ddot{b})
 +K_{{\rm B},t}''(b,b)
 +4K_{{\rm B},t}'(\dot{b},b)\\
&+2K_{{\rm B},t}(\ddot{b},b)
 +2K_{{\rm B},t}(\dot{b},\dot{b}).
\end{aligned}
\end{equation}
The nine terms obey
\begin{equation}\label{eq:polygonal-transport-043}
\begin{array}{c|c}
\text{row}&\text{bound}/G\\ \hline
 K_{{\rm B},t}''(P,b),\ K_{{\rm B},t}'(z,b)
 &ChL_N^C\\
K_{{\rm B},t}'(P,\dot{b}),\ K_{{\rm B},t}(z,\dot{b})
 &Ch^{1/2}L_N^C\\
K_{{\rm B},t}(P,\ddot{b}),\
 K_{{\rm B},t}(\ddot{b},b)&ChL_N^C\\
K_{{\rm B},t}''(b,b),\
 K_{{\rm B},t}'(\dot{b},b),\
 K_{{\rm B},t}(\dot{b},\dot{b})&Ch^2L_N^C
\end{array}.
\end{equation}
At the regular endpoint, rotation and cyclic relabeling leave the differential invariant. Its tangent $(a-1,y)$ has both coordinate sums zero, hence $(D_N^{\rm raw})'(0)=0$. Taylor's integral formula, the estimate $G\le CE_N(P)$, and $h^{1/2}L_N^C\to0$ prove the comparison.
\end{proof}

\subsection{The spectral remainder in material coordinates}\label{sec:material-comparison}

The turning data now live on a fixed cyclic mesh. We reconstruct radial states and transport the spectral forms along this same path, so that the estimates refer to the actual geometric endpoints. Two kinds of uniform control are needed: radial derivatives must retain the quadratic deformation size, and derivatives of the transported operators must depend on the mesh distortion with constants uniform in $N$. The theorem below keeps these inputs separate before combining them in the endpoint comparison.

Along the same geometric chord, set
\begin{equation}\label{eq:material-comparison-001}
 A=G_X^{1/2},\qquad D_*=G^{1/2},\qquad \delta=\|\eta\|_\infty.
\end{equation}
Here $A\le D_*$,
$\delta\le Ch^{-2}A$, and $\|\eta\|_2\le Ch^{-3/2}A$.  Define the
material radial states by
\begin{equation}\label{eq:material-comparison-002}
 \bar p_t=\int J_tP_t,\qquad
 \partial_xV_{L,t}=J_t(P_t-\bar p_t),\qquad \int J_tV_{L,t}=0,
\end{equation}
\begin{equation}\label{eq:material-comparison-003}
 R_t=e^{V_t}-1,
 \qquad W_t=R_t-V_{L,t}.
\end{equation}
The turning and area constraints imply the exact identities
\begin{equation}\label{eq:material-comparison-004}
 \partial_xW_t=J_t(b_t+\bar p_t),
 \qquad
 \int J_tW_t=-\frac12\int J_tR_t^2.
\end{equation}
For $C_tf=f\circ T_t^{-1}$ let $S_{{\rm tr},t}=C_t^{-1}SC_t$ denote the
transport of the fixed lower-Bessel smoothing multiplier, and let
$B_{{\rm tr},t}(f,g)=\langle \Lambda_\theta(C_tf),C_tg\rangle$.

\begin{theorem}
\label{thm:material-endpoint-comparison}
\label{thm:annular-material-comparison}
For every near-regular datum satisfying \textup{\eqref{eq:polygonal-transport-001}--\eqref{eq:polygonal-transport-003}}, the path above admits a
fixed radial collar realization and a uniform analytic eigenpair expansion.
The estimates below are uniform in the path parameter. Each row may contain an additional fixed power of $L_N$.

\smallskip
\noindent\textup{\bfseries (i) Radial jets.}
Uniformly for $0\le t\le1$,
\begin{equation}\label{eq:material-comparison-005}
\begin{array}{c|ccc}
 &L^\infty&W^{1,\infty}&H^{1/2}\\ \hline
V_{L,t},\ R_t&h^2&h&h^{3/2}\\
W_t&h^3&h^3&h^3\\
\dot V_{L,t},\ \dot R_t&D_*&h^{-1}D_*&D_*\\
\dot W_t&hD_*&hD_*&hD_*\\
\ddot V_{L,t},\ \ddot R_t&h^{-2}G&h^{-3}G&h^{-2}G\\
\ddot W_t&h^{-1}G&h^{-1}G&h^{-1}G
\end{array}.
\end{equation}

\smallskip
\noindent\textup{\bfseries (ii) Uniform analytic tail.}
There is a linear collar map $Z$ from
$W^{1,\infty}(\mathbb T)\cap H^{1/2}(\mathbb T)$ to radial disk fields,
independent of $N$, such that
\begin{equation}\label{eq:material-comparison-006}
 L(Z_f):=\|DZ_f\|_\infty\le C\|f\|_{W^{1,\infty}},
 \qquad
 E(Z_f):=\|DZ_f\|_2\sim \|f\|_{H^{1/2}}.
\end{equation}
For $U_t=Z(R_t)$ the pulled lowest eigenvalue has a uniform analytic
expansion
\begin{equation}\label{eq:material-comparison-007}
 \lambda_{t,U}-j^2=\sum_{r\ge2}J_{{\rm sp},r,t}(U,\ldots,U),
 \qquad
 J_{{\rm sp},r,t}(Z_1,\ldots,Z_r)
 :=\frac1{r!}\left.\partial_{\varepsilon_1}\cdots
 \partial_{\varepsilon_r}
 (\lambda_{t,\sum_i\varepsilon_iZ_i}-j^2)\right|_{\varepsilon=0}.
\end{equation}
For $r\ge2$ and $a=0,1,2$,
\begin{equation}\label{eq:material-comparison-008}
 |\partial_t^aJ_{{\rm sp},r,t}(Z_1,\ldots,Z_r)|
 \le C_0^r\delta^a
 \sum_{\alpha\ne\beta}E(Z_\alpha)E(Z_\beta)
 \prod_{\ell\notin\{\alpha,\beta\}}L(Z_\ell).
\end{equation}
Consequently, for
\begin{equation}\label{eq:material-comparison-009}
 \Phi_{\ge4}(t)=j^{-2}\sum_{r\ge4}
 J_{{\rm sp},r,t}(U_t,\ldots,U_t),
\end{equation}
the series is $C^2$ and
\begin{equation}\label{eq:material-comparison-010}
 \sup_t|\Phi_{\ge4}''(t)|\le ChL_N^CG,
 \qquad
 |\Phi_{\ge4}(1)-\Phi_{\ge4}(0)|
 \le ChL_N^CE_N(P).
\end{equation}

\smallskip
\noindent\textup{\bfseries (iii) Cubic differences and endpoint identity.}
Let
\begin{equation}\label{eq:material-comparison-011}
\begin{aligned}
 C_{{\rm low}}(v_L;t)&=C_{{\rm low},t}(V_{L,t},V_{L,t},V_{L,t}),\\
 C_3(\rho;t)&=j^{-2}J_{{\rm sp},3,t}(ZR_t,ZR_t,ZR_t),\\
 C_3(v_L;t)&=j^{-2}J_{{\rm sp},3,t}(ZV_{L,t},ZV_{L,t},ZV_{L,t}).
\end{aligned}
\end{equation}
Their first derivatives vanish at the regular endpoint, and
\begin{equation}\label{eq:material-comparison-012}
 |C_{{\rm low}}^\circ(v_L)|
 \le ChL_N^CE_N(P),
\end{equation}
\begin{equation}\label{eq:material-comparison-013}
 \left|[C_3(\rho)-C_3(v_L)]^\circ\right|
 \le Ch^{3/2}L_N^CE_N(P).
\end{equation}
If $(\nu_P,\rho,v_L)$ and $(\sigma_N,\rho_R,v_{L,R})$ denote the two
endpoints, then
\begin{equation}\label{eq:material-comparison-014}
\begin{aligned}
\widetilde{R}_N^\circ(P)
={}&-8\operatorname {Re}[T_N(\nu_P)-T_N(\sigma_N)]\\
&+C_{{\rm low}}(v_L)-C_{{\rm low}}(v_{L,R})\\
&+[C_3(\rho)-C_3(v_L)]
 -[C_3(\rho_R)-C_3(v_{L,R})]\\
&+\Phi_{\ge4}(1)-\Phi_{\ge4}(0).
\end{aligned}
\end{equation}
Moreover $\nu_P$ has $N$ distinct cyclic nodes, positive mass-one weights,
$\widehat\nu_P(1)=0$, and
\begin{equation}\label{eq:material-comparison-015}
 B_N(\nu_P)=B_N(P)=N^4E_N(P).
\end{equation}
Therefore, if additionally
\begin{equation}\label{eq:material-comparison-016}
 e^{-N^{1/3}}\le B_N(P)\le C_0h\log^3(2N),
\end{equation}
then there is $\omega_{C_0,N}\to0$ such that
\begin{equation}\label{eq:material-comparison-017}
 |\widetilde{R}_N^\circ(P)|
 \le \omega_{C_0,N}E_N(P).
\end{equation}
No minimizing or critical-point equation is used in this comparison.
\end{theorem}

\begin{proof}
The radial states can be estimated before the spectral expansion is used. Let $I_{\rm per}$ be the normalized periodic primitive. Cell cancellation gives, for regular steps with the required zero means,
\begin{equation}\label{eq:material-comparison-018}
 \|aP_0\|_\infty\le Ch\|a\|_\infty,
 \quad \|aP_0\|_{\dot H^{-1/2}}\le Ch^{3/2}\|a\|_2,
 \quad
 \|I_{\rm per}q\|_\infty
 \le C\sqrt{\log(2N)}\|q\|_{\dot H^{-1/2}}.
\end{equation}
The linear trace obeys
\begin{equation}\label{eq:material-comparison-019}
 \partial_xV_{L,t}=S_t-J_tm_t,
 \qquad
 S_t=J_t^2P_0+tJ_tQ_{\rm st},
 \qquad m_t=\int S_t=t^2\int\eta Q_{\rm st}.
\end{equation}
Differentiate twice and use the transport bounds, $\delta\le Ch^{-2}A$, and the cell estimates:
\begin{equation}\label{eq:material-comparison-020}
\begin{array}{c|ccc}
 &\|\partial_xV_{L,t}^{(a)}\|_\infty
 &\|\partial_xV_{L,t}^{(a)}\|_{\dot H^{-1/2}}
 &\|I_{\rm per}\partial_xV_{L,t}^{(a)}\|_\infty\\ \hline
 a=0&Ch&Ch^{3/2}L_N^C&Ch^2L_N^C\\
 a=1&Ch^{-1}L_N^CD_*&CL_N^CD_*&CL_N^CD_*\\
 a=2&Ch^{-3}L_N^CG&Ch^{-2}L_N^CG&Ch^{-2}L_N^CG
\end{array}.
\end{equation}
The moving-mean condition then bounds $V_{L,t}$. For the nonlinear trace use
\begin{equation}\label{eq:material-comparison-021}
 \|V_t\|_\infty\le Ch^2L_N,
 \quad
 \|\dot V_t\|_\infty+\|\dot V_t\|_{H^{1/2}}
 \le CL_N^CD_*,
 \quad
 \|\ddot V_t\|_\infty+\|\ddot V_t\|_{H^{1/2}}
 \le Ch^{-2}L_N^CG,
\end{equation}
and
\begin{equation}\label{eq:material-comparison-022}
 \dot R_t=e^{V_t}\dot V_t,
 \qquad
 \ddot R_t=e^{V_t}(\ddot V_t+\dot V_t^2).
\end{equation}
The difference trace satisfies
\begin{equation}\label{eq:material-comparison-023}
 \partial_x\dot W_t
 =\eta(b_t+\bar p_t)
   +J_t(\dot{b}_t+\dot{\bar p}_t).
\end{equation}
One further derivative, together with the exact mean identity \eqref{eq:material-comparison-004}, gives the remaining radial-jet rows. Products with a step factor are controlled in $L^2\hookrightarrow H^{-1/2}$.

For the analytic realization, fix an even mollifier and a cutoff, and define
\begin{equation}\label{eq:material-comparison-024}
 z_f(r,x)=
 \begin{cases}
 (\varphi_{1-r}*f)(x),&3/4\le r\le1,\\
 \chi(r)(\varphi_{1/4}*f)(x),&1/2\le r\le3/4,\\
 0,&0\le r\le1/2.
 \end{cases}.
\end{equation}
Fourier integration gives \eqref{eq:material-comparison-006}. For small $L(Z_f)$, the collar has $r+z_f\ge1/4$ and $1+\partial_rz_f\ge1/2$. The map $(r,x)\mapsto(r+z)e^{iT_t(x)}$ has coefficients rational in $J_t^{\pm1}$, $R=r+z$, $R_r$, and $R_x$, with denominators bounded away from zero. At zero radial amplitude it is a disk self-map, with normalized ground state $(u_0,j^2)$. The disk gap provides a uniform saddle inverse on the moving complement, and hence an analytic eigenpair on one complex ball independent of $t$, $N$, and finite-dimensional restrictions.

In a mixed coefficient retain two distinct $L^2$ deformation factors and estimate the rest in $L^\infty$. A material derivative of a rational coefficient costs $C\delta$, giving \eqref{eq:material-comparison-008} by the common-domain coefficient lemma. For the lower-Bessel operator, transport derivatives of the logarithmic kernel and the integrable order-minus-two remainder insert only $\eta$ and differences of $Y$. Schur's test gives the supremum estimate. Localizing at the diagonal and applying the Calder\'on commutator bound \citep{Calderon65,CMM82} gives the Sobolev estimates; the remaining periodic kernel is smooth and integrable. Therefore
\begin{equation}\label{eq:material-comparison-025}
 \|\partial_t^aS_{{\rm tr},t}\|_{L^\infty\to L^\infty}
 +\|\partial_t^aS_{{\rm tr},t}\|_{H^{1/2}\to H^{1/2}}
 +\|\partial_t^aS_{{\rm tr},t}\|_{L^2\to H^1}
 \le C\delta^a,
 \quad a=0,1,2.
\end{equation}
The Gagliardo representation of $\Lambda_\theta$ also gives, for every cubic generator,
\begin{equation}\label{eq:material-comparison-026}
 |\partial_t^a\Lambda(C_tf,C_tg,C_th)|
 \le C\delta^a\sum_{\rm cyc}
 \|f\|_\infty\|g\|_{H^{1/2}}\|h\|_{H^{1/2}}.
\end{equation}

Up to $L_N^C$, the deformation jets are
\begin{equation}\label{eq:material-comparison-027}
\begin{array}{c|cc}
 &L&E\\ \hline
U_t&h&h^{3/2}\\
\dot U_t&h^{-1}D_*&D_*\\
\ddot U_t&h^{-3}G&h^{-2}G
\end{array}.
\end{equation}
Insert them into the mixed coefficient bound. For $\vartheta_N\le ChL_N^C\to0$ and $r\ge4$,
\begin{equation}\label{eq:material-comparison-028}
 |f_r|\le Cr^2\vartheta_N^{r-4}h^5L_N^C,
 \quad
 |f_r'|\le Cr^3\vartheta_N^{r-4}h^3L_N^CD_*,
 \quad
 |f_r''|\le Cr^4\vartheta_N^{r-4}hL_N^CG.
\end{equation}
The tail and both derivatives converge uniformly. Summation bounds $\Phi_{\ge4}''$ as in \eqref{eq:material-comparison-010}. Cyclic averaging gives a zero first derivative at the regular endpoint, so Taylor's formula and $G\le CE_N(P)$ give the endpoint bound.

Cyclic averaging also applies to the lower cubic and the difference of the cubic traces: their differentials are invariant linear functionals of $(a-1,y)$. The multiplier and radial-jet estimates give
\begin{equation}\label{eq:material-comparison-029}
 \left|\frac{d^2}{dt^2}C_{{\rm low},t}(V_{L,t},V_{L,t},V_{L,t})\right|
 \le ChL_N^CG,
\end{equation}
and
\begin{equation}\label{eq:material-comparison-030}
 \left|\frac{d^2}{dt^2}
 [C_3(\rho;t)-C_3(v_L;t)]\right|
 \le Ch^{3/2}L_N^CG.
\end{equation}
For the latter, expand $R_t^3-V_{L,t}^3=3V_{L,t}^2W_t+3V_{L,t}W_t^2+W_t^3$. Each term contains $W_t$. The largest two-energy contributions are $\delta^2h^{11/2}$, $\delta h^{7/2}D_*$, $h^{3/2}G$, and $h^2G$. Taylor's formula gives \eqref{eq:material-comparison-012}--\eqref{eq:material-comparison-013}.

The physical endpoint identities are
\begin{equation}\label{eq:material-comparison-031}
 C_1V_{L,1}=v_L,
 \quad C_0V_{L,0}=v_{L,R},
 \qquad
 C_1R_1=\rho,
 \quad C_0R_0=\rho_R.
\end{equation}
Their spectral expansions are
\begin{equation}\label{eq:material-comparison-032}
 \widetilde{R}_N(P)
 =C_3(\rho)+R_{\ge4}(\rho),
 \qquad
 \widetilde{R}_N(R_N)
 =C_3(\rho_R)+R_{\ge4}(\rho_R).
\end{equation}
The difference of the tails is the endpoint difference in \eqref{eq:material-comparison-009}. On the similarity slice,
\begin{equation}\label{eq:material-comparison-033}
 C_3(v_L)
 =-8\operatorname {Re}T_N(\nu_P)+C_{{\rm low}}(v_L),
\end{equation}
with the same formula at $R_N$. Add and subtract the linear-trace cubics to obtain \eqref{eq:material-comparison-014}, with normalization \eqref{eq:material-comparison-015}.

Finally choose $b=1/4$. For fixed $C_b=C_b(C_0)$, the upper bound $B_N(P)\le C_0h\log^3(2N)$ is in the relative cubic range. The endpoint identities and quasi-uniformity supply the other assumptions, so
\begin{equation}\label{eq:material-comparison-034}
 8|\operatorname {Re}[T_N(\nu_P)-T_N(\sigma_N)]|
 \le o(1)\,E_N(P).
\end{equation}
Combining it with the lower-cubic and tail bounds gives
\[
 |\widetilde{R}_N^\circ(P)|
 \le\{o(1)+C(hL_N^C+h^{3/2}L_N^C)\}E_N(P).
\]
The argument uses only the path estimates, not a minimizing or critical-point equation.
\end{proof}

\section{Exact constraints and the removal of the defect cutoff}\label{sec:exact-constraints}

The relative estimate must now be extended below its exponential cutoff. We use exact mass and centering coordinates, with area imposed in reconstruction, so that every point in the chart represents an admissible polygon. Polynomial bounds through the third derivative allow a comparison at an inverse-power scale to determine the quadratic coefficients at the regular configuration. Taylor expansion then extends the estimate to smaller defects. Chart coverage ensures that this argument reaches every sufficiently close polygon.

\subsection{Realization of the exact constraint manifold}

In this subsection $x_j=x_j^{\rm ex}$ is the angular displacement in
\eqref{eq:curvature-coordinates-015}, not the unit-circle node
$\psi_j/(2\pi)$. Thus
\[
 w_j=N^{-1}+u_j,\qquad \psi_j=jh+x_j,
\]
and we write
\[
 C_{{\rm cen},N}(u,x)=\sum_j(N^{-1}+u_j)e^{i(jh+x_j)},\qquad
 L_N^{\rm cen}(u,x)=\sum_j\left(u_j+\frac iN x_j\right)e^{ijh}.
\]
In the zero-mass, zero-rotation gauge set
\[
 \begin{aligned}
 \mathcal M_N
 &=\{(u,x):\textstyle\sum_j u_j=\sum_j x_j=0,
                  \ C_{{\rm cen},N}(u,x)=0\},\\
 \mathcal Z_N
 &=\{(u,x):\textstyle\sum_j u_j=\sum_j x_j=0,
                  \ L_N^{\rm cen}(u,x)=0\}.
 \end{aligned}
\]
For $q\in\mathbb C$ let
\begin{equation}\label{eq:exact-constraints-001}
 [n(q)]_j=\frac2N\operatorname {Re}(qe^{-ijh});
 \qquad
 \sum_jn(q)_j=0,
 \quad L_N^{\rm cen}(n(q),0)=q,
 \quad \|n(q)\|_2\le2N^{-1/2}|q|.
\end{equation}

\begin{theorem}
\label{thm:constraint-chart}
There are fixed exponents $A_{\rm geo},C_{\rm geo}<\infty$ such that, for all
sufficiently large $N$, the following assertions hold.

\smallskip
\noindent\textup{(i) Exact constraint chart.}
There is a real-analytic chart
\[
 \Xi_N:B_{N^{-4}}(0)\cap\mathcal Z_N\longrightarrow\mathcal M_N,
 \qquad \Xi_N(0)=0,\qquad D\Xi_N(0)=I,
\]
whose derivatives through order three, and those of its local inverse, are
bounded by $CN^6$.  For one fixed $c>0$,
\begin{equation}\label{eq:exact-constraints-002}
 \mathcal M_N\cap B_{cN^{-4}}(0)
 \subset\Xi_N(B_{N^{-4}}(0)\cap\mathcal Z_N),
 \qquad |\Xi_N^{-1}(\zeta)|\le5|\zeta|.
\end{equation}
Every chart point has positive weights and consecutive angular gaps in
$[h/2,3h/2]$.

\smallskip
\noindent\textup{(ii) Polygon realization.}
For $|z|<N^{-A_{\rm geo}}$, if
\[
 \Xi_N(z)=(v,x),\qquad w_j=N^{-1}+v_j,\qquad \psi_j=jh+x_j,
\]
then these data determine a unique strictly convex area-$\pi$ polygon
$P_N(z)$ whose polar center is the origin, whose vertex directions are
$\psi_j$, and whose normalized exterior angles are $w_j$.  Moreover
\begin{equation}\label{eq:exact-constraints-003}
 P_N(0)=R_N,\qquad
 R_N=\operatorname {int}\operatorname {conv}
 \{r_N^{\rm out}e^{ijh}:0\le j<N\},\qquad
 r_N^{\rm out}=\left(\frac{2\pi}{N\sin h}\right)^{1/2}.
\end{equation}
The vertices, the fixed-fan map and its inverse, the pulled stiffness and
mass coefficients, and the simple eigenpair are real analytic, with derivatives through order
three bounded by $N^{C_{\rm geo}}$.  The fan maps are uniformly
bi-Lipschitz, the pulled forms uniformly elliptic, and in particular
$z\mapsto F(P_N(z))$ has derivatives through order three bounded by
$N^{C_{\rm geo}}$.

\smallskip
\noindent\textup{(iii) Near-regular compatibility and coverage.}
Fix $A>\max\{A_{\rm geo},8,C_{\rm geo}+2\}$.  If $|z|<N^{-A}$, then $P_N(z)$ is an
area-$\pi$ exact-centered near-regular datum and, in its canonical cyclic gauge,
\begin{equation}\label{eq:exact-constraints-004}
 \max_j(|Nw_j-1|+|\eta_j|)\le\frac14,
 \qquad \|\rho_z\|_{W^{1,\infty}}\le\varepsilon_{\rm chart}.
\end{equation}
Conversely, if $K/2>\max\{4,A\}+1$, every area-$\pi$ exact-centered
near-regular configuration with $B_N(P)<N^{-K}$ is, after the canonical rotation
and cyclic relabeling, $P_N(z)$ for a unique $z\in\mathcal Z_N$ with
$|z|<N^{-A}$.  Rotation and relabeling leave $B_N$, $D_N^\circ$,
and $\widetilde{R}_N^\circ$ unchanged.
\end{theorem}

\begin{proof}
We construct the constraint chart in the weight variables and then recover the polygon from its turning angles. For $z=(u,x)\in\mathcal Z_N$, use
\[
 C_N^{\rm res}(z)=C_{{\rm cen},N}(u,x),\qquad
 A_N^{\rm cen}(x)q=C_{{\rm cen},N}(n(q),x),
 \qquad \Psi_N(z,q)=C_N^{\rm res}(z)+A_N^{\rm cen}(x)q.
\]
At the regular configuration, root-of-unity cancellation gives $A_N^{\rm cen}(0)=I$. The bound for $n(q)$ and $|e^{is}-1|\le|s|$ imply
\begin{equation}\label{eq:exact-constraints-005}
 \|A_N^{\rm cen}(x)-I\|\le2N^{-1/2}|x|.
\end{equation}
The tangent equation removes the first-order terms of the residual, leaving
\[
 C_N^{\rm res}(z)=\sum_ju_je^{ijh}(e^{ix_j}-1)
 +\frac1N\sum_je^{ijh}(e^{ix_j}-1-ix_j).
\]
In particular, $|C_N^{\rm res}(z)|\le3|z|^2$ for $|z|\le N^{-4}$. On this ball the correction matrix is invertible, and we set
\begin{equation}\label{eq:exact-constraints-006}
 c_N^{\rm cen}(z)=-A_N^{\rm cen}(x)^{-1}C_N^{\rm res}(z),\qquad |c_N^{\rm cen}(z)|\le6|z|^2,
 \qquad \Xi_N(z)=(u+n(c_N^{\rm cen}(z)),x).
\end{equation}
The correction is quadratic at zero, so $\Xi_N(0)=0$ and $D\Xi_N(0)=I$. Differentiating the finite sums and the inverse identity gives
\begin{equation}\label{eq:exact-constraints-007}
 \max_{1\le k\le3}\bigl(\|D^kc_N^{\rm cen}\|+\|D^k\Xi_N\|\bigr)\le CN^6.
\end{equation}
The inverse on the constraint manifold is
\begin{equation}\label{eq:exact-constraints-008}
 \Xi_N^{-1}(v,y)=\bigl(v-n(L_N^{\rm cen}(v,y)),y\bigr),\qquad
 |\Xi_N^{-1}(v,y)|\le5|(v,y)|.
\end{equation}
Composing in either order gives the identity. The norm estimate proves coverage of the smaller ball. Also $|w_j-N^{-1}|<(2N)^{-1}$ and $|x_{j+1}-x_j|<h/2$, including the last cyclic gap. This proves the assertions about weights and gaps.

For the geometric realization, choose increasing lifts $\psi_j=jh+x_j$ and introduce
\[
 \phi_j(c)=c+\frac h2+2\pi\sum_{\ell=1}^jw_\ell,
 \qquad
 g_N^{\rm close}(c)=\sum_j\int_{\psi_j}^{\psi_{j+1}}
             \tan(\theta-\phi_j(c))\,d\theta.
\]
A sufficiently small polynomial ball keeps all arguments of the tangent in a fixed compact subinterval of $(-\pi/2,\pi/2)$. There,
\begin{equation}\label{eq:exact-constraints-009}
 \partial_cg_N^{\rm close}
 =-\sum_j\int_{\psi_j}^{\psi_{j+1}}
                 \sec^2(\theta-\phi_j(c))\,d\theta\le-2\pi.
\end{equation}
The regular value is zero and the perturbation is $O(N^{8-A_{\rm geo}})$. After increasing $A_{\rm geo}$, the closing equation therefore has a unique real-analytic solution $c_N(v,x)$.

Integrate $(V_z^0)'=\tan(\theta-\phi_j(c_N))$ on successive cells, starting with $V_z^0(\psi_0)=0$. The closing equation makes the resulting primitive periodic. The area constraint fixes its additive constant through
\[
 \kappa_z=-\frac12\log\left(\frac1{2\pi}\int e^{2V_z^0}\right),
 \qquad r_z=e^{V_z^0+\kappa_z}.
\]
On every cell,
\[
 \frac d{d\theta}\log
 \bigl(r_z(\theta)\cos(\theta-\phi_j(c_N))\bigr)=0.
\]
Thus each radial piece is a line segment. The bounds $(2N)^{-1}\le w_j\le3(2N)^{-1}$ put the normal increments in $(0,\pi)$ for large $N$, giving strict convexity. Moreover, $\Phi_{P_N(z)}(0)=2\pi\sum_jw_je^{i\psi_j}=0$, so uniqueness of the polar center places it at the origin. The monotone closing equation, integration of the radial equation, and area normalization also prove uniqueness of the polygon.

The vertex expressions are finite sums and analytic scalar operations on fixed compact ranges. Their derivatives through order three have polynomial bounds in $N$. On a regular fan triangle the affine map is determined by
\[
 A_j(z)[\zeta_j^0\ \zeta_{j+1}^0]=[\zeta_j(z)\ \zeta_{j+1}(z)].
\]
The regular determinant is comparable to $N^{-1}$. Shrinking the polynomial ball makes the vertex perturbation $o(h)$; consequently the fan maps, their inverses, and their pulled forms have uniform bounds and polynomial $C^3$ bounds. Uniform ellipticity and perturbation of the simple eigenpair give the spectral assertion.

For compatibility, $|z|<N^{-A}$ gives $|\Xi_N(z)|\le CN^{6-A}$ and
\[
 \max_j(|Nv_j|+|(x_{j+1}-x_j)/h|)\le CN^{7-A}=o(1).
\]
The vertex displacement is at most $\delta_N=N^{C_{\rm geo}}|z|$. Since regular edges have length comparable to $h$, the supporting-line formula yields
\[
 \|r_z-1\|_\infty+\|r_z'\|_\infty
 \le C(h+N\delta_N)=o(1).
\]
This places the realization in the required near-disk chart. Conversely, take the zero-mean cyclic coordinates $\zeta\in\mathcal M_N$ of a small-defect configuration. Microscopic rigidity gives
\begin{equation}\label{eq:exact-constraints-010}
 |\zeta|^2\le CB_N(P).
\end{equation}
The condition on $K$ puts these coordinates in $B_{cN^{-4}}$ and their inverse image in $B_{N^{-A}}$. The chart and the unique geometric realization recover the original polygon. Finally, rotation and relabeling preserve the spectral functional and the atomic Fourier energies, proving the stated invariances.
\end{proof}

\begin{lemma}
\label{lem:vanishing-defect}
With $P_N(z)$ from the exact near-regular chart, set
\[
 B_N^{\rm ch}(z)=B_N(P_N(z))=N^4E_N(P_N(z)),
\]
\begin{equation}\label{eq:exact-constraints-011}
 S_N(z)=N^4\bigl[
   D_N^\circ(P_N(z))
  +\widetilde{R}_N^\circ(P_N(z))\bigr].
\end{equation}
There are constants $c_0>0$, exponents $A_0>A_1+2$, a fixed
$K_0>2A_0$, and a sequence $\omega_N^{(0)}\to0$ such that, for all large
$N$,
\begin{equation}\label{eq:exact-constraints-012}
 B_N^{\rm ch}(0)=DB_N^{\rm ch}(0)=0,
 \qquad \frac12D^2B_N^{\rm ch}(0)[z,z]\ge c_0|z|^2,
 \qquad S_N(0)=DS_N(0)=0,
\end{equation}
\begin{equation}\label{eq:exact-constraints-013}
 \sup_{|z|<N^{-A_0}}
 \bigl(\|D^3B_N^{\rm ch}(z)\|+\|D^3S_N(z)\|\bigr)
 \le N^{A_1},
\end{equation}
and
\begin{equation}\label{eq:exact-constraints-014}
 |S_N(z)|\le\omega_N^{(0)}B_N^{\rm ch}(z)
 \qquad
 (|z|<N^{-A_0},\ 0<B_N^{\rm ch}(z)<N^{-K_0}).
\end{equation}
\end{lemma}

\begin{proof}
The comparison will be extended along rays in the tangent space. First we identify its second-order terms and obtain a common bound for the Taylor errors. Put $\zeta(z)=\Xi_N(z)$ and $\nu_{u,x}=\sum_j(N^{-1}+u_j)\delta_{jh+x_j}$. On the exact constraint manifold,
\begin{equation}\label{eq:exact-constraints-015}
 d_{3,N}(\nu_{u,x})=: d_{3,N}^{\rm B}(u,x),
 \qquad
 \frac14Q_N^\circ(\nu_{u,x})=: d_{B,N}^{\rm B}(u,x).
\end{equation}
Here $D\Xi_N(0)=I$ on $\mathcal Z_N$. Both ambient first variations vanish on the mass-zero, rotation-zero space. The vector $D^2\Xi_N(0)[z,z]$ belongs to this space, so the corresponding chain-rule term is zero. Local atomic coercivity gives
\begin{equation}\label{eq:exact-constraints-016}
 \frac12D^2B_N^{\rm ch}(0)[z,z]
 =N^4H_{3,N}(z)\ge c_0|z|^2.
\end{equation}
Take initially $A_0>A_{\rm geo}+7$, ensuring that the realized polygon lies in the local atomic box. The coordinate energies have third-order bounds $O(N^4)$ and the chart derivatives have bounds $O(N^6)$. The third-order chain rule yields
\[
 \|D^3B_N^{\rm ch}\|\le CN^{26}.
\]
For the remaining spectral terms, use the exact identity
\begin{equation}\label{eq:exact-constraints-017}
 S_N(z)=N^4\left[
 \frac{F(P_N(z))-F(R_N)}{F(\mathbb D)}
 -Q_N^\circ(P_N(z))\right],
\end{equation}
together with the bound $N^{C_{\rm geo}}$ for the realized functional. This yields $\|D^3S_N\|\le N^{C_*}$. After increasing a fixed exponent, we may take $C_*=\max\{27,C_{\rm geo}+5\}$.

Both values at zero vanish by definition. To compute the first derivatives, conjugate the cyclic shift by the chart: $U_N^{\rm cyc}=\Xi_N^{-1}\Gamma\Xi_N$. Its derivative at zero is $\Gamma$. Relabeling rotates the realized polygon, so the two functions are invariant under $U_N^{\rm cyc}$ and their differentials are $\Gamma$-invariant. Since $\sum_{k=0}^{N-1}\Gamma^kz=0$ on $\mathcal Z_N$, averaging gives zero differentials. Choose
\[
 A_1>\max\{C_*,C_{\rm geo},6\},
 \qquad
 A_0>\max\{10,A_{\rm geo}+7,C_{\rm geo}+2,A_1+2\}
\]
to obtain the common derivative bounds on the stated ball.

The geometric and material estimates already control the range
\begin{equation}\label{eq:exact-constraints-018}
 e^{-N^{1/3}}\le B_N^{\rm ch}(z)\le3h\log^3(2N),
\end{equation}
by
\begin{equation}\label{eq:exact-constraints-019}
 |S_N(z)|\le\omega_N^{\rm ann}B_N^{\rm ch}(z).
\end{equation}
Fix $K_0>2A_0+2A_1+8$ and let $\tau_N=N^{-K_0}$. For each unit vector $e\in\mathcal Z_N$, Taylor expansion gives
\begin{equation}\label{eq:exact-constraints-020}
 B_N^{\rm ch}(te)=b_N^{(2)}(e)t^2+O(N^{A_1}t^3),\qquad
 S_N(te)=s_N^{(2)}(e)t^2+O(N^{A_1}t^3),
\end{equation}
with $b_N^{(2)}(e)=\tfrac12D^2B_N^{\rm ch}(0)[e,e]\ge c_0$. Choose $\rho_e=(2\tau_N/b_N^{(2)}(e))^{1/2}$. This radius is smaller than $N^{-A_0}$, and
\[
 \frac32\tau_N\le B_N^{\rm ch}(\rho_ee)\le\frac52\tau_N.
\]
The inequalities $e^{-N^{1/3}}\ll\tau_N\ll h\log^3(2N)$ place this comparison point inside the annulus. Apply the relative estimate there and divide the second Taylor expansion by $\rho_e^2$. The result is
\begin{equation}\label{eq:exact-constraints-021}
 \frac{|s_N^{(2)}(e)|}{b_N^{(2)}(e)}
 \le C\omega_N^{\rm ann}+CN^{A_1-K_0/2}=o(1).
\end{equation}
The condition $A_0>A_1+2$ also gives $B_N^{\rm ch}(te)\ge\tfrac12b_N^{(2)}(e)t^2$ for $0\le t\le N^{-A_0}$. Hence $0<B_N^{\rm ch}(se)<\tau_N$ implies $s<\rho_e$. Returning to the two Taylor expansions now gives
\[
 \frac{|S_N(se)|}{B_N^{\rm ch}(se)}
 \le C\omega_N^{\rm ann}+CN^{A_1-K_0/2}
 =:\omega_N^{(0)}\longrightarrow0,
\]
uniformly in $e$ and in the positive defect. Thus no lower cutoff remains.
\end{proof}

\section{The global spectral comparison}\label{sec:spectral-rigidity}

Every global minimizer lies in the centered, nearly regular class obtained above. On this class, the annular and small-defect estimates together bound the full remainder by a vanishing fraction of the positive atomic term. The minimizing inequality therefore forces zero defect. The atomic equality case determines the turning measure, and polar reconstruction identifies the regular polygon.

For the final comparison, recall the scaled defect
\[
 h=\frac{2\pi}{N},\qquad B_N(P)=N^4E_N(P).
\]
With the cyclic ordering,

\begin{equation}\label{eq:spectral-rigidity-001}
 a_j=Nw_j,\qquad
 \psi_{j+1}-\psi_j=h(1+\eta_j),\qquad \sum_j\eta_j=0.
\end{equation}

The term \emph{quasi-uniform} means
$\max_j(|a_j-1|+|\eta_j|)\le1/4$.  An exact-centered polygon has area
$\pi$ and is represented about its polar center; in particular,
$\widehat\nu_P(1)=0$.
The decomposition to be used is

\begin{equation}\label{eq:spectral-rigidity-002}
 \frac{F(P)-F(R_N)}{F(\mathbb D)}
 =Q_N^\circ(P)+D_N^\circ(P)
  +\widetilde{R}_N^\circ(P).
\end{equation}

The first term is controlled by atomic Bessel coercivity. The next theorem combines the two available remainder comparisons so that their bound holds for every positive defect in the localized class.

\begin{proposition}
\label{prop:minimizer-localization}
There are $C_{\rm ent}<\infty$ and $N_{\rm ent}$ such that every global
minimizer in $\mathcal P_{\le N}(\pi)$, $N\ge N_{\rm ent}$, after translation
to its polar center, lies in the exact near-disk chart and satisfies
\begin{equation}\label{eq:spectral-rigidity-003}
 M=N,\qquad \widehat\nu_P(1)=0,\qquad
 \max_j(|Nw_j-1|+|\eta_j|)\le\frac14,
 \qquad B_N(P)\le C_{\rm ent}h\log^3(2N).
\end{equation}
Moreover,
\begin{equation}\label{eq:spectral-rigidity-004}
 E_N(P)\le C_{\rm ent}h^5\log^3(2N).
\end{equation}
\end{proposition}

\begin{proof}
Proposition~\ref{prop:full-activation} gives $N$ effective sides, and Proposition~\ref{prop:global-minimizer-radial-entry} supplies the polar-centered radial chart. Corollary~\ref{cor:first-localization} gives $B_N\to0$ uniformly over minimizers. The microscopic rigidity theorem therefore puts their weights and gaps in the stated quasi-uniform range. The two quantitative bounds are exactly those of Corollary~\ref{cor:refined-localization}.
\end{proof}

\begin{theorem}
\label{thm:relative-remainder}
For every fixed $C_0<\infty$, there is a sequence $\omega_N(C_0)\to0$ such
that, for all sufficiently large $N$, every area-$\pi$ exact-centered
configuration in the quasi-uniform chart satisfying
\[
 E_N(P)\le C_0h^5\log^3(2N),
 \qquad
 B_N(P)\le C_0h\log^3(2N),
 \qquad
 E_N(P)>0,
\]
obeys
\begin{equation}\label{eq:spectral-rigidity-005}
 \left|D_N^\circ(P)
       +\widetilde{R}_N^\circ(P)\right|
 \le \omega_N(C_0)E_N(P).
\end{equation}
No criticality assumption is used here; the estimate holds uniformly on the
whole exact constraint chart.
\end{theorem}

\begin{proof}
The annular comparison and the exact constraint chart cover overlapping ranges. Set $C^\sharp=\max\{C_0,3\}$. In the range
\begin{equation}\label{eq:spectral-rigidity-006}
 e^{-N^{1/3}}\le B_N(P)\le C_0h\log^3(2N),
\end{equation}
the geometric and material estimates give
\begin{equation}\label{eq:spectral-rigidity-007}
 |\widetilde{R}_N^\circ(P)|
 \le\omega^{(R)}_{C^\sharp,N}E_N(P),\qquad
 |D_N^\circ(P)|
 \le\varepsilon_N(C^\sharp)E_N(P),
\end{equation}
with both coefficients tending to zero.

Choose $A_0,K_0$ from Lemma~\ref{lem:vanishing-defect}, increasing $A_0$ if needed for the exact chart, and fix
\begin{equation}\label{eq:spectral-rigidity-008}
 K_*\ge K_0,\qquad K_*/2>\max\{4,A_0\}+1.
\end{equation}
The coverage part of Theorem~\ref{thm:constraint-chart} represents every configuration with $0<B_N(P)<N^{-K_*}$ as $P_N(z)$, $|z|<N^{-A_0}$, modulo rotation and cyclic relabeling. All quantities under comparison are invariant under these operations. Lemma~\ref{lem:vanishing-defect} gives
\begin{equation}\label{eq:spectral-rigidity-009}
 |D_N^\circ(P)+\widetilde{R}_N^\circ(P)|
 \le\omega_N^{(0)}E_N(P).
\end{equation}
For large $N$, $e^{-N^{1/3}}<N^{-K_*}$. The two ranges therefore cover every positive defect allowed in the statement. Taking the maximum of their error coefficients proves the uniform estimate.
\end{proof}

\begin{theorem}
\label{thm:spectral-gap}
Fix $C_{\rm ent}<\infty$.  There are $N_{\rm tr}$ and a sequence
$\omega_N\downarrow0$ such that, for $N\ge N_{\rm tr}$, every area-$\pi$
convex $N$-gon in the near-disk polar chart, centered at its polar center and
satisfying
\begin{equation}\label{eq:spectral-rigidity-011}
 \widehat\nu_P(1)=0,\qquad
 \max_j(|Nw_j-1|+|\eta_j|)\le\frac14,
 \qquad B_N(P)\le C_{\rm ent}h\log^3(2N),
\end{equation}
obeys
\begin{equation}\label{eq:spectral-rigidity-012}
 \frac{F(P)-F(R_N)}{F(\mathbb D)}
 \ge\left(\frac14-\omega_N\right)E_N(P).
\end{equation}
In particular, any polygon in this regime with
$F(P)\le F(R_N)$ must be a rigid image of $R_N$.
\end{theorem}

\begin{proof}
Use $Nh=2\pi$ to convert the assumed bound on $B_N=N^4E_N$ into $E_N\le Ch^5\log^3(2N)$. Theorem~\ref{thm:relative-remainder} applies. Replace its error sequence by its decreasing tail supremum, and take $N_{\rm tr}$ large enough that $\omega_N<1/4$.

Atomic positivity gives $E_N\ge0$. For $E_N>0$, the exact splitting and Bessel coercivity yield
\[
\begin{aligned}
 \frac{F(P)-F(R_N)}{F(\mathbb D)}
 &=Q_N^\circ(P)+D_N^\circ(P)
   +\widetilde{R}_N^\circ(P)\\
 &\ge\left(\frac14-\omega_N\right)E_N(P).
\end{aligned}
\]
The right side is positive. If $E_N=0$, atomic rigidity gives a rotated uniform root measure, and Lemma~\ref{lem:polar-equality-reconstruction} gives its regular polygon. These two cases prove both the inequality and the rigidity statement.
\end{proof}

\begin{proof}[Proof of Theorem~\ref{thm:main}]
Take $N_0$ above the localization and spectral comparison thresholds. For $N\ge N_0$, the relaxed class has a minimizer $P_N^*$. Comparison with the regular polygon gives
\[
 F(P_N^*)\le F(R_N).
\]
Translate the minimizer to its polar center. This preserves the admissible class and the eigenvalue. Proposition~\ref{prop:minimizer-localization} and Theorem~\ref{thm:spectral-gap} then identify it as a rigid image of $R_N$. Therefore every admissible $P$ satisfies
\[
 \lambda_1(P)\ge\lambda_1(P_N^*)=\lambda_1(R_N).
\]
A polygon attaining equality is another minimizer and is identified in the same way. Rigid motions preserve equality, proving the converse.

For a polygon of arbitrary positive area, dilate to area $\pi$. Since $\lambda_1(sP)=s^{-2}\lambda_1(P)$, the product $|P|\lambda_1(P)$ is unchanged. Congruence of the normalized polygons is equivalent to similarity of the original polygons. This gives the scale-invariant formulation and its equality case.
\end{proof}

\appendix

\section{Periodic estimates for the transported mesh}\label{sec:mesh-tools}

The material comparison uses three facts about the periodic mesh: cyclic invariance identifies vanishing first variations, projection onto cellwise constants preserves the relevant energy, and a variation estimate controls the negative half-order norm of cellwise errors. The estimates are stated here independently of the spectral problem. Their constants depend on comparability of cell lengths and remain uniform as the number of cells increases.

\begin{lemma}
\label{lem:regular-cyclic-criticality}
Let \(G_{{\rm inv},N}\) be a \(C^1\) real function of positive weights
\(a=(a_j)\) and cyclically ordered lifted nodes \(\psi=(\psi_j)\) near
\((\mathbf1,(jh)_j)\), where \(h=2\pi/N\). If \(G_{{\rm inv},N}\) is
invariant under cyclic relabeling and common rotation, then
\begin{equation}\label{eq:mesh-tools-001}
DG_{{\rm inv},N}(\mathbf1,jh)[\beta,y]
=c_a\sum_j\beta_j+c_\psi\sum_jy_j,
\end{equation}
for constants \(c_a,c_\psi\) depending on \(G_{{\rm inv},N}\). In
particular the derivative vanishes on every mass- and rotation-tangent
direction \(\sum_j\beta_j=\sum_jy_j=0\).
\end{lemma}

\begin{proof}
A cyclic relabeling followed by rotation through $-h$ fixes the regular configuration and cyclically shifts each tangent array. Invariance forces equal coefficients on all weight coordinates and, separately, on all angle coordinates in the differential. This gives \eqref{eq:mesh-tools-001}. Each sum vanishes on the corresponding zero-mean tangent space.
\end{proof}

\begin{lemma}
\label{lem:periodic-cell-calculus}
Let \(\mathcal I_h\) be a cyclic partition of \(\mathbb T\) whose cell
lengths lie in \([c_0h,C_0h]\), and use normalized Haar measure. The
following estimates have constants depending only on \(c_0,C_0\).

\smallskip
\noindent\textup{\bfseries (i) Cell cancellation.}
If \(f\in L^2\) has integral zero on every cell, then
\begin{equation}\label{eq:mesh-tools-002}
 \|f\|_{\dot H^{-1/2}}\le Ch^{1/2}\|f\|_2.
\end{equation}
If \(f\) has ordinary mean zero and \(I_{\rm per}f\) denotes its
zero-mean primitive, then
\begin{equation}\label{eq:mesh-tools-003}
 \|I_{\rm per}f\|_{\dot H^{1/2}}=\|f\|_{\dot H^{-1/2}}.
\end{equation}
When \(f\) is cellwise centered and bounded, its primitive may be chosen
to vanish at every cell endpoint and
\(\|I_{\rm per}f\|_\infty\le C h\|f\|_\infty\).

\smallskip
\noindent\textup{\bfseries (ii) Regular-step aliases.}
On the regular \(N\)-mesh, \(h=2\pi/N\), let \(q=v_j\) on \(I_j\),
\(\int q=0\), and
\(\widehat v_N(r)=N^{-1}\sum_jv_je^{-irjh}\). Write
\(\operatorname{sinc}s=\sin s/s\), with value one at zero. Then
\begin{equation}\label{eq:mesh-tools-004}
 \widehat q(n)=e^{-inh/2}\operatorname{sinc}(nh/2)\widehat v_N(r),
 \qquad n\equiv r\pmod N.
\end{equation}
Writing \(d(r)=\min(r,N-r)\), \(1\le r<N\),
\begin{equation}\label{eq:mesh-tools-005}
 \sum_{n\equiv r(N)}\frac{|\operatorname{sinc}(nh/2)|^2}{|n|}
 \sim\frac1{d(r)},\qquad
 \sum_{n\equiv r(N)}\frac{|\operatorname{sinc}(nh/2)|}{|n|}
 \le\frac C{d(r)}.
\end{equation}
Consequently
\begin{equation}\label{eq:mesh-tools-006}
 \|I_{\rm per}q\|_\infty
 \le C\sqrt{\log(2N)}\,\|q\|_{\dot H^{-1/2}}.
\end{equation}

\smallskip
\noindent\textup{\bfseries (iii) Products of regular steps.}
If \(u,v\) are regular step functions and \(q_0=\int uv\), then
\begin{equation}\label{eq:mesh-tools-007}
 |q_0|\le\|u\|_2\|v\|_2,
 \qquad
 \|uv-q_0\|_{\dot H^{-1/2}}
 \le C\sqrt{\log(2N)}\,\|u\|_2\|v\|_2.
\end{equation}

\smallskip
\noindent\textup{\bfseries (iv) The centered sawtooth.}
Let \(P_0\) be affine of slope one on every regular cell, vanish in
cell average, and satisfy \(\|P_0\|_\infty\le h/2\). For every regular
step \(a\),
\begin{equation}\label{eq:mesh-tools-008}
\begin{gathered}
 \|aP_0\|_\infty\le Ch\|a\|_\infty,
 \qquad \|aP_0\|_2\le Ch\|a\|_2,\\
 \|aP_0\|_{\dot H^{-1/2}}\le Ch^{3/2}\|a\|_2,
 \qquad \|I_{\rm per}(aP_0)\|_\infty\le Ch^2\|a\|_\infty.
\end{gathered}
\end{equation}
\end{lemma}

\begin{proof}
For the cancellation estimate, subtract the average of a smooth test function on each cell. The variance identity and the bound on cell lengths give
\[
 \sum_{I\in\mathcal I_h}\|\phi-\phi_I\|_{L^2(I)}^2
 \le Ch[\phi]_{\dot H^{1/2}}^2.
\]
Duality yields \eqref{eq:mesh-tools-002}. The Fourier multiplier of the zero-mean primitive is $1/(in)$, proving \eqref{eq:mesh-tools-003}. Integration within a cell gives the endpoint normalization and the supremum bound.

For a regular step function, integration against $e^{-inx}$ gives \eqref{eq:mesh-tools-004}. On the residue class $n=r+mN$, sum the bound $|\operatorname{sinc}(nh/2)|\le C\min\{1,N/|n|\}$ to obtain \eqref{eq:mesh-tools-005}. Cauchy--Schwarz and $\sum_{r=1}^{N-1}d(r)^{-1}\le C\log(2N)$ give the primitive estimate. Every nonzero discrete Fourier coefficient of $uv$ is bounded by $\|u\|_2\|v\|_2$; applying the same alias estimate proves the product bound. Finally, $aP_0$ has zero mean on each cell and satisfies $|aP_0|\le(h/2)|a|$. The cell-cancellation estimate and direct integration give all the sawtooth bounds.
\end{proof}

\begin{lemma}\label{lem:finite-mesh-endpoint}

For $\Pi_0^\perp F=F-\overline F$ and \(F\in BV(\mathbb T)\), one has, for \(0<h\le1/2\),

\[
 \|\Pi_0^\perp F\|_{\dot H^{-1/2}}
 \le C\left[
  \sqrt{\log(2+h^{-1})}\,\|F\|_{L^1}
  +h\,\operatorname {TV}(F)
 \right].
\]

Consequently, let \(\mathcal I_h\) be a quasi-uniform partition of the
circle, with all cell lengths between fixed positive multiples of \(h\).
If \(F_h\) is piecewise \(W^{1,1}\), has mean zero, and satisfies

\[
\sum_{I\in\mathcal I_h}\int_I|F_h'|
 +\sum_{x\in\partial\mathcal I_h}|[F_h](x)|
 \le C_0h^{-1}\|F_h\|_{L^1}.
\]

Then

\[
 \|F_h\|_{\dot H^{-1/2}(\mathbb T)}
 \le C_{C_0}\sqrt{\log(2/h)}\,\|F_h\|_{L^1(\mathbb T)}.
\]

\end{lemma}

\begin{proof}
With normalized Fourier coefficients,
\[
\widehat F(n)=\frac1{2\pi}\int_{\mathbb T}F(x)e^{-inx}\,dx,
\qquad
\|\Pi_0^\perp F\|_{\dot H^{-1/2}}^2
=\sum_{n\ne0}\frac{|\widehat F(n)|^2}{|n|}.
\]
For $n\ne0$, the derivative measure satisfies $in\widehat F(n)=(2\pi)^{-1}\int e^{-inx}\,d(DF)(x)$. This and the direct integral bound give
\[
|\widehat F(n)|
 \le C\|F\|_1,
 \qquad
 |n|\,|\widehat F(n)|\le C\operatorname {TV}(F).
\]
Choose $K=\lfloor h^{-1}\rfloor\sim h^{-1}$, which is at least two. Estimate low frequencies by the first inequality and high frequencies by the second:
\[
\begin{aligned}
 \|\Pi_0^\perp F\|_{\dot H^{-1/2}}^2
 &\le C\|F\|_1^2
   \sum_{1\le |n|\le K}\frac1{|n|}
 +C\operatorname {TV}(F)^2
   \sum_{|n|>K}\frac1{|n|^3}\\
 &\le C\log(2+h^{-1})\|F\|_1^2
 +Ch^2\operatorname {TV}(F)^2.
\end{aligned}
\]
Taking square roots proves the first assertion.

For a piecewise $W^{1,1}$ function, integration by parts on each cell gives
\[
\begin{aligned}
\langle DF_h,\varphi\rangle
&=-\sum_j\int_{x_j}^{x_{j+1}}F_h\varphi'\\
&=\sum_j\int_{x_j}^{x_{j+1}}F_h'\varphi
 +\sum_j\bigl(F_h(x_j+)-F_h(x_j-)\bigr)\varphi(x_j).
\end{aligned}
\]
The absolutely continuous derivative and the atoms at cell endpoints are mutually singular. Thus
\[
\operatorname {TV}(F_h)
=\sum_{I\in\mathcal I_h}\int_I|F_h'|
 +\sum_{x\in\partial\mathcal I_h}|[F_h](x)|.
\]
By hypothesis this is at most $C_0h^{-1}\|F_h\|_1$. Substitute into the first assertion and remove $\Pi_0^\perp$ using the zero-mean condition. We obtain
\[
\|F_h\|_{\dot H^{-1/2}}
\le C_{C_0}\sqrt{\log(2/h)}\,\|F_h\|_1.
\]
The logarithms and the additional $L^1$ term are absorbed by
\[
2+h^{-1}\le 2h^{-1},
\qquad \log(2+h^{-1})\le\log(2/h),
\qquad 1\le\sqrt{\log(2/h)},
\]
which hold for $0<h\le1/2$, since $2/h\ge4$ and $\log4>1$.
\end{proof}

\section*{Acknowledgements}
The first and third authors gratefully acknowledge the hospitality of the University of Macau during the workshop ``Mini Symposium on AI for Mathematics,'' held there from June 9 to 12, 2026. The proof scheme underlying the present paper was proposed by Jaume de Dios Pont during the workshop. The authors thank him for stimulating discussions.

\section*{AI statement}
Artificial-intelligence-assisted tools were used in the preparation of this manuscript for language polishing, organization, exploratory checking, and assistance with routine derivations. All central mathematical ideas, the formulation of the main results, and all major proof decisions were human-generated; the authors retained control of the proof strategy and mathematical judgment throughout. As an additional verification layer, the complete proof was checked using the \texttt{qmd-prover} skill, developed by Xiao Ma; for details, see the GitHub project at \url{https://github.com/powergiant/qmd-prover}. Independently, the human authors also verified the correctness of the arguments and conclusions. The authors take full responsibility for the content and mathematical correctness of the manuscript.

\raggedbottom

\end{document}